\documentclass[11pt,english,a4paper]{smfart}
\usepackage[english]{babel}
\usepackage{amssymb,xspace}
\usepackage{amstext}
\usepackage{smfthm}
\theoremstyle{plain}
\usepackage{amsbsy,amssymb,amsfonts,latexsym}

\usepackage{enumerate}
\usepackage{graphics}

\date{\today}
\title[Operators without non-trivial invariant closed subspaces]
{A general approach to Read's type constructions of operators without non-trivial invariant closed subspaces}

\author{Sophie Grivaux}
\address{CNRS,
Laboratoire Paul Painlev\' e, UMR 8524, Universit\'e Lille 1, Cit\' e Scientifique, 59655 Villeneuve d'Ascq
Cedex, France}
\email{grivaux@math.univ-lille1.fr}

\author{Maria Roginskaya}
\address{Department of Mathematical Sciences,
Chalmers University of Technology, SE-41296 G\"oteborg, Sweden, \emph{and}
Department of Mathematical Sciences,
G\"oteborg University, SE-41296 G\"oteborg, Sweden}
\email{maria@chalmers.se}



\newcommand{\ls}{\lesssim}

\def\K{\ensuremath{\mathbb K}}

\def\R{\ensuremath{\mathbb R}}

\def\C{\ensuremath{\mathbb C}}

\def\N{\ensuremath{\mathbb N}}

\newcommand{\sep}{separable}
\newcommand{\hy}{hypercyclic}

\newcommand{\ops}{operators}
\newcommand{\op}{operator}
\newcommand{\cy}{cyclic}

\newcommand{\wrt}{with respect to}
\newcommand{\nt}{non-trivial}
\newcommand{\nz}{non-zero}
\newcommand{\re}{reflexive}
\newcommand{\nr}{non-reflexive}
\newcommand{\inv}{invariant}
\newcommand{\pol}{polynomial}

\newcommand{\ifff}{if and only if}

\newcommand{\pss}[2]{\ensuremath{{\langle #1,#2\rangle}}}

\newcommand{\loi}{lay-off interval}
\newcommand{\woi}{working interval}
\newcommand{\To}{\longrightarrow}
\newcommand{\qr}{quasi-reflexive}
\newcommand{\ww}{$w^{*}$-closed}

\newtheorem{theorem}{Theorem}[section]

\newtheorem{lemma}[theorem]{Lemma}

\newtheorem{claim}[theorem]{Claim}

\newtheorem{proposition}[theorem]{Proposition}

\newtheorem{corollary}[theorem]{Corollary}

{\theoremstyle{definition}}

{\theoremstyle{definition}}

\newtheorem{fact}[theorem]{Fact}

{\theoremstyle{definition}}

{\theoremstyle{definition}}

{\theoremstyle{definition}\newtheorem{remark}[theorem]{Remark}}

\begin{document}

\begin{abstract}
We present a general method for constructing operators without non-trivial invariant closed subsets on a large class of non-reflexive Banach spaces. In particular, our approach  unifies and generalizes several constructions due to Read of operators without non-trivial invariant subspaces on the spaces $\ell_{1}$, $c_{0}$ or $\oplus_{\ell_{2}}J$, and without non-trivial invariant subsets on $\ell_{1}$. We also investigate how far our methods can be extended to the Hilbertian setting, and construct an operator on a quasireflexive dual Banach space which has no non-trivial $w^{*}$-closed invariant subspace.
\end{abstract}

\maketitle

\section{Introduction}

Let $X$ be a real or complex separable Banach space of infinite
dimension, and $T\in \mathcal{B}(X)$ a bounded linear \op\ on $X$.
 An important question in \op\ theory is to determine whether or not
 $T$ always admits a non trivial invariant closed subspace, i.e. if there
 exists a closed subspace $M$ of $X$ with $M\not =\{0\}$ and $M\not = X$
 such that $T(M)\subseteq M$. This problem is called the \emph{Invariant
 Subspace Problem}, and it has been answered in the negative in the $80$'s
 by Enflo \cite{E} and Read \cite{R}: they constructed some separable spaces
 $X$ along with some \ops\ $T$ on $X$ having no non-trivial invariant closed
  subspaces. Then Read \cite{R1}, \cite{R2} gave examples of such \ops\ on some
  classical Banach spaces such as $\ell_{1}$ or $c_{0}$.
\par\smallskip
All these \ops\ live on non-reflexive Banach
spaces. The Invariant Subspace Problem remains open for
reflexive Banach spaces, and in particular for the Hilbert spaces. The best counterexample
in this direction is due to Read \cite{R2}: he constructed an \op\ without \nt\ \inv\
  closed subspace on
 $X=\oplus_{\ell_{2}}J$, the $\ell_{2}$-sum of countably many copies of the
 James space $J$. Since $J$ has codimension $1$ in its bidual, the space
 $X^{**}/X$ is separable.
\par\smallskip
The \emph{Invariant Subset Problem},  which is to know whether a bounded
\op\  $T$ on $X$ always admits a closed subset $F$ with $F\not =\{0\}$,
$F\not  = X$ such that $T(F)\subseteq F$, is even more widely open than
the Invariant Subspace Problem. The only known counterexamples to this problem are due to Read, who constructed in \cite{R3} \ops\ on
$\ell_{1}$ (and more generally on any space containing a complemented copy of $\ell_1$) with no \nt\ \inv\ closed subset.
\par\smallskip
The Invariant Subspace and Subset Problems can be reformulated in terms of \cy\  and \hy\ vectors: recall that if $x$ is any vector of $X$, the
\emph{orbit of $x$ under the action of $T$} is the set
 $\mathcal{O}\textrm{rb}(x,T)=\{T^{n}x \textrm{ ; } n\geq 0\}$. Its closure
$\overline{\mathcal{O}\textrm{rb}}(x,T)$ is the smallest closed
subset of $X$ which is invariant by $T$ and contains $x$, while
the closure of the linear span of the orbit of $x$ is the smallest
closed subspace of $X$ which is \inv\ by $T$ and contains $x$.
The vector $x$ is said to be \emph{cyclic} for $T$ is the linear span
of the orbit of $x$ is dense in $X$, and \emph{hypercyclic} if the orbit
 itself is dense in $X$. Hence $T$ has no \nt\ \inv\ closed subspace \ifff\
  every non-zero vector is cyclic for $T$, and no \nt\ \inv\ closed subset \ifff\
   every non-zero vector is \hy\ for $T$.
\par\smallskip
We will be concerned in this paper with Read's type \ops: by this we mean \ops\ of the kind constructed by Read in his various works \cite{R}, \cite{R1}, \cite{R2} and \cite{R3} on invariant subspaces and subsets issues. They are of a very combinatorial nature, and although the constructions of \cite{R}, \cite{R1}, \cite{R2} and \cite{R3} rely on a basis of common techniques, each of them has to be adapted to the particular space one is working with. For instance, the constructions of \cite{R1} and \cite{R3} rely heavily on the additive properties of the $\ell_{1}$ norm, while those of \cite{R2} depend on some particular features of the canonical basis of the spaces involved, namely $c_{0}$ or $\oplus_{\ell_{2}}J$.
\par\smallskip
Our first aim in this paper is to present a unified approach to all these constructions, and to show how all the known counterexamples to the Invariant Subspace or Subset Problem on ``concrete'' Banach spaces can be derived from a single general statement. We are especially interested in determining which geometric properties of a Banach space will ensure that it supports an \op\ without non-trivial \inv\ closed subset. This is of interest in the view of the recent works of Argyros and Haydon \cite{AH} and Argyros and Motakis \cite{AM}: in \cite{AH}, examples are constructed of spaces on which every \op\ is of the form $\lambda I+K$, with $\lambda $ a scalar and $K$ a compact \op, and it is well-known that such operators always have a non-trivial closed \inv\ subspace \cite{Lom}. In \cite{AM}, the authors exhibit reflexive Banach spaces on which every \op\  has a non-trivial \inv\ closed subspace. The spaces of of \cite{AH} and \cite{AM} are hereditarily indecomposable, so they are certainly very far from the kind of spaces which we are going to consider in our forthcoming Theorem \ref{th1}.
\par\smallskip
Here is our first main result, which yields a large class of non-reflexive Banach spaces on which \ops\ without non-trivial \inv\ closed subsets can be constructed. Contrary to the spaces of \cite{AH} and \cite{AM}, which support no \op\ with no \nt\ \inv\ closed subspace, 
the spaces to which Theorem \ref{th1} applies are decomposable in a very strong sense.

\begin{theorem}\label{th1}
 Let $Z$ be a non-reflexive Banach space admitting a Schauder basis. Let $X$ be one of the spaces $\oplus_{\ell_{p}}Z$, $1\le p<+\infty$, or $\oplus_{c_{0}}Z$, where we denote by  $\oplus_{\ell_{p}}Z$ or $\oplus_{c_{0}}Z$ the direct $\ell_{p}$- or $c_{0}$-sum of infinitely many copies of the space $Z$. 
 
 Then
   there exists a bounded \op\ $T$ on $X$ which has no non-trivial \inv\ closed subset. The same conclusion holds true for any \sep\ space which contains a complemented copy of one of the spaces $X$ above.
\end{theorem}

We retrieve in particular the existence of an \op\ on $\ell_{1} $ with no non-trivial \inv\ closed subset \cite{R3}, and improve the results of \cite{R2} by showing that $c_{0}$ and $\oplus_{\ell_{2}}J$ support \ops\ without non-trivial \inv\ closed subsets.
\par\smallskip
An interesting consequence of Theorem \ref{th1} is:

\begin{corollary}\label{cor2}
Let $X$ be an infinite-dimensional non-\re\ \sep\ space having an unconditional basis. Then there exists a bounded \op\ on $X$ with no \nt\ \inv\ closed subset.
\end{corollary}

Indeed by a classical result of James \cite{J}, such a space $X$ contains a complemented subspace which is isomorphic to either $c_0$ or $\ell_1$. It then suffices to apply Theorem \ref{th1}. If we are only interested in operators without \nt\ \inv\ closed subspaces, the analogue of Corollary \ref{cor2} follows already from this argument and the works \cite{R1} and \cite{R2} of Read.
\par\smallskip

Theorem \ref{th1} and Corollary \ref{cor2} impose as an assumption that the space $X$ one is working with be non-reflexive. As mentioned previously, all known counterexamples to the Invariant Subspace Problem live on non-reflexive Banach spaces, and the role played by reflexivity in these questions is not really clear. Our second goal in this work is to shed some light on this question, and to explain at which point of a Read's type construction the non-reflexivity assumption becomes really crucial. Two-thirds of such a construction can be adapted to any kind of ``reasonable'' space (for instance, to all the $\ell_p$ spaces), in particular to the Hilbert space. This was shown in the paper \cite{GR}, where \ops\ on the Hilbert space with few \inv\ closed subsets were constructed. We quote here the main result of \cite{GR} (combined with a remark from \cite{GR2}), which shows which kind of properties of Read's type \ops\ can be enforced in the Hilbertian setting.

\begin{theorem}\label{th0}\cite{GR}
 There exists a bounded \op\ on a \sep\ (real or complex) Hilbert space $H$ which satisfies the following two properties:

$\quad (P1)$ for every $x\in H$ the closure of the sets
 $\mathcal{O}\textrm{rb}(x,T)$ and $\textrm{sp}[\mathcal{O}\textrm{rb}(x,T)]$
 coincide;

$\quad (P2)$ the family
 $(\overline{\mathcal{O}\textrm{rb}}(x,T))_{x\in H}$ of the
 closures of the  orbits of $T$ is totally ordered, i.e. for any pair $(x,y)$ of
 vectors of $H$, either $\overline{\mathcal{O}\textrm{rb}}(x,T)\subseteq
\overline{\mathcal{O}\textrm{rb}}(y,T)$ or
$\overline{\mathcal{O}\textrm{rb}}(y,T)
\subseteq \overline{\mathcal{O}\textrm{rb}}(x,T)$.
\par\smallskip
\noindent It follows from this that the set of non-\hy\ vectors for $T$ is very small in the sense that it is contained in a countable union of closed subspaces of $H$ which are of infinite codimension in $H$.
\end{theorem}

Property (P1) can be reformulated as ``the closure of any orbit is a subspace'', while \ops\ satisfying Property (P2) are usually called \emph{orbit-unicellular}. The combination of Properties (P1) and (P2) is already a strong requirement on the structure of the lattice of invariant closed subspaces of $T$, and this is what forces the set of non-\hy\ vectors to be very small (see \cite[Section 5.2]{GR2} for more details). Still, properties (P1) and (P2) are no sufficient in order to guarantee that $T$ has no non-zero non-\hy\ vector, and it is proved in \cite{GR2} that the Read's type \ops\  on the Hilbert space constructed in \cite{GR} do have non-trivial \inv\ closed subsets (which are automatically subspaces by (P1)). In order to go one step further and gain more information on vectors which are close to the orbit of any non-zero vector, non-reflexivity of the space becomes essential. We prove the following theorem, which can be seen as a kind of ``first step'' towards the construction of \ops\ without non-trivial \inv\ closed subsets:

\begin{theorem}\label{th2}
 Let $X$ be a separable \nr\ Banach space having a Schauder basis and
 containing a complemented copy of one of the spaces $\ell_{p}$,
 $1\leq p<+\infty $, or $c_{0}$. Let $(g_{j})_{j\geq 0}$ denote the
 canonical basis of this space $\ell_{p}$ or $c_{0}$, where we suppose
 without loss of generality that $||g_{0}||=1$.
For any $\varepsilon \in (0,1)$ there exists a bounded \op\ $T$ on $X$ for which $g_{0}$ is hypercyclic and
which has the following three properties:

$\quad (P1)$ for every $x\in X$ the closure of the sets
 $\mathcal{O}\textrm{rb}(x,T)$ and $\textrm{sp}[\mathcal{O}\textrm{rb}(x,T)]$
 coincide;

$\quad (P2)$ the family
 $(\overline{\mathcal{O}\textrm{rb}}(x,T))_{x\in X}$ of the
 closures of its  orbits is totally ordered;

$\quad (P3)$ for any
 \nz\ vector $x\in X$ the distance of the orbit $\mathcal{O}\textrm{rb}(x,T)$ to
 $g_{0}$ is less than $\varepsilon $: $d(\mathcal{O}\textrm{rb}(x,T), g_0)
 <\varepsilon $.
\end{theorem}

We have observed that an
\op\ $T\in \mathcal{B}(X)$ has no \nt\ \inv\ closed subset \ifff\ every \nz\
 vector is \hy\ for $T$. A natural way to check this is the following: pick a
  given \hy\ vector of norm $1$, in our context the vector $g_{0}$. Then prove
  that whenever $\varepsilon \in (0,1)$ and $x\in X$ a \nz\ vector, there exists a
  nonnegative integer $n$  such that $||T^{n}x-g_{0}||<\varepsilon $, i.e.
  that the distance of  $\mathcal{O}\textrm{rb}(x,T)$ to $g_{0}$ is less
  than $\varepsilon $. Property (P3)
gives this for
a fixed $\varepsilon \in (0,1)$. In order to show that any non-zero vector $x$ is \hy\ for $T$, we would need Property (P3) to hold true  for a sequence $(\varepsilon_{k})_{k\ge 1} $ of real numbers in $(0,1)$ going to zero as $k$ goes to infinity. This is where the assumption of Theorem \ref{th1} comes into the picture: infinitely many copies of the non-reflexive space $Z$ are needed in order to $\varepsilon$-approximate the vector $g_{0}$ by an element of the orbit of any non-zero vector for arbitrary $\varepsilon$. But if we require this approximation for one \emph{single} $\varepsilon$, one copy of this space $Z$ is enough.
\par\smallskip
Property (P3) can be very close to forcing some \op\ on the space to have no non-trivial \inv\ closed subspace at all. When $X$ is a Hilbert space, this is indeed the case. This is the content of our next theorem, which highlights again the role played by the geometry of the space in such considerations:

\begin{theorem}\label{th3}
Let $g_{0}$ be a norm-one vector of the complex Hilbert space $\ell_{2}$. Let $\varepsilon \in (0,1)$ be a fixed real number. Suppose that 
 there exists a bounded \op\ $T$ on $\ell_{2}$ which has no eigenvector and
 satisfies the following property: 
 
 $\quad (P3')$ for any
 \nz\ vector $x\in X$, the distance of the closed invariant subspace
$M_{x}=\overline{\textrm{sp}}\,[T^{n}x \textrm{ ; } n\ge 0]$ generated by $x$ to the vector $g_{0}$ is less than $\varepsilon $: $d(M_{x}, g_0)
 <\varepsilon $.
 \par\smallskip
 Then there exists a bounded \op\ $T'$ on $\ell_{2}$ which has no non-trivial \inv\ closed subspace.
\end{theorem}

As a straightforward corollary one obtains:

\begin{corollary}\label{cor3bis}
 Let $g_{0}$ be a norm-one vector of the complex Hilbert space $\ell_{2}$. Let $\varepsilon \in (0,1)$ be a fixed real number. Suppose that 
 there exists a bounded \op\ $T$ on $\ell_{2}$ which satisfies Property (P3) of Theorem \ref{th2} above.
 
  Then there exists a bounded \op\ $T'$ on $\ell_{2}$ which has no non-trivial \inv\ closed subspace.
\end{corollary}

It is funny to note that one can come surprisingly close to constructing an \op\ on a Hilbert space which satisfies the assumption of Corollary \ref{cor3bis}:

\begin{theorem}\label{th3bis}
Let $g_{0}$ be a norm-one vector of the complex Hilbert space $\ell_{2}$. Let $\varepsilon \in (0,1)$ be a fixed real number. There exists a bounded \op\ $T$ on $\ell_{2}$ and a non-zero vector $x_{0}\in\ell_{2}$ such that
$g_{0}$ is a hypercyclic vector for $T$, $||x_{0}-g_{0}||<\varepsilon$, and
\par\smallskip
$\quad (P3'')$ for any vector
  $x\in \ell_{2}$ not colinear to $x_{0}$, the distance of the orbit $\mathcal{O}\textrm{rb}(x,T)$ to
 $g_{0}$ is less than $\varepsilon $.
\par\smallskip
\noindent In particular, the distance of $g_{0}$ to any closed non-trivial \inv\ subspace of $T$ is less than $\varepsilon$.
\end{theorem}

The \op\ $T$ constructed in the proof of Theorem \ref{th3bis} is not injective: $Tx_{0}=0$, and so it definitely has non-trivial \inv\ subspaces. The proof of Theorem \ref{th3bis} shows that any non-trivial invariant closed subspace of $T$ necessarily contains the vector $x_{0}$.
\par\smallskip

There is an obvious class of non-reflexive Banach spaces to which Theorem \ref{th1} does not apply, although Theorem \ref{th2} might: these are the quasi-reflexive spaces, i.e. the non-reflexive spaces $Z$ such that the quotient $Z^{**}/Z$ is finite-dimensional. For instance the James space $J$, which contains a complemented copy of $\ell_{2}$, supports by Theorem \ref{th2} \ops\ which satisfy (P1), (P2) and (P3), but it is not known whether it supports an \op\ with no non-trivial \inv\ closed subspace. There seems to be a good reason for this:

\begin{theorem}\label{th4}
 Suppose that there exists a \sep\ quasi-reflexive space $Z$ of order $1$ which supports an \op\ with no non-trivial \inv\ closed subspace. Then there exists a \sep\ reflexive space $X$ which supports an \op\ with no non-trivial \inv\ closed subspace.
\end{theorem}

This is not hard to prove, and is implied by the following more general statement (see Proposition \ref{propnewD}): suppose that we are given a \sep\ Banach space $Z$ and a bounded \op\ $T$ on $Z$ which has no non-trivial \inv\ closed subspace (resp. subset), and that $T$ is weakly compact. Then there exists a reflexive \sep\ Banach space $X$ and a bounded \op\ on $X$ which has no non-trivial \inv\ closed subspace (resp. subset).
 Now every \op\ on a quasi-reflexive space 
$Z$ of order $1$ can be written as $\lambda I+V$, where $\lambda $ is a scalar and $V$ is a weakly compact \op\ on $Z$ (see \cite{FL}). Theorem \ref{th4} follows from this.
\par\smallskip
The closest one can get for the time being to constructing an \op\ without non-trivial \inv\ subspaces on a quasi-reflexive space is the following:

\begin{theorem}\label{th5}
 There exists a \sep\ quasi-reflexive space $Z$ of order $1$ whose dual $Z^{*}$ (which is also a \sep\ quasi-reflexive space of order $1$) supports an \op\ $T$ with no non-trivial \inv\ $w^{*}$-closed subspace. 
\end{theorem}

As this \op\ is again obtained via a Read's type construction, it is not an adjoint. It was first proved by Troitsky and Schlumprecht in \cite{TS} that the \ops\ constructed by Read in \cite{R1} on $\ell_{1}$ are not adjoints of an \op\ on any predual of $\ell_1$, and it is a general fact that no Read's type \op\ on a dual space $X$ is ever an adjoint of some \op\ on a predual of $X$ (see Section \ref{sec7} for details).
\par\smallskip
The paper is organized as follows: we recall in Section \ref{sec2} the general principle of Read's type constructions, define the \ops\ which will be used in the proof of Theorem \ref{th2} and present some of their general properties. Theorem \ref{th2} itself is proved in Section \ref{sec3}. We then show how the \ops\ introduced previously have to be modified if we want them to satisfy the conclusion of Theorem \ref{th1}, and we prove Theorem \ref{th1} in Section \ref{sec4}. We investigate in Section \ref{sec5} how far these constructions may be extended to the Hilbert space setting, and prove there Theorems \ref{th3} and \ref{th3bis}. We consider the case of quasi-reflexive spaces in Section \ref{sec6}, where Theorems \ref{th4} and \ref{th5} are proved. 
Section \ref{sec7} contains miscellaneous remarks and comments.

\section{Read's type operators}\label{sec2}

\subsection{An informal introduction}
We present first in an informal way the general philosophy of Read's type constructions: they can be carried out on spaces of the form $X=\ell_{p}\oplus Z$, $1\le p<+\infty$, or $c_{0}\oplus Z$, or more generally on any space of the form $Z_{0}\oplus Z$, where $Z_{0}$ is a space with a symmetric basis on which weighted shifts  with bounded weights can safely be defined. The space $Z$ is a \sep\ space, which is usually assumed to have a Schauder basis. It is always possible  to incorporate in the construction an arbitrary \sep\ Banach space $Y$: the proofs can be extended without any trouble to spaces of the form $\ell_{p}\oplus Z\oplus Y$ or $c_{0}\oplus Z\oplus Y$. This does not present any new difficulty in this setting compared to the works \cite{R2} or \cite{R3} of Read, so we refer the reader to one of these papers where the argument is given in detail, and restrict ourselves from now on to the case of spaces of the form
$\ell_{p}\oplus Z$ or $c_{0}\oplus Z$. When writing $\ell_{p}\oplus Z$, we mean that the direct sum between $\ell_{p}$ and $Z$ is an $\ell_{p}$-sum, and when writing $c_{0}\oplus Z$ that it is an $\ell_{\infty}$-sum.

\par\smallskip

Denoting by $(g_{j})_{j\ge 0}$ the canonical basis of $\ell_{p}$ or $c_{0}$, and by $(z_{j})_{j\ge 0}$ a Schauder basis of $Z$, we construct two sequences $(f_{j})_{j\ge 0}$ and $(e_{j})_{j\ge 0}$ of vectors of $X$ in the following way: $(f_{j})_{j\ge 0}$ is a basis of $X$, obtained by choosing a certain amount of vectors from the basis $(g_{j})_{j\ge 0}$, let us say vectors $g_{0}, g_{1},\ldots, g_{j_{1}}$, then another amount of vectors $z_{0}, \ldots, z_{j_{2}}$ from the basis $(z_{j})_{j\ge 0}$, then again vectors
$g_{j_{1}+1},\ldots, g_{j_{3}}$
from $(g_{j})_{j\ge 0}$, etc. Since we choose these vectors without changing the order in the two sequences $(g_{j})_{j\ge 0}$ and $(z_{j})_{j\ge 0}$, and without skipping any of them, we get a basis $(f_{j})_{j\ge 0}$ of $X$ with $f_{0}=g_{0}$. The vectors $e_{j}$ are finitely supported \wrt\ the basis $(f_{j})_{j\ge 0}$, $e_{0}=f_{0}=g_{0}$, and for every $j\ge 1$ they are such that 
$\textrm{sp}[e_{0},\ldots, e_{j}]=\textrm{sp}[f_{0},\ldots, f_{j}]$. In other words, $e_{j}$ belongs to the linear span of the vectors $f_{0}, \ldots, f_{j}$, and its $j^{th}$ coordinate in the basis $(f_{j})_{j\ge 0}$ is non-zero. The \op\ $T$ is defined on the set of finitely supported vectors \wrt\ the basis $(f_{j})_{j\ge 0}$ by the formula
$$Te_{j}=e_{j+1}\quad \textrm{for every } j\ge 0.$$
If the vectors $e_{j}$ are suitably constructed, it will be possible to extend $T$ into a bounded \op\ on $X$, which hopefully will have few \inv\ closed subsets.
\par\smallskip
The construction of the vectors $f_{j}$ and $e_{j}$, $j\ge 0$, is done by induction. At step $n$, $f_{j}$ and $e_{j}$ are defined for $j$ belonging to a certain interval $[\xi _{n}+1, \xi _{n+1}]$, where $(\xi  _{n})_{n\ge 0}$ is a sequence of integers with $\xi  _{0}=0$ which will be chosen to grow very fast. There are different types of definitions for the vectors $e_{j}$ for $j\in [\xi  _{n}+1,\xi  _{n+1}]$, depending on whether $j$ lies in some sub-intervals of 
$j\in [\xi  _{n}+1,\xi  _{n+1}]$ which are called \emph{working intervals}, or in their complement which consists of \emph{\loi s}. Working intervals are those on which something happens. They are separated by very long  \loi s, on which $T$ acts simply as a (forward) weighted shift, and their role is to allow us to sew harmoniously the \woi s together, and to prevent side effects from one \woi\ on another. When $j$ belongs to a \loi, $e_{j}$ is defined by the relation $f_{j}=\lambda _{j}e_{j}$, where $\lambda _{j}$ is a positive coefficient which is very large if $j$ lies in the beginning of the \loi\ and very small if it lies towards the end, with $\lambda _{j}/\lambda _{j+1}$ extremely close to $1$. When both $j$ and $j+1$ belong to a \loi, $Tf_{j}=(\lambda _{j}/\lambda _{j+1})f_{j+1}$, i.e. $T$ acts on these vectors as a weighted shift with weights very close to $1$.
\par\smallskip
Working intervals are of three types, which we call (a), (b) and (c).
    The (c)-part is the part which makes Property (P1) hold true; then adding
    the (b)-part yields Property (P2). The (a)-part is the part in \cite{R}, \cite{R1}, \cite{R2}
    and \cite{R3} which provides the vectors belonging to the closures of the
    orbits of \nz\ vectors.
The (c)- and (b)-parts of a Read's type construction use only the copy of the space $\ell_{p}$ or $c_{0}$ which is present in the decomposition of $X$ as $X=\ell_p\oplus Z$ or $X=c_{0}\oplus Z$. The space $Z$ is simply incorporated in the construction, but could as well be dispensed with: non-reflexivity plays no role here, and this is why \ops\ enjoying properties (P1) and (P2) could be constructed in the Hilbertian setting in \cite{GR}. But non-reflexivity becomes crucial  in the construction of the (a)-part.
\par\smallskip
The way non-reflexivity appears
here is via a result of Zippin \cite{Z} which characterizes \re\ spaces within
the class of spaces having a Schauder basis. It states that a space $Z$ having
a Schauder basis is \re\ \ifff\ all the Schauder bases of $Z$ are boundedly
complete. Recall that if $(z_{j})_{j\geq 1}$ is a Schauder basis of $Z$,
$(z_{j})_{j\geq 1}$ is said to be \emph{boundedly complete}
 if whenever $(\alpha _{j})_{j\geq 1}$ is a sequence of scalars such that
$\sup_{J\geq 1}||\sum_{j=1}^{J}\alpha _{j}z_{j}||$ is finite,
the series $\sum \alpha _{j}z_{j}$ is automatically convergent
in $Z$. The prototype of a non-boundedly complete basis in a \nr\
Banach space is the canonical basis $(h_{j})_{j\geq 0}$ of $c_{0}$:
$\sup_{J\geq 1}||\sum_{j=1}^{J} h_{j}||=1$ but $\sum h_{j}$ does not
 converge in $c_{0}$. The result of \cite{Z} is that such a basis
 $(z_{j})_{j\geq 1}$ exists in any \nr\ space with a Schauder basis. We will actually need
 a more precise version of the result of Zippin, which follows from
 the proof of \cite{Z}, and is also proved in a more general context
  in Kalton's paper \cite{K1}: if $Z$ is a \nr\ Banach space with a
  Schauder basis, it admits a Schauder basis $(z_{j})_{j\geq 1}$ with $\inf_{j\ge 1}||z_{j}||>0$
  which is not \emph{semi-boundedly complete}: there exists a sequence of positive
  scalars $(\alpha _{j})_{j\geq 1}$ and a strictly increasing sequence
   $(\kappa_{j})_{j\geq 1}$ of integers such that  $\sup_{J\geq 1}
   ||\sum_{j=1}^{J}\alpha _{j}z_{\kappa_{j}}||$ is finite but the
   sequence $(\alpha _{j})_{j\geq 1}$ is bounded away from zero.
   In particular the series $\sum \alpha_{j}z_{\kappa_{j}}$ obviously cannot converge in $Z$.
Very roughly speaking, this basis $(z_{j})_{j\geq 1}$ of $Z$ is
used in the
construction of the (a)-part in the following way: for instance in the proof of Theorem \ref{th2}, we choose
$(\alpha _{j})$ and ${\kappa_{j}}$ as above, with
$\sup_{J\geq 1}||\sum_{j=1}^{J}\alpha _{j}z_{\kappa_{j}}||<\varepsilon $.
 The construction is done in such a way that for any \nz\ vector $x\in X$
 the distance of $g_{0}+\sum_{j=1}^{J}\alpha _{j}z_{\kappa_{j}}$ to
 $\mathcal{O}\textrm{rb}(x,T)$
can be made arbitrarily small for infinitely many integers $J$. Hence the
distance of $g_{0}$ to
$\mathcal{O}\textrm{rb}(x,T)$ is less than $\varepsilon $.
The construction fails if the series $\sum \alpha _{j}z_{\kappa_{j}}$
is convergent to a certain point $z\in Z$,
 because the construction of the \op\ $T$ forces the norms of the vectors
 $T(\sum_{j=1}^{J}\alpha _{j}z_{\kappa_{j}}+g_{0})$ to go to zero as $J$
  goes to infinity. Then necessarily $T(z+g_{0})=0$. Since the orbit of  every \nz\ vector under the action of $T$ comes within $\varepsilon $-distance of $g_{0}$ infinitely many times, this forces $z$ to be equal to $-g_{0}$, which is impossible.
\par\smallskip

We now suppose that the assumptions of Theorem \ref{th2} are in force.

\subsection{Definition of the operator $T$}

We denote by $(g_{j})_{j\geq 0}$ the canonical basis
 of either $\ell_{p}$ or $c_{0}$, and by $(z_{j})_{j\geq 1}$
  a non semi-boundedly complete basis of $Z$ as above. Then
there exists a sequence of positive
  scalars $(\alpha _{j})_{j\geq 1}$ which is bounded away from zero
  (there exists a $\delta _{0}>0$ such that $\alpha _{j}\geq \delta_{0} $
  for every $j\geq 1$) and a strictly increasing sequence
   $(\kappa_{j})_{j\geq 1}$ of integers such that  $\sup_{J\geq 1}
   ||\sum_{j=1}^{J}\alpha _{j}z_{\kappa_{j}}||<\varepsilon $.
   The way of defining the \op\ is similar in many places to the
   one employed in \cite{R2}, \cite{R3} or \cite{GR},
 so   we will sometimes refer
   the reader to these works for some details.
\par\smallskip
Let $\K[\zeta]$ denote the space of polynomials with
 coefficients in $\K=\R$ or $\C$, and let $\K_{d}[\zeta]$
 denote the space of polynomials of degree at most $d$.
  If $p(\zeta )=\sum_{k=0}^{d}a_{k}\zeta ^{k}$ is a polynomial, the modulus
   $|p|$ of $p$ is $|p|=\sum_{k=0}^{d}|a_{k}|$. We construct
   by induction on $n$:
\par\smallskip
$\bullet$ a sequence $(\xi_{n})_{n\geq 0}$ of integers
 increasing very quickly;
 
$\bullet$ two sequences $(f_{j})_{j\geq 0}$ and
 $(e_{j})_{j\geq 0}$
 of vectors of $X$, where at step $n$ the vectors $f_{j}$ and $e_{j}$ are constructed for $j\in [\xi  _{n}+1,\xi  _{n+1}]$;

$\bullet$ two sequences $(a_{n})_{n\geq 0}$ and
$(b_{n})_{n\geq 1}$ of integers increasing very quickly;

$\bullet$ a sequence $(\varepsilon _{n})_{n\geq 1}$
of positive numbers going very quickly to zero and a
sequence $(l_{n})_{n\geq 1}$ of integers going to infinity in a suitable fashion;

$\bullet$ for each $n\ge 1$, a sequence $(p_{k,n})_{1\leq k\leq k_{n}}$ of polynomials
which form an
$\varepsilon _{n}$-net of the closed ball of $\K_{l_{n}}[\zeta ]$ of radius
 $2$ for the norm $|\,.\,|$,  and a sequence $(c_{k,n})_{1\leq k\leq k_{n}}$
 of integers also increasing very quickly.
\par\smallskip
These sequences are constructed in such a way that for each $n\ge 1$ we have
 $$\xi_{n}\ll a_{n}\ll b_{n}\ll c_{1,n}\ll c_{2,n}\ll
\ldots\ll c_{k_{n},n}\ll \xi_{n+1},$$ where the symbol ``$\ll$''
means that the quantity on the right is very much larger than
the quantity on the left.
\par\smallskip
We set $f_{0}=e_{0}=g_{0}$ (the first basis vector
 of $\ell_{p}$ or $c_{0}$), and $\xi_{0}=a_{0}=0$.
 The vectors $f_{j}$ and $e_{j}$ are then defined
 inductively: at step $n$ we define $f_{j}$ and
 $e_{j}$ for $j\in [\xi_{n}+1, \xi_{n+1}]$.
As for the sequences $(a_{n})_{n\geq 1}$, $(b_{n})_{n\geq 1}$,
$(c_{k,n})_{1\leq k\leq k_{n}}$,
$(\varepsilon _{n})_{n\geq 1}$ and $(p_{k,n})_{1\leq k\leq k_{n}}$,
we define at step $n$ first $a_{n}$, then $b_{n}$ and then
$\varepsilon _{n}$, $(p_{k,n})_{1\leq k\leq k_{n}}$ and
lastly
$(c_{k,n})_{1\leq k\leq k_{n}}$.
The quantity $a_{n}$ for instance depends only on the
blocks which have been constructed until step $n-1$.
Once $a_{n}$ is firmly defined, we construct $b_{n}$
depending only on the construction until step $a_{n}$, etc... We will
often write in the rest of the paper that some quantity depends only on
$a_{n}$, for instance: by this sentence we will mean that the quantity
depends only on the vectors $e_{j}$ for $j\le a_{n}$.
\par\smallskip
We denote by
 $\sigma  $ the unique increasing bijection from the set
 $[1,+\infty [ \setminus\bigcup_{n\geq 1}J_{n}$, where $J_{n}$ is the interval
  $J_{n}=[a_{n}-(\kappa_{n}-\kappa_{n-1}-1), a_{n}]$,
   onto the set $[1,+\infty [$.
\par\medskip
$\bullet$ \textbf{Definition of the vectors $f_{j}$: }
We set
\begin{eqnarray*}
&& f_{0}=g_{0}\\
&&f_{j}=g_{\sigma  (j)} \quad \textrm{ for } j\not\in \bigcup_{n\geq 1}
J_{n}\\
&&f_{a_{n}-k}=z_{\kappa_{n}-k} \quad \textrm{ for } k=0, \ldots, \kappa_{n}-\kappa_{n-1}-1.
\end{eqnarray*}

Since in our definition of the vectors $f_{j}$ we simply alternate in our choice
between the vectors
$g_{n}$ and $z_{n}$, but without changing the order,
it is not difficult to see that
 the vectors $f_{j}$ form a normalized Schauder basis of $X$. If
$(f_{j}^{*})_{j\geq 0}$ denotes the sequence of coordinate functionals \wrt\ this basis, we have
$$||\sum_{j\geq 0}f_{j}^{*}(x)f_{j}||=\left(\sum_{j\not\in
\bigcup_{n\geq 1}J_{n}}|f_{j}^{*}(x)|^{p}+
||\sum_{j\in \bigcup_{n\geq 1}J_{n}} f_{j}^{*}(x) z_{j}||^{p}\right)^{\frac{1}{p}}$$
if $X=\ell_{p}\oplus_{\ell_{p}} Z$, and
$$||\sum_{j\geq 0}f_{j}^{*}(x)f_{j}||=\max \left(\sup_{j\not\in
\bigcup_{n\geq 1}J_{n}}|f_{j}^{*}(x)|\,,\,
||\sum_{j\in \bigcup_{n\geq 1}J_{n}} f_{j}^{*}(x) z_{j}||\right)$$ if $X=c_{0}
\oplus_{\ell_{\infty}}
Z$.
Observe that $||f_{j}^{*}||_{X^{*}}=1$ if $j\not \in \bigcup_{n\geq 1}J_{n}$,
and $||f_{a_{n}-k}^{*}||_{X^{*}}=||z_{\kappa_{n}-k}^{*}||_{Z^{*}}$ for $n\geq 1$
and $k=0,\ldots, \kappa_{n}-\kappa_{n-1}$.
\par\smallskip
For any $N\geq 0$ we denote by $\pi_{[0,N]}$ the canonical projection on
the first $N+1$ vectors of the basis $(f_{j})_{j\ge 0}$: for $x=\sum_{j\geq 0}f_{j}^{*}(x)f_{j}$,
$\pi_{[0,N]}x=\sum_{j=0}^{N}f_{j}^{*}(x)f_{j}$. Obviously $\pi_{[0,N]}x$
tends to $x$ as $N$ tends to infinity.
\par\smallskip
The vectors $e_{j}$ are defined depending on the vectors $f_{j}$, with
$e_{0}=f_{0}=g_{0}$, and in such a way that for every $j\geq 0$ we have
$\textrm{sp}[e_{0},\ldots, e_{j}]=\textrm{sp}[f_{0},\ldots, f_{j}]$. Thus
the vectors $e_{j}$ will also be linearly independent and will span a
dense subspace of $X$.

\par\medskip
$\bullet$ \textbf{Definition of the operator $T$: }
The \op\ $T$ is defined by $Te_{j}=e_{j+1}$ for $j\geq 0$. Since the vectors
$e_{j}$ are linearly independent, the definition makes sense. We will
show later on that the vectors $e_{j}$ are defined in such a way that $T$
extends to a bounded linear \op\ on $X=\overline{\textrm{sp}}[e_{j}
\textrm{ ; } j\geq
0]$.
\par\medskip
As mentioned already in our informal introduction, 
 $e_{j}$ is defined for $j\in [\xi_{n}+1,\xi _{n+1}]$ differently,
depending on whether $j$ belongs to a \emph{working interval} or to a
 \emph{lay-off
interval}. The working intervals are of three different types:
the (a)-, (b)- and (c)-working intervals. At step $n$ they are constructed using $a_{n}$, $b_{n}$ and $c_{1,n},\ldots, c_{k_{n},n}$ respectively.
\par\medskip
$\bullet$ \textbf{Definition of the vectors $e_{j}$ for $j$ in an (a)-working interval: }
There is only one (a)-working interval at step $n$, the interval
$[a_{n}-(\kappa_{n}-\kappa_{n-1}-1), a_{n}]$. In the constructions of \cite{R1}, \cite{R2}, \cite{R3},
there are several successive (a)-working intervals of decreasing length,
whose role is to improve at each step the approximation of $e_{0}$ by
 vectors of any
orbit of a \nz\ vector. There will also be several (a)-working intervals in the proof of Theorem \ref{th1}.
But here since $\varepsilon $ is fixed in the statement of Theorem \ref{th2},
 we only need to consider the simpler case where there is just one (a)-interval.
 Of course $a_{n}$ is chosen very large \wrt\ $\xi_{n}$ for each $n\ge 1$.
\par\smallskip
If $\kappa _{n}>\kappa _{n-1}-1$, we define for $j\in[a_{n}-(\kappa_{n}-\kappa_{n-1}-1),
 a_{n}-1]=J_{n}\setminus \{a_{n}\}$ 
$$f_{a_{n}-k}=\frac {1}{a_{n}^{k+1}}e_{a_{n}-k}\quad \textrm{ for }k=1,\ldots, \kappa_{n}-\kappa_{n-1}-1,$$
i.e. $$e_{a_{n}-k}=a_{n}^{k+1}f_{a_{n}-k}.$$
\par\smallskip
Then we define $$f_{a_{n}}=\frac {1}{\alpha _{n}}(e_{a_{n}}-e_{a_{n-1}}), $$ i.e.
$$e_{a_{n}}=\alpha _{n}f_{a_{n}}+e_{a_{n-1}}=\alpha _{n}z_{\kappa_{n}}+e_{a_{n-1}}.$$
\par\smallskip
Let us record already here the following immediate consequence of this
definition, which will be important later on in the proof:

\begin{fact}\label{facta}
For every $n\geq 1$ we have
$e_{a_{n}}=\sum_{k=1}^{n}\alpha _{k}z_{\kappa_{k}}+e_{0},$ so that
$||e_{a_{n}}-e_{0}||<\varepsilon .$
\end{fact}

\par\medskip
$\bullet$ \textbf{Definition of the vectors $e_{j}$ for $j$ in a (b)-working interval: }
The (b)-fan consists of the collection of the $a_{n}$ (b)-working intervals $[r(b_{n}+1),rb_{n}+a_{n}]$, $r=1,\ldots,
a_{n}$, where $b_{n}$ is extremely large \wrt\ $a_{n}$. These successive intervals have decreasing length, from length $a_{n}$ for the first one to length $1$ for the last one. For $j$
in one of these intervals, $e_{j}$ is defined so that
$$f_{j}=e_{j}-b_{n}e_{j-b_{n}},$$ i.e. $$e_{j}=f_{j}+b_{n}e_{j-b_{n}}.$$
The (b)-fan terminates at the index $\nu _{n}=a_{n}(b_{n}+1)$.
\par\smallskip
These (b)-working intervals are the same as the ones considered in
\cite{GR}, except for the fact that there are $a_{n}$ of them, and not
$\xi_{n}$, and that their length decreases from $a_{n}$ to $1$, and
not from $\xi_{n}$ to $1$. Their role is to ensure that
$||(T^{b_{n}+1}/b_{n})x-Tx||$
is very small whenever $x$ belongs to $F_{a_{n}}=\textrm{sp}[e_{0},\ldots,
e_{a_{n}}]$ (see Fact \ref{factg} below).

\par\medskip
$\bullet$ \textbf{Definition of the vectors $e_{j}$ for $j$ in a (c)-working interval: }
Let us choose polynomials $p_{k,n}$, $k=1,\ldots, k_{n}$ forming an
$\varepsilon _{n}$-net of the ball of $\K_{b_{n}+a_{n}+1}[\zeta ]$ of radius
$2$ for the norm $|\,.\,|$ (the number $\varepsilon _{n}$ will be
chosen extremely small after the construction of $a_{n}$ and $b_{n}$). Let $l_{n}=b_{n}+a_{n}+1$.
\par\smallskip
The (c)-fan consists the collection of the $h_{n}k_{n}$ disjoint (c)-working intervals of length
$\nu_{n}$:
$$I_{s_{1},\ldots, s_{k_{n}}}=[s_{1}c_{1,n}+\ldots+s_{k_{n}}c_{k_{n},n},
s_{1}c_{1,n}+\ldots+s_{k_{n}}c_{k_{n},n}+\nu_{n}],$$ where $\nu_{n}=
a_{n}(b_{n}+1)\ll c_{1,n}\ll c_{2,n}
\ll \ldots \ll c_{k_{n},n}$, and $s_{1},\ldots , s_{k_{n}}$ are
nonnegative integers, at least one of which is \nz, belonging to
$[0,h_{n}]$. Here $h_{n}$ is an integer which is very large depending on
the construction until step
$\nu_{n}$, but not on the integers $c_{k,n}$. Let us write $s=(s_{1},\ldots , s_{k_{n}})$
and $|s|=s_{1}+\ldots +s_{k_{n}}$. Let $t$ be the largest integer
such that $s_{t} $ is \nz. For $j\in I_{s_{1},\ldots , s_{k_{n}}}$ we
define $f_{j}$ so that
$$f_{j}=\frac {4^{1-|s|}}{\gamma _{n}}(e_{j}-p_{t,n}(T)e_{j-c_{t,n}}),$$
i.e. $$e_{j}=\gamma _{n}4^{|s|-1}f_{j}+p_{t,n}(T)e_{j-c_{t,n}}$$ where $\gamma _{n}$
is a very small positive number depending only on $\nu _{n}$ which is
suitably chosen in the proof. The role of the initial (c)-intervals $[c_{k,n}, c_{k,n}+\nu _{n}]$, $1\le k\le k_{n}$, is to make
$||T^{c_{k,n}}x-p_{k,n}(T)x||$ very small whenever $x$ belongs to $F_{\nu _{n}}=
\textrm{sp}[e_{0},\ldots, e_{\nu _{n}}]$ (see Fact \ref{factb}), and thus to ensure that
Property (P1) holds true. Observe that it ensures also that $e_0$ is a hypercyclic vector for $T$. The other (c)-intervals (i.e. those for which $|s|>1$) are ``shades'' of these initial intervals, and their role is to ensure a uniform tail estimate on the \ops\ $T^{c_{k,n}}$, $1\le k\le k_{n}$ (see Proposition \ref{prop2}).

\par\medskip
$\bullet$ \textbf{Definition of the vectors $e_{j}$ for $j$ in a \loi: }
The \loi s are the intervals which lie between the working intervals. If
we write such an interval as $[k+1,k+l]$, we define $e_{j}$ for $j$ in
it so that $f_{j}=\lambda _{j}e_{j}$ where
$$\lambda _{j}=2^{\frac {1}{\sqrt{l}}(\frac {1}{2}l+k+1-j)}.$$
Hence when the length $l$ of the interval becomes very large, $\lambda _{j}$ is
very large when $j$ lies in the beginning of the \loi\ ($\lambda _{j}$ is approximately equal
to $2^{\frac {1}{2}\sqrt{l}}$ there) and very small when $j$ lies in the end of
the \loi\ ($\lambda _{j}$ is approximately equal to $2^{-\frac {1}{2}\sqrt{l}}$ this
time). The ratio $\lambda _{j}/\lambda _{j+1}$, which does not depend on
$j$ and is equal to $2^{-\frac{1}{\sqrt{l}}}$, becomes very close to $1$ when $l$ is large.
\par\smallskip
Just as in \cite{GR}, the definition of the scalars $\lambda_{j}$ when $j$
lies in some of the \loi s between the working intervals will be slightly
modified, just in order to make computations simpler. For $j\in [a_{n}+1, b_{n}]$, we set
$$\lambda _{j}=2^{\frac {1}{\sqrt{b_{n}}}(\frac {1}{2}b_{n}+a_{n}+1-j)}$$
 and for $j\in [rb_{n}+a_{n}+1, (r+1)b_{n}-1]$ we set
 $$\lambda _{j}=2^{\frac {1}{\sqrt{b_{n}}}(\frac {1}{2}b_{n}+rb_{n}+a_{n}+1-j)}.$$ This modification does not change anything to the asymptotic size of the coefficients $\lambda _{j}$.
\par\smallskip
The last index $\xi_{n+1}$ is defined as being very large \wrt\
$h_{n}c_{k_{n},n}$, so that the interval $[\xi_n+1,\xi_{n+1}]$ contains all the working intervals constructed at step $n$.

\subsection{Boundedness of $T$}
Let us show first that $T$ extends to a bounded linear \op\ on $X$. Most
of the work for this has been done already in \cite{GR}: we are considering
the same type of \op\ as in \cite{GR} as far as the (b)- and (c)-parts are concerned, and we are adding to it the (a)-\woi s. We will prove a
more precise statement:

\begin{proposition}\label{prop1}
Let $\rho >0$ be any (small) positive number.
Provided the sequences $(a_{n})$, $(b_{n})$, etc... increase
sufficiently fast, $T$ can be written as $T=S+K$, where

$\bullet$ $S$ is an \op\ of the form $S=S_{0}\oplus 0$ when seen as
acting on $\ell_{p}\oplus Z$ or $c_{0}\oplus Z$, and $S_{0}$ is a
forward weighted shift on $\ell_{p}$ or $c_{0}$ \wrt\ the basis $(g_{j})_{j\geq
0}$: $S_{0}g_{j}=w_{j}g_{j+1}$, where $(w_{j})_{j\geq 0}$ is a certain
sequence of real numbers with $0\leq w_{j}\leq 1+\rho $ for every $j\geq 0$;

$\bullet$ $K$ is a nuclear \op\ on $X$;
 there exists a
sequence $(u_{j})_{j\geq 0}$ of vectors of $X$ with
$$\sum_{j\geq 0}(1+||f_{j}^{*}||)||u_{j}||<\rho $$ such that
$$Kx=\sum_{j\geq 0}f_{j}^{*}(x)u_{j} \quad \textrm{ for every } x\in X.$$
So $T$ is bounded with $||T||\leq 1+2\rho $, and in particular $||T||\leq
2$.
\end{proposition}

Recall that for each $n\ge 1$, $J_{n}$ denotes the (a)-\woi\ $[a_{n}-(\kappa_{n}-\kappa_{n-1}), a_{n}]$. We will call $J$ the set of all right endpoints of either a working or a \loi, and $\tilde{J}$ the set $\tilde{J}=J\cup\bigcup_{n\ge 1} J_{n}$.

\begin{proof}
It is not difficult to see that when $j$ does not belong to  $\tilde{J}$, then $Tf_{j}=w_{j}f_{j+1}$ for some real number $w_{j}$ with
$1-\rho \leq w_{j}\leq 1+\rho $
provided each one of the quantities $a_{n}$, etc... involved in the construction
is chosen sufficiently large \wrt\ the previous one. Observe indeed that when $j\not \in \tilde{J}$, then $j+1$
is not in the set $\bigcup_{n\geq 1} J_{n}$,
so $f_{j+1}=g_{\sigma (j+1)}=g_{\sigma
(j)+1}$. It then follows from the definition of $e_{j}$ in the (b)- and (c)-\woi s and in the \loi s 
that $Tg_{\sigma (j)}=w_{j}g_{\sigma (j)+1}$. Set $w_{j}=0$ when
$j\in J\setminus \bigcup_{n\geq 1} J_{n}$, and define $S_{0}$ on $\ell_{p}$
or $c_{0}$ by setting $S_{0}f_{j}=w_{j}f_{j+1}$ for $j\not\in \bigcup_{n\geq 1}
J_{n}$. Then $S_{0}$ is a bounded forward shift \wrt\ the basis $(g_{j})_{j\geq 0}$
of $\ell_{p}$ or $c_{0}$. If $S=S_{0}\oplus 0$ on $\ell_{p}\oplus Z$ or $c_{0}\oplus
Z$, then $Sf_{j}=S_{0}f_{j}$ for $j\not \in \bigcup_{n\geq 1} J_{n}$,
and $Sf_{j}=0$ for $j\in \bigcup_{n\geq 1} J_{n}$.
\par\smallskip
When $j$ is an endpoint of either a working or a \loi\ (i.e. $j\in J$), but $j\not\in \bigcup_{n\geq 1}
J_{n}$, exactly the same proof as in \cite[Prop. 2.4]{GR} shows that

$\bullet$ if $j$ is a right endpoint of a (b)-working or \loi\ (we denote by $J_{b}$ the set of such integers), $||Tf_{j}||\ls
b_{n}^{a_{n}}2^{-\frac {1}{2}\sqrt{b_{n}}}$.

$\bullet$ if $j$ is a right endpoint of a (c)-working or \loi\ (we denote by $J_{c}$ the set of such integers), then we have the estimate
 $||Tf_{j}||\ls
C_{h_{n},k_{n}} 2^{-\frac {1}{2}\sqrt{c_{1,n}}}$ where $C_{h_{n},k_{n}}$ is a constant depending on $h_{n}$ and $k_{n}$ only.
\par\smallskip
To complete the proof let us check the estimates for $||Tf_{j}||$ in the
remaining cases:
\par\smallskip
\textbf{Case 1:} if $j=a_{n}-(\kappa_{n}-\kappa_{n-1})$
we have
$$||Tf_{a_{n}-(\kappa_{n}-\kappa_{n-1})}||=||\lambda _{a_{n}-(\kappa_{n}-\kappa_{n-1})}
e_{a_{n}-(\kappa_{n}-\kappa_{n-1}-1)}||\simeq 2^{-\frac {1}{2}\sqrt{a_{n}}}a_{n}^{\kappa_{n}-\kappa_{n-1}}
$$ so that $||Tf_{a_{n}-(\kappa_{n}-\kappa_{n-1})}||$
can be made arbitrarily small if $a_{n}$ is large enough. This estimates applies in particular to $||Tf_{a_{n}-1}||$ in the case where $\kappa _{n}=\kappa _{n-1}+1$.
\par\smallskip
\textbf{Case 2:} if $\kappa _{n}>\kappa _{n-1}-1$ and $j$ belongs to the set $J_{n}=[a_{n}-(\kappa_{n}-\kappa_{n-1}-1), a_{n}]$, then we have two subcases to consider:
\par\smallskip
\textbf{Case 2a:} if  $k=2,\ldots, \kappa_{n}-\kappa_{n-1}-1$ we have
$$Tf_{a_{n}-k}=\frac {1}{a_{n}^{k+1}}e_{a_{n}-(k-1)}=\frac {a_{n}^{k}}{a_{n}^{k+1}}f_{a_{n}-k+1}=\frac {1}{a_{n}}
f_{a_{n}-k+1}$$ and $||Tf_{a_{n}-k}||=\frac {1}{a_{n}}$.
\par\smallskip
\textbf{Case 2b:} if $k=1$, then
$$Tf_{a_{n}-1}=\frac {1}{a_{n}^{2}}e_{a_{n}}$$ 
Hence
$||Tf_{a_{n}-1}||\le \frac {1}{a_{n}^{2}} (1+\varepsilon ),$
which can again be made arbitrarily small  provided $a_{n}$ is large enough.
\par\smallskip
\textbf{Case 3:} lastly $$Tf_{a_{n}}=\frac {1}{\alpha _{n}}(e_{a_{n}+1}-e_{a_{n-1}+1})=\frac {1}{\alpha _{n}}
\left(\frac {1}{\lambda _{a_{n}+1}}f_{a_{n}+1}-
\frac {1}{\lambda _{a_{n-1}+1}}f_{a_{n-1}+1}\right)$$ so that
$||Tf_{a_{n}}||\ls \frac {1}{\alpha _{n}}2^{-\frac {1}{2}\sqrt{a_{n-1}}}\ls
2^{-\frac {1}{4}\sqrt{a_{n-1}}}$.
\par\smallskip
Hence we have shown that if $j$ belongs to the set $\tilde{J}$, we have very small bounds on the norms of the vectors
$||Tf_{j}||$.
Let us now define, for $x=\sum_{j\geq 0}f_{j}^{*}(x)f_{j}$, $$Kx=\sum_{j\in \tilde{J}}
f_{j}^{*}(x)Tf_{j}.$$ We have
\begin{eqnarray*}
\sum_{j\in \tilde{J}}(1+ ||f_{j}^{*}||)||Tf_{j}||&=&\sum_{n\geq 1} \sum_{j\in [\xi _{n}+1,
\xi _{n+1}]} (1+||f_{j}^{*}||)||Tf_{j}||\\
&\ls&\sum_{n\geq 1} \Bigl(
\sum_{j\in [\xi _{n}+1,
\xi _{n+1}]\cap J_{b}} (1+||f_{j}^{*}||)
b_{n}^{a_{n}}2^{-\frac {1}{2}\sqrt{b_{n}}}\Bigr.\\
&+&
\sum_{j\in [\xi _{n}+1,
\xi _{n+1}]\cap J_{c}}(1+||f_{j}^{*}||) C_{h_{n},k_{n}}2^{-\frac
{1}{2}\sqrt{c_{1,n}}}\\
&+&\Bigl.
\sum_{k=1}^{\kappa_{n}-\kappa_{n-1}}
(1+ ||f_{a_{n}-k}^{*}||) \frac {1}{a_{n}}
+||f_{a_{n}}^{*}||
2^{-\frac {1}{2}\sqrt{a_{n-1}}}\Bigr).
\end{eqnarray*}
Since
there are $a_{n}$ (b)-working intervals, and $h_{n}k_{n}$ (c)-working
intervals, where $h_{n}$ depends only on $\nu_{n}$ but not on the integers
$c_{k,n}$, and
$||f_{j}^{*}||=1$ if $j\not \in \bigcup_{n\geq 1}J_{n}$,
$$\sum_{j\in [\xi _{n}+1,
\xi _{n+1}]\cap J_{b}}(1+ ||f_{j}^{*}||) b_{n}^{a_{n}} 2^{-\frac {1}{2}\sqrt{b_{n}}}\quad \textrm{ and }
\quad
\sum_{j\in [\xi _{n}+1,
\xi _{n+1}]\cap J_{c}}(1+ ||f_{j}^{*}||) C_{h_{n},k_{n}}2^{-\frac {1}{2}\sqrt{c_{1,n}}}
$$
can be made arbitrarily small if $b_{n}$ and $c_{1,n}$ are large enough.
Since $||f_{a_{n}-k}^{*}||=||z_{\kappa_{n}-k}^{*}||_{Z^{*}}$ for $n\geq 1$
and $k=0,\ldots, \kappa_{n}-\kappa_{n-1}+1$, these quantities do not depend on the
choice of $\xi_{n}$, $a_{n}$, etc... Moreover $\kappa_{n}$ is a fixed quantity which does not depend on $a_{n}$, and so the last terms can be made
arbitrarily small too. This shows that for any $\rho >0$ we can
manage so that $\sum_{j\in \tilde{J}}(1+ ||f_{j}^{*}||)||Tf_{j}||<\rho.$
It is clear that $Tf_{j}=(S+K)f_{j}$ for every $j\geq 0$, and this
proves in particular that $T$ is a bounded \op\ on $X$.
\end{proof}

\begin{remark}\label{remadd}
The proof of Proposition \ref{prop1} gives us a bit more information on
the structure of the operators $S$ and $K$: it actually shows that
 $Sf_{j}=0$ if and only if $j\in \tilde{J}$. When $j$ does not belong to $
\tilde{J}$, $j$ and $j+1$ are not in the set $\bigcup_{n\ge 1}J_{n}$, and so $f_{j}=g_{\sigma  (j)}$, $f_{j+1}=g_{\sigma  (j+1)}=g_{\sigma  (j)+1}$, and  $Sf_{j}=w_{j}f_{j+1}$, where $w_{j}$ becomes closer and
closer to $1$ as $j$ tends to infinity. Also $Kf_{j}=0$ as soon as $j\not\in
\tilde{J}$. Hence $Tf_{j}=Sf_{j}$ if $j\not\in\tilde{J}$ and $Tf_{j}=Kf_{j}=u_{j}$
if $j\in\tilde{J}$, so that
$$Tx=\sum_{j\not\in\tilde{J}} f_{j}^{*}(x)w_{j}f_{j+1}+\sum_{j\in\tilde{J}}f_{j}^{*}(x)u_{j}.$$ In particular if we write the \op\ $K$ as $Kx=\sum_{j\ge 0}
f_{j}^{*}(x)u_{j}$, then $u_{j}$ is \nz\ \ifff\
$j\not\in\tilde{J}$.
\par\smallskip
If we denote by $\tilde{S}$ the forward weighted shift defined by $\tilde{S}f_{j}=0$ if $j\in\tilde{J} $ and $Sf_{j}=f_{j+1}$ if $j\not\in\tilde{J}$, it is not difficult to see that $S-\tilde{S}$ is a compact operator. Hence $T$ can also be written as $T=\tilde{S}+L$, where $L$ is a compact operator. Since $\tilde{S}$ is a contraction, $T$ is in particular a compact perturbation of a power-bounded operator. We will come back to this observation later in Section \ref{sec7} of the paper.
\end{remark}

\begin{remark}\label{remaddbis}
Observe also that if $X=\ell_{p}\oplus Z$ with $1<p<+\infty $, then $T$
 is weakly compact.
\end{remark}

\section{Operators with few non-trivial invariant closed subsets  on non-reflexive spaces with a basis: proof of Theorem \ref{th2}}\label{sec3}

\subsection{Proof of Property (P1)}
Let us begin by proving that $T$ satisfies Property (P1) of Theorem \ref{th2}
above: the idea is the same as in \cite[Section 2]{GR}, but the introduction of the
(a)-interval makes some modifications necessary. Let us recall the
argument of \cite{GR}: we want to show that for any $x\not =0$, any
\pol\ $p$ and any $\delta  >0$, it is possible to find an integer $c$
such that
$||T^{c}x-p(T)x||<\delta $. It is not difficult to see that it suffices
to do this for $p$ satisfying the condition $|p|\leq 2$. Indeed if for any $\delta >0$ and any polynomial $p$ with $|p|\leq 2$ there exists an integer $c$ such that $||T^{c}x-p(T)x||<\delta $, then if $q$ is any polynomial, with $|q|\le 2^j$ for some positive integer $j$, we can find an integer $c_j$ such that
$||T^{c_j}x-2^{-j}q(T)x||<\delta 2^{-2j}$. Then we can find an integer $c_{j-1}$ such that
$||T^{c_{j-1}}x-2T^{c_j}x||<\delta 2^{-(2j-1)}$, etc... and thus there exists an integer $c_0$ such that $||T^{c_0}x-q(T)x||<\delta $.
\par\smallskip
In order to prove that for any $\delta >0$ and any polynomial $p$ with $|p|\le2$ there exists an integer $c$ such that $||T^{c}x-p(T)x||<\delta $, we proceed in the proof of \cite[Theorem 1.3]{GR} as follows:
 we first decompose $x$ as $x=\pi_{[0,\nu_{n}]}x+(x-
\pi_{[0,\nu_{n}]}x)$ where $\pi_{[0,\nu_{n}]}x=\sum_{j\geq 0}^{\nu _{n}}f_{j}^{*}(x)f_{j}$
is the projection of $x$ onto $F_{\nu _{n}}=\textrm{sp}[e_{j} \textrm{ ; }j\leq \nu
_{n}]$. Then the two crucial steps are the following:
\par\smallskip
$\bullet$ observe that for any $\delta _{n}>0$ the construction can be
carried out in such a way that for every $y$  in $F_{\nu _{n}}$ and  every $k\in [1,k_{n}]$, we have
$||T^{c_{k,n}}y-p_{k,n}(T)y||\leq \delta _{n}||y||$ (this is \cite[Fact 2.1]{GR});
\par\smallskip
$\bullet$ show that whenever $f_{j}^{*}(y)=0$ for every $j=0,\ldots, \nu_{n}$,
$||T^{c_{k,n}}y||\leq 100\,||y||$ for every $n\geq 1$ and every $k\in
[1,k_{n}]$, i.e. that $||T^{c_{k,n}}(I-\pi_{[0,\nu_{n}]})||\le 100$ (this is \cite[Proposition 2.2]{GR}).
\par\smallskip
Then it is not difficult to see that $||T^{c_{k,n}}x-p_{k,n}(T)x||$
becomes very small for $n$ very large, and using the fact that
$(p_{k,n})_{1\leq k\leq k_{n}}$ forms an $\varepsilon _{n}$-net of the set
of \pol s of degree at most $l_{n}$
with $|p|\leq 2$, with $\varepsilon _{n}$ so small that in particular $\varepsilon _{n}<4^{-\nu _{n}}$, we easily get that Property (P1) is satisfied.
\par\smallskip
Introducing the (a)-interval does not change anything to the first step:

\begin{fact}\label{factb}
Let $\delta _{n}$ be any small positive number. Provided $\gamma _{n}$
is small enough, we have
$||T^{c_{k,n}}y-p_{k,n}(T)y||\leq \delta _{n}||y||$ for any $y$
belonging to $F_{\nu _{n}}$ and any $k\in [1,k_{n}]$.
\end{fact}

But the uniform estimate on $||T^{c_{k,n}}y||$ when $y$ is supported in $[\nu _{n}+1,
+\infty [$ is not true anymore. Consider the action of $T^{c_{k,n}}$ on the vector $f_{a_{n+1}}$. We have $$T^{c_{k,n}}f_{a_{n+1}}=\frac {1}{\alpha _{n+1}}
(e_{a_{n+1}+c_{k,n}}-e_{a_{n}+c_{k,n}}).$$ Obviously $||e_{a_{n+1}+c_{k,n}}||$
can be made sufficiently small, since $a_{n+1}+c_{k,n}$ lies in the beginning of a \loi\ of length comparable to $b_{n+1}$, but this is not the case for the norm of
$e_{a_{n}+c_{k,n}}=\gamma _{n}f_{c_{k,n}+a_{n}}+p_{k,n}(T)e_{a_{n}}$: we
have no control at all on the behavior of this vector, which may have
very large norm. To circumvent this problem, the idea is to replace the
natural projection $\pi_{[0,\nu_{n}]}x$ of $x$ on $F_{\nu _{n}}$ by
another projection $Q_{\nu _{n}}x$ on $F_{\nu _{n}}$. Such a projection
was used by Read in \cite{R1}, \cite{R2} and \cite{R3}, but we did not
need it in \cite{GR} because there was no (a)-interval there. The
definition of $Q_{\nu _{n}}x$ is very natural: we just take out the
part of $x$ which is making trouble, i.e. the part corresponding to $a_{n+1}$:
\begin{equation*}
Q_{\nu _{n}}f_{j}=\begin{cases}
f_{j} &\textrm{ if } j\leq \nu _{n}\\
-\frac {1}{\alpha _{n+1}}e_{a_{n}} & \textrm{ if } j=a_{n+1}\\
0 & \textrm{ in all the other cases.}
\end{cases}
\end{equation*}

Clearly $Q_{\nu _{n}}x$ belongs to $F_{\nu _{n}}$ for every $x\in X$,
and $Q_{\nu _{n}}$ is a projection of $X$ onto $F_{\nu _{n}}$. We will need the following straightforward fact about these projections:

\begin{fact}\label{factbis}
There exists a positive constant $M$ such that $||Q_{\nu _{n}}||\le M$ for each $n\ge 1$.
\end{fact}

\begin{proof}
 The proof is immediate, recalling the fact that $\alpha_n\ge \delta_{0}$ for all $n\ge 1$: we have $$Q_{\nu _{n}}x=\pi_{[0,\nu _{n}]}x-\frac{1}{\alpha _{n+1}}f^{*}_{a_{n+1}}(x)e_{a_{n}},$$ so that $$||Q_{\nu _{n}}||\le ||\pi_{[0,\nu _{n}]}||+\frac{1}{\alpha _{n+1}}||f^{*}_{a_{n+1}}||\,||e_{a_{n}}||\le C+\frac{1+\varepsilon }{\delta _{0}}||z^{*}_{\kappa_{n+1}}||\le C+\frac{1}{\delta _{0}}\,\sup_{k}||z^{*}_{k}||=:M,$$ 
 where $C$ denotes the constant of the basis $(f_{j})_{j\ge 0}$.
\end{proof}

We have
\begin{equation*}
(I-Q_{\nu _{n}})f_{j}=\begin{cases}
0 &\textrm{ if } j\leq \nu _{n}\\
\frac {1}{\alpha _{n+1}}e_{a_{n+1}} & \textrm{ if } j=a_{n+1}\\
f_{j} & \textrm{ in all the other cases.}
\end{cases}
\end{equation*}
so that $$||T^{c_{k,n}}(I-Q_{\nu _{n}})f_{a_{n+1}}||=\frac {1}{\alpha _{n}}||e_{a_{n+1}+c_{k,n}}||
\ls 2^{-\frac {1}{2}\sqrt{a_{n+1}}}$$ is very small.
Then the same reasoning as in the proof of \cite[Proposition 2.2]{GR} will yield a uniform estimate on the norm
$||T^{c_{k,n}}(I-Q_{\nu _{n}})x||$ when $x$ is supported in $[\nu _{n}+1,+\infty
[$. The only difference between the situation in \cite{GR} and the
situation here lies in the estimates of $||T^{c_{k,n}}(I-Q_{\nu _{n}})f_{j}||=||T^{c_{k,n}}f_{j}||$
when $j\in [\xi _{n+1}+1, a_{n+1}-1]$. For such indices $j$, $f_{j}=\lambda _{j}e_{j}$ where
$$\lambda _{j}=2^{\frac {\frac {1}{2}(a_{n+1}-(\kappa_{n+1}-\kappa_{n}))+\xi _{n+1}+1-j}{
\sqrt{a_{n+1}-(\kappa_{n+1}-\kappa_{n})-1}}}\quad \textrm{ for } j\in [\xi _{n+1}+1,
 a_{n+1}-(\kappa_{n+1}-\kappa_{n})]$$
 and
$$\lambda _{j}={a_{n+1}^{-(a_{n+1}-j+1)}}\quad \textrm{ if } \kappa _{n}>\kappa _{n-1}+1 \textrm{ and }
j\in [ a_{n+1}-(\kappa_{n+1}-\kappa_{n})+1, a_{n+1}-1].$$ Thus $T^{c_{k,n}}f_{j}=\lambda _{j}
e_{j+c_{k,n}}$.
\par\smallskip
$\bullet$ When $j\in [\xi _{n+1}+1,
 a_{n+1}-(\kappa_{n+1}-\kappa_{n})-c_{k,n}]$, $j+c_{k,n}\in
 [\xi _{n+1}+1,
 a_{n+1}-(\kappa_{n+1}-\kappa_{n})]$ so $$T^{c_{k,n}}f_{j}=\frac {\lambda _{j}}{\lambda _{j+c_{k,n}}}
 f_{j+c_{k,n}}\textrm{ which is very close to } 2^{\frac {1}{2}c_{k,n}\frac
 {1}{\sqrt{a_{n+1}}}}f_{j+c_{k,n}}.$$ If $a_{n+1}$ is extremely large
 \wrt\ the coefficients $c_{k,n}$, the quantity $2^{\frac {1}{2}c_{k,n}\frac
 {1}{\sqrt{a_{n+1}}}}$ is very close to $1$, and thus $T^{c_{k,n}}g_{\sigma (j)}$ is a multiple of $
 g_{\sigma (j+c_{k,n})}$ which is very close to $
 g_{\sigma (j+c_{k,n})}$. Thus $T^{c_{k,n}}$ acts as a shift with weights very close to $1$ on this part of the basis.

$\bullet$ When $j\in [a_{n+1}-(\kappa_{n+1}-\kappa_{n})-c_{k,n}+1, a_{n+1}-(\kappa_{n+1}-\kappa_{n})]$, we have
$\lambda _{j}\ls 2^{-\frac {1}{2}\sqrt{a_{n+1}}}$. Then we have to estimate the norm
$||e_{j+c_{k,n}}||$, using that $j+c_{k,n}$ belongs to the interval $
[a_{n+1}-(\kappa_{n+1}-\kappa_{n})+1, a_{n+1}-(\kappa_{n+1}-\kappa_{n})+c_{k,n}]$. The norm of $e_{j+c_{k,n}}$ is the largest when $j+c_{k,n}=a_{n+1}-(\kappa_{n+1}-\kappa_{n})+1$: we have (if $\kappa _{n+1}>\kappa _{n}-1$) $$||e_{a_{n+1}-(\kappa_{n+1}-\kappa_{n}+l)}||=a_{n+1}^{\kappa_{n+1}-\kappa_{n}-l+1} \textrm{ for } l\in [1, \kappa_{n+1}-\kappa_{n}-1];$$ if 
$j+c_{k,n}=a_{n+1}$,
$||e_{a_{n+1}}||\leq 1+\varepsilon $, and if $j+c_{k,n}>a_{n+1}$,
$||e_{j+c_{k,n}}||\ls 2^{-\frac {1}{2}\sqrt{b_{n+1}}}$ is extremely
small. This gives the bound
 $||e_{j+c_{k,n}}||
\ls a_{n+1}^{\kappa_{n+1}-\kappa_{n}}$. So in any case $$||T^{c_{k,n}}f_{j}||\ls a_{n+1}^{\kappa_{n+1}-\kappa_{n}}\,
 2^{-\frac {1}{2}\sqrt{a_{n+1}}}$$ which is extremely small.

$\bullet$ When
$j\in [a_{n+1}-(\kappa_{n+1}-\kappa_{n})+1, a_{n+1}-1]$ (and $\kappa _{n+1}>\kappa _{n}-1$), $j+c_{k,n}$ lies in the
beginning of the \loi\ beginning at $a_{n+1}+1$: indeed $\kappa_{n+1}-\kappa_n$ is fixed at the very beginning of the proof and does not depend on the construction of the vectors $f_j$,
so we may choose each $c_{k,n}$  extremely large \wrt\ $\kappa_{n+1}-\kappa_n$.
Hence $||e_{j+c_{k,n}}||\ls 2^{-\frac {1}{2}
\sqrt{b_{n+1}}}$ and $\lambda _{j}\ls \frac {1}{a_{n+1}^{2}}$ so $||T^{c_{k,n}}f_{j}||=\lambda _{j}||e_{j+c_{k,n}}||$
is very small in this case too.

\par\smallskip
We have already seen that
$||T^{c_{k,n}}(I-Q_{\nu _{n}})f_{a_{n+1}}||=\frac {1}{\alpha _{n}}||e_{a_{n+1}+c_{k,n}}||
\ls 2^{-\frac {1}{2}\sqrt{a_{n+1}}}$ is very small. Incorporating the estimates above in the proof of \cite[Proposition 2.2]{GR}, we obtain that

\begin{proposition}\label{prop2}
For every $n \geq 1$, every $k\in [1,k_{n}]$ and every $x\in X$, we have
\begin{equation}\label{EQ1}
 ||T^{c_{k,n}}(I-Q_{\nu _{n}})x||\leq 100\,||x||,
\end{equation}
so that 
\begin{equation}\label{EQ2}
 ||T^{c_{k,n}}(I-Q_{\nu _{n}})x||\leq 100\,||x-\pi_{[0,\nu_{n}] }x||.
\end{equation}
\end{proposition}

\par\smallskip
Indeed, in order to see that (\ref{EQ2}) is true, since $(I-Q_{\nu _{n}})(x-\pi_{[0,\nu_{n}] }x)=
x-\pi_{[0,\nu_{n}] }x-Q_{\nu _{n}}x
+Q_{\nu _{n}}\pi_{[0,\nu_{n}] }x=(I-Q_{\nu _{n}})x$, we have by (\ref{EQ1})
$$||T^{c_{k,n}}(I-Q_{\nu _{n}})x||=||T^{c_{k,n}}(I-Q_{\nu _{n}})
(x-\pi_{[0,\nu_{n}] }x)||\leq 100\,||x-\pi_{[0,\nu_{n}] }x||.$$
\par\smallskip
In other words, the family of \ops\ $T^{c_{k,n}}(I-Q_{\nu _{n}})$, $n\ge 1$, $k\in [1,k_{n}]$, is bounded in norm, and we have just seen that this implies that for each choice of $k=k(n)\in [1,k_{n}]$, $||T^{c_{k,n}}(I-Q_{\nu _{n}})x||$ tends to $0$ as $n$ tends to infinity for every vector $x\in X$. This uniform estimate is really crucial in the proof of all the results relying on a Read's type construction: as will be seen in the forthcoming proofs, it ensures the control of what can be thought  of as the tails of all vectors of the form $T^{c_{k,n}}x$, $x\in X$. So in order to show that for certain vectors $x$ and $y$ of $X$ there exists an integer $c$ such that $||T^{c}x-y||<\varepsilon$, it suffices to show that there exist arbitrarily large integers $n$ for which $||T^{c_{k,n}}Q_{\nu_n}x-\pi_{[0,\nu_n]}y||<\varepsilon$ for some index $k\in [1,k_{n}]$. Hence the problem boils down to a question concerning vectors living in the finite-dimensional spaces $F_{\nu _{n}}$, which is much more tractable. 

\par\medskip
Let us now go back to the proof of Property (P1), and suppose that in the construction of the (c)-part we have chosen $\varepsilon_n$ so small that $\varepsilon_n<\frac{1}{||Q_{\nu_{n}}||}4^{-\nu_{n}}$.
If $p$ is a polynomial of degree $l$ with $|p|\leq 2$ and $n$ is such that $l_{n}\ge l$, let $k\in [1,k_{n}]$ be such that
$|p-p_{k,n}|\leq \frac{1}{||Q_{\nu_{n}}||}4^{-\nu_{n}}$. Then
\begin{eqnarray*}
||T^{c_{k,n}}x-p(T)x||&\leq &||T^{c_{k,n}}(I-Q_{\nu _{n}})x||+||T^{c_{k,n}}Q_{\nu _{n}}x-p_{k,n}(T)Q_{\nu
_{n}}x||\\
&+&||p_{k,n}(T)Q_{\nu _{n}}x-p(T)Q_{\nu _{n}}x||+||p(T)(I-Q_{\nu _{n}})x||\\
&\leq& 100\,||x-\pi_{[0,\nu _{n}]}x||+\delta _{n}\,||Q_{\nu _{n}}||\,||x||\\
&+&|p-p_{k,n}|\,||T||^{l_{n}}||Q_{\nu _{n}}||
\,||x||+||p(T)||\,||(I-Q_{\nu _{n}})x||\\
&\leq& 100\,||x-\pi_{[0,\nu _{n}]}x||+\delta _{n} M||x||\\
&+& 4^{-\nu _{n}}2^{l_{n}}||x||+||p(T)||\,||(I-Q_{\nu _{n}})x||.
\end{eqnarray*}
The first term is very small if $n$ is large
because $||x-\pi_{[0,\nu _{n}]}x||$ tends to zero as $n$ goes to
infinity, and the second and third term are very small as well. In order to be able to control the last term, we have to show
that $||(I-Q_{\nu _{n}})x||$ tends to zero as $n$ tends to infinity for every $x\in X$. It is at
this point that we use in a really crucial way the fact that the sequence $(\alpha _{n})_{n\geq 1}$
is bounded away from zero:

\begin{fact}\label{factc}
For every $x\in X$, $||(I-Q_{\nu _{n}})x||$ tends to zero as $n$ tends
to infinity.
\end{fact}

\begin{proof}
 We have
$$(I-Q_{\nu _{n}})x=\frac {1}{\alpha _{n+1}}f_{a_{n+1}}^{*}(x)
e_{a_{n+1}}+\sum_{j>\nu _{n}, j\not =a_{n+1}} f_{j}^{*}(x)f_{j}. $$
Obviously the second term tends to $0$ as $n$ tends to infinity since $(f_{j})_{j\ge 0}$ is a Schauder basis of $X$. For the
first term, Fact \ref{facta} combined with the assumption that $\alpha _{n}\ge \delta _{0}>0$ for each $n$ yields that
$$||\frac {1}{\alpha _{n+1}}f_{a_{n+1}}^{*}(x)e_{a_{n+1}}||\leq
\frac {1}{\alpha _{n+1}}|f_{a_{n+1}}^{*}(x)|(1+\varepsilon )\leq |f_{a_{n+1}}^{*}(x)|
\frac {1+\varepsilon }{\delta _{0}} \cdot $$ Since $|f_{a_{n+1}}^{*}(x)|$
tends to zero as $n$ goes to infinity,  $||(I-Q_{\nu _{n}})x||$ tends to zero
too.
\end{proof}
Going back to the proof of Property (P1), we thus get that for each polynomial $|p|\leq 2$
and each vector $x\in X$ there exists for each $n$
an integer $k=k(p,n)$ belonging to $ [1,k_{n}]$ (with $k(p,n)$ depending on $p$ and $n$) such that the quantities $||T^{c_{k(p,n),n}}x-p(T)x||$ go to zero as $n$ goes to infinity. We already explained how this implies the
corresponding assertion with no restriction on $|p|$, and hence property (P1)
is proved.

\subsection{Proof of Property (P2)}
The proof of Property (P2) is extremely similar to the proof presented in \cite[Section 5]{GR}, except for the fact that the projection $\pi_{[0,\nu _{n}]}$ is replaced by a projection $Q_{a_{n}}$ whose definition is very close to that of $Q_{\nu _{n}}$. We define again
\begin{equation*}
Q_{a _{n}}f_{j}=\begin{cases}
f_{j} &\textrm{ if } j\leq a_{n}\\
-\frac {1}{\alpha _{n+1}}e_{a_{n}} & \textrm{ if } j=a_{n+1}\\
0 & \textrm{ in all the other cases.}
\end{cases}
\end{equation*}
The \op\ $Q_{a_{n}}$ defined in this way is a projection of $X$ onto $F_{a_{n}}$, with $Q_{\nu_n}-Q_{a_n}=\pi_{[a_n+1,\nu_n]}$, and
\begin{equation*}
(I-Q_{a_{n}})f_{j}=\begin{cases}
0 &\textrm{ if } j\leq a_{n}\\
\frac {1}{\alpha _{n+1}}e_{a_{n+1}} & \textrm{ if } j=a_{n+1}\\
f_{j} & \textrm{ in all the other cases}
\end{cases}
\end{equation*}
so that in the same way as in Fact \ref{factc} we have

\begin{fact}\label{factd}
 For every $x\in X$, $||(I-Q_{a_{n}})x||$ tends to zero as $n$ tends to infinity.
\end{fact}

If $x$ belongs to $F_{a_{n}}$, we can write $x$ in the $(e_{j})$-basis of $F_{a_{n}}$ as $$x=\sum_{j=0}^{a_{n}}
e_{j}^{*(a_{n})}(x)e_{j}.$$ The reason Property (P2) holds true can be heuristically explained as follows: given two \nz\ vectors $x$ and $y$ of $X$, either there are infinitely many $n$ such that the first ``large'' $e_{j}$-coordinate of $Q_{a_{n}}x$ comes not later than the first ``large'' $e_{j}$-coordinate of $Q_{a_{n}}y$, or the other way round.
In the first case we will have $\overline{\mathcal{O}rb}(y,T)\subseteq\overline{\mathcal{O}rb}(x,T)$, and in the second case
$\overline{\mathcal{O}rb}(x,T)\subseteq\overline{\mathcal{O}rb}(y,T)$.
\par\smallskip

Let us first quantify what is a ``large'' $e_{j}$-coordinate:

\begin{fact}\label{facte}
For every sequence $(C_{a_{n}})_{n\geq 1}$ of positive real numbers increasing sufficiently quickly, the following property holds true:

for every $x\in X$ with $||x||=1$, there exists an $n_{0}$ such that for every $n\geq n_{0}$ there exists an integer $j\in[0,a_{n}]$ with $$|e_{j}^{*(a_n)}(Q_{a_{n}}x)|\geq \frac{1}{C_{a_{n}}^{(a_{n}-j+1)!^{2}}}\cdot$$
\end{fact}

The quantities $C_{a_{n}}^{(a_{n}-j+1)!^{2}}$ may appear unnecessarily complicated in the statement of Fact \ref{facte}, but the precise size of these coefficients will be needed later on in the proof.

\begin{proof}
Indeed suppose that $|e_{j}^{*(a_n)}(Q_{a_{n}}x)|\leq \frac{1}{C_{a_{n}}^{(a_{n}-j+1)!^{2}}}$ for every $j\in [0,a_{n}]$. Then $$||Q_{a_{n}}x||
\leq \sum_{j=0}^{a_{n}} \frac{1}{C_{a_{n}}^{a_{n}-j+1}}||e_{j}||
\leq \frac{1}{C_{a_{n}}-1}\sup_{0\leq j\leq a_{n}}||e_{j}||\leq \frac{\sqrt{C_{a_{n}}}}{C_{a_{n}}-1}$$ if $C_{a_{n}}$ grows sufficiently fast. But by Fact \ref{factd} $||Q_{a_{n}}x||$ tends to $||x||=1$ as $n$ tends to infinity, so this is a contradiction if $C_{a_{n}}$ grows fast enough.
\end{proof}

In order to prove Property (P2), by Property (P1) it suffices to show that if $x$ and $y$ are two vectors of $X$ with $||x||=||y||=1$, either 
$Ty$ belongs to $ \overline{\mathcal{O}\textrm{rb}}(x,T)$ or
$Tx$ belongs to $ \overline{\mathcal{O}\textrm{rb}}(y,T)$. Suppose indeed that we have proved that $
Ty$ belongs to $ \overline{\mathcal{O}\textrm{rb}}(x,T)$.
 Since $y$ can be written as a limit of a certain sequence of vectors $(T^{p_{k}}y)_{k\geq 1}$ with $p_{k}\geq 1$, $\overline{\mathcal{O}\textrm{rb}}(y,T)=\overline{\mathcal{O}\textrm{rb}}(Ty,T)$, and thus we will have shown that $\overline{\mathcal{O}\textrm{rb}}(y,T)\subseteq
\overline{\mathcal{O}\textrm{rb}}(x,T)$.
\par\smallskip

If $j_{n}(x)$ denotes the smallest integer $j\in [0,a_{n}]$ such that
$$|e_{j}^{*(a_n)}(Q_{a_{n}}x)|\geq \frac{1}{C_{a_{n}}^{(a_{n}-j+1)!^{2}}},$$ where $(C_{a_n})_{n\geq 1}$ is a sequence which grows very fast (so fast that in particular Fact \ref{facte} holds true),
then either for infinitely many $n$ the inequality $j_{n}(x)\leq j_{n}(y)$ holds true,  or for infinitely many $n$  we have on the contrary that $j_{n}(y)\leq j_{n}(x)$. Suppose that we are in the first situation, and write $
j_{n}(x)$ as $j_{n}$. Then
$$|e_{j_{n}}^{*(a_n)}(Q_{a_{n}}x)|\geq \frac{1}{C_{a _{n}}^{(a _{n}-j_{n}+1)!^{2}}}
$$ and for every $j<j_{n}$,
$$|e_{j}^{*(a_n)}(Q_{a_{n}}x)|\leq \frac{1}{C_{a _{n}}^{(a _{n}-j+1)!^{2}}}
\;\;\textrm{ and }\;\;
|e_{j}^{*(a_n)}(Q_{a_{n}}y)|\leq \frac{1}{C_{a _{n}}^{(a _{n}-j+1)!^{2}}} \cdot$$

Just as in \cite[Section 5.1]{GR}, we are going to show that for infinitely many $n$ there exists a \pol\ $p_{n}$ of degree less than $a_{n}$ with $|p_{n}|$ bounded by a constant depending only on $a_{n}$ such that $$||p_{n}(T)Q_{a_{n}}x-Q_{a_{n}}y||\leq \frac{3}{a_{n}},$$ for instance. We use the following fact, see \cite[Lemma 5.1]{GR} for the proof, which is exactly the same in our context here:

\begin{fact}\label{factf}
Given any two sequences $(\beta_n)_{n\geq 1}$ and $(M_n)_{n\geq 1}$ with $0<\beta_n\leq M_n$ depending only on $a_n$,  there exists a sequence $(D_{a_{n}})_{n\geq 1}$ (with $D_{a_{n}}$ depending only on $a_{n}$ for each $n$)  such that the following property holds true:

for every vector $x$ of $F_{a_{n}}$ with $\beta_n\leq ||x||\leq M_n$, of the form $$x=\sum_{j=i_{n}}^{a_{n}}e_{j}^{*(a_n)}(x)e_{j},
\textrm{ where the integer } i_{n}\in [0,a_n] \textrm{ is such that }
e_{i_{n}}^{*(a_n)}(x)\not =0,$$ and every vector $y$ with $\beta_n\leq ||y||\leq M_n$ belonging to the linear span of the vectors $e_{i_{n}},\ldots, e_{a_{n}}$, there exists a \pol\ $p$ of degree at most $a_{n}$ with
$$|p|\leq \frac{D_{a_{n}}}{|e_{i_{n}}^{*(a_n)}(x)|^{a_{n}-i_{n}+1}}$$ such that $p(T_{a_{n}})x=y$. Here $T_{a_{n}}$ denotes the shift on the finite-dimensional space $F_{a_{n}}$ \wrt\ the basis $e_{0},\ldots, e_{a_{n}}$ of $F_{a_{n}}$. It is defined as $T_{a_{n}}e_{j}=e_{j+1}$ for every $j<a_{n}$, and $T_{a_{n}}e_{a_{n}}=0$.
\end{fact}

By Fact \ref{factf} applied to the two vectors
$$x'_n=\sum_{j=j_{n}}^{a _{n}}e_{j}^{*(a_n)}(Q_{a_{n}}x)e_{j}
\quad\textrm{ and }\quad y'_n=\sum_{j=j_{n}}^{a _{n}}e_{j}^{*(a_n)}(Q_{a_{n}}y)e_{j}$$ 
(observe that since the $e_j$-coordinates of the vectors $x$ and $y$ for $j<j_n$ are very small, and the projections $\pi_{[0,a_n]}$ tend to the identity in the Strong Operator Topology, the norms of $x'_n$ and $y'_n$ become closer and closer to $1$ as $n$ grows to infinity),
there exists a
polynomial $p_{n}$ of degree less than $a_{n}$ with
$$|p_{n}|\leq D_{a _{n}}\,. \,C_{a_{n}}^{(a_{n}-j_{n}+1)((a_{n}-j_{n}+1)!^{2})}$$ such that
$$p_{n}(T_{a _{n}})\left(\sum_{j=j_{n}}^{a _{n}}e_{j}^{*(a_n)}(Q_{a_{n}}x)e_{j}
\right)=\sum_{j=j_{n}}^{a_{n}}e_{j}^{*(a_n)}(Q_{a_{n}}y)e_{j}.$$
Since
\begin{eqnarray*}
\left|\left|Q_{a_{n}}y-
\sum_{j=j_{n}}^{a_{n}}e_{j}^{*(a_n)}(Q_{a_{n}}y)e_{j}\right|\right|&=&
\left|\left|\sum_{j=0}^{j_{n}-1}e_{j}^{*(a_n)}(Q_{a_{n}}y)e_{j}\right|\right|\\
&\leq&
\sum_{j=0}^{j_{n}-1}|e_{j}^{*(a_n)}(Q_{a_{n}}y)|\; \sup_{0\leq j<j_{n}} ||e_{j}||\\
&\leq& \frac{\sqrt{C_{a_{n}}}}{C_{a _{n}}-1}\leq \frac{1}{a _{n}}
\end{eqnarray*}
if $ \sup_{0\leq j
\leq a_{n}}||e_{j}|| \leq \sqrt{C_{a _{n}}} $ as above and
 $C_{a_{n}}$ grows fast enough, we get
$$\left|\left|p_{n}(T_{a _{n}})\left(\sum_{j=j_{n}}^{a _{n}}e_{j}^{*(a_n)}(Q_{a_{n}}x)e_{j}
\right)-Q_{a_{n}}y\right|\right|\leq \frac{1}{a_{n}}\cdot$$
Then

\begin{eqnarray*}
\left|\left|p_{n}(T_{a _{n}})\left(
\sum_{j=0}^{j_{n}-1}e_{j}^{*(a_n)}(Q_{a_{n}}x)e_{j}\right)\right|\right|&\leq&|p_{n}|\,||T_{a_{n}}||^{a_{n}}
\sum_{j=0}^{j_{n}-1}|e_{j}^{*(a_n)}(Q_{a_{n}}x)| \sqrt{C_{a_{n }}}
\end{eqnarray*}
and this quantity is less than
\begin{eqnarray*}
 {D_{a_{n}}}{C_{a_{n}}^{(a_{n}-j_{n}+1)(a_{n}-j_{n}+1)!^{2}}}||T_{a_{n}}||^{a_{n}}
\sum_{j=0}^{j_{n}-1}\frac {1}{C_{a_{n}}^{(a _{n}-j+1)!^{2}}}
 \sqrt{C_{a _{n}}}
\end{eqnarray*}
which is in turn less than
\begin{eqnarray*}
 {D_{a _{n}}}{C_{a _{n}}^{(a_{n}-j_{n}+1)(a_{n}-j_{n}+1)!^{2}}}||T_{a_{n}}||^{a_{n}}
\frac {2}{C_{a _{n}}^{(a _{n}-j_{n}+2)!^{2}}} \sqrt{C_{a _{n}}}\\
\end{eqnarray*}
and thus less than
\begin{eqnarray*}
 2{D_{a_{n}} \,||T_{a_{n}}||^{a_{n}}}
\frac{\sqrt{C_{a_n}}}{C_{a_{n}}^{((a_{n}-j_{n}+1)!^{2})((a_{n}-j_{n}+2)^{2}-(a_{n}-j_{n}+1))}}
\leq{2D_{a_{n}} \,||T_{a_{n}}||^{a_{n}}}
\frac{\sqrt{C_{a_n}}}{C_{a_{n}}^{2}}\cdot
\end{eqnarray*}

If we choose $C_{a _{n}}$ very large \wrt\ $D_{a _{n}}$ and $||T_{a_{n}}||$, we
can ensure that the quantity on the right-hand side is less than $1/a
_{n}$, and hence that
$$||p_{n}(T_{a_{n}})Q_{a_{n}}x-Q_{a_{n}}y||\leq \frac {2}{
a_{n}}\cdot$$
Now $|p_{n}|$ is controlled by a constant 
which depends only on $a_{n}$, the degree of $p_{n}$ is less than $a_{n}$, and assuming that $b_{n}$ is chosen
very large \wrt\ $a _{n}$ we get that 
\begin{equation}\label{eqnouvelle}
 ||p_{n}(T)Q_{a_{n}}x-Q_{a_{n}}y||\leq \frac {3}{
a _{n}},
\end{equation}
 as $p_{n}(T_{a_n})Q_{a_{n}}x-p_{n}(T)Q_{a_{n}}x $ is supported in the beginning of the \loi\ $[a_n+1, b_n]$.
\par\smallskip
Now the first (b)-working interval is $[b_{n}+1,b_{n}+a_{n}]$, and the definition of $e_{j}$ for $j$ in this interval implies that the following fact holds true:

\begin{fact}\label{factg}
 There exists a constant $C'_{a_{n}}$ depending only on $a_{n}$ such that for every vector $y\in F_{a_{n}}$, $$||(\frac{T^{b_{n}}}{b_{n}}-I)Ty||\leq \frac{C'_{a_{n}}}{b_{n}}||y||.$$
\end{fact}

 Thus if we set $q_{n}(\zeta )=\frac{\zeta ^{b_{n}+1}}{b_{n}}p_{n}(\zeta )$, we have
\begin{eqnarray*}
 ||q_{n}(T)Q_{a_{n}}x-TQ_{a_{n}}y||&\leq&||p_{n}(T)(\frac{T^{b_{n}}}{b_{n}}-I)TQ_{a_{n}}x||+
||p_{n}(T)TQ_{a_{n}}x-TQ_{a_{n}}y||\\
&\leq& |p_{n}|\,2^{a_{n}}\frac{C'_{a_{n}}}{b_{n}}\,||Q_{a_{n}}||\,||x||+\frac{3||T||}{a_{n}}\leq \frac{7}{a_{n}}\
\end{eqnarray*}
 since $||T||\le 2$, $|p_n|\leq D_{a_n}$ which depends only on $a_n$, $||Q_{a_{n}}||\le M$ by Fact \ref{factd}, and $b_n$ is very large \wrt\ $a_n$.
Then, since the degree of $q_{n}$ is less than $a_{n}+b_{n}+1=l_{n}$, there exists a $k\in [1,k_{n}]$ such that $|p_{k,n}-q_{n}|\leq \varepsilon_n$. If $\varepsilon_n$, which is chosen after the construction of the (a)- and (b)-working intervals at step $n$, is so small that $\varepsilon_n< \frac{1}{||Q_{a_{n}}||}4^{-\nu _{n}}$ for instance, we get that
$$||p_{k,n}(T)-q_{n}(T)||<\frac{1}{||Q_{a_{n}}||}4^{-\nu _{n}}2^{a_{n}},$$ so that
$$||p_{k,n}(T)Q_{a_{n}}x-q_{n}(T)Q_{a_{n}}x||<\frac{1}{a_{n}}\cdot$$ Hence
\begin{eqnarray}\label{star}
 ||p_{k,n}(T)Q_{a_{n}}x-TQ_{a_{n}}y||\leq \frac{8}{a_{n}}\cdot
\end{eqnarray}
Thus
\begin{eqnarray}\label{triangle}
 ||T^{c_{k,n}}x-Ty||&\leq&||T^{c_{k,n}}(I-Q_{\nu _{n}})x||+
||(T^{c_{k,n}}-p_{k,n}(T))Q_{\nu _{n}}x||\\
&+&||p_{k,n}(T)(Q_{\nu _{n}}-Q_{a_{n}})x||+||p_{k,n}(T)Q_{a_{n}}x-TQ_{a_{n}y}||\notag\\
&+&||T(I-Q_{a_{n}})y||\notag.
\end{eqnarray}
We know how to control all the terms except the third one, which we now proceed to estimate: we have $(Q_{\nu _{n}}-Q_{a_{n}})x=\pi_{[a_{n}+1,\nu_{n}]}x$, so that
\begin{eqnarray*}
||p_{k,n}(T)(Q_{\nu _{n}}-Q_{a_{n}})x||&\leq& ||p_{k,n}(T)-q_{n}(T)||\,||\pi_{[a_{n}+1,\nu_{n}]}x||\\
&+& ||q_{n}(T)\pi_{[a_{n}+1,\nu_{n}]}x||.
\end{eqnarray*}
Now $|p_{k,n}-q_{n}|<\varepsilon_n$ is so small that the first term in this sum is very small (recall again that $\varepsilon_n$ is chosen after the (a)- and (b)-parts of the construction at step $n$, and that it can be chosen so small as to compensate the norms of $Q_{\nu_n}$ and $Q_{a_n}$, as well as $||T||^{\nu_n}$). 
We can ensure for instance that
\begin{eqnarray}\label{carre}
 ||p_{k,n}(T)-q_{n}(T)||\,||\pi_{[a_{n}+1,\nu_{n}]}x||<\frac{1}{a_{n}}\cdot
\end{eqnarray}
The second term is equal to
\begin{eqnarray*}
||p_{n}(T)\frac{T^{b_{n}+1}}{b_{n}}\pi_{[a_{n}+1,\nu_{n}]}x||&\leq& ||p_{n}(T)||\,||\frac{T^{b_{n}+1}}{b_{n}}\pi_{[a_{n}+1,\nu_{n}]}x||\\
&\leq& D_{a_{n}}{C_{a _{n}}^{(a_{n}-j_{n}+1)(a_{n}-j_{n}+1)!^{2}}}
2^{a_{n}}||\frac{T^{b_{n}+1}}{b_{n}}\pi_{[a_{n}+1,\nu_{n}]}x||.
\end{eqnarray*}
So it remains to estimate the quantity
$||\frac{T^{b_{n}+1}}{b_{n}}\pi_{[a_{n}+1,\nu_{n}]}x||$. But exactly the same proof as in \cite[Lemma 4.9]{GR} shows that $||{T^{b_{n}+1}}\pi_{[a_{n}+1,\nu_{n}]}x||\leq 2||x||$ for every $x\in X$. Hence

\begin{proposition}\label{prop3}
 For every  $n\geq 1$ and every $x\in X$, $$||\frac{T^{b_{n}+1}}{b_{n}}\pi_{[a_{n}+1,\nu_{n}]}x||\leq \frac{2}{b_n}||x||.$$
\end{proposition}

So by taking $b_{n}$ large enough we can make sure that
\begin{eqnarray}\label{coeur}
||p_{n}(T)\frac{T^{b_{n}+1}}{b_{n}}\pi_{[a_{n}+1,\nu_{n}]}x||<\frac{1}{a_{n}}\cdot
\end{eqnarray}

Putting this into (\ref{triangle}),
 we get that for every $n$ there exists a $k\in [1,k_{n}]$ such that 
 \begin{eqnarray}
  ||T^{c_{k,n}}x-Ty||&\leq& 100\,||x-\pi_{[0,\nu _{n}]}x||+\frac{1}{a_{n}} +\frac{1}{a_{n}}
 +\frac{8}{a_{n}}+
 2\,||(I-Q_{a_{n}})y||\\
 &\leq& 100\,||x-\pi_{[0,\nu _{n}]}x||+\frac{10}{a_{n}} +2\,||(I-Q_{a_{n}})y||\notag
 \end{eqnarray}
 where the estimates of the first four terms in (\ref{triangle}) follow respectively from (\ref{EQ2}) in Proposition \ref{prop2}, Fact \ref{factb} with $\delta _{n}$ sufficiently small, (\ref{carre}) and (\ref{star}). 
 For any fixed couple $(x,y)$, the terms $||x-\pi_{[0,\nu _{n}]}x||$ and $||(I-Q_{a_{n}})y||$ (by Fact \ref{factd}) are also very small when $n$ is large. Hence $Ty$  belongs to $\overline{\mathcal{O}\textrm{rb}}(x,T)$. If $j_{n}(y)\le j_{n}(x)$ for infinitely many $n$, we prove in exactly the same way that $Tx\in\overline{\mathcal{O}\textrm{rb}}(y,T)$. Property (P2)
is proved.

\subsection{Proof of Property (P3)}
In order to show that Property (P3) holds true, a key step is to prove the following proposition:

\begin{proposition}\label{prop4}
 For any vector $x$ with $||x||=1$ and any integer $n_{0}\geq 1$, there exists an $n\geq n_{0}$ and a \pol\ of degree less than $a_{n}$, with $|p_{n}|$ bounded by a constant depending only on $a_{n}$, such that
$$||p_{n}(T)Q_{a_{n}}x-e_{a_{n}-1}||\leq \frac{1}{a_{n}}\cdot$$
\end{proposition}

Indeed, supposing for the moment that Proposition \ref{prop4} is proved, the same argument as in the proof of Property (P2) above, starting from (\ref{eqnouvelle}) with $y=Q_{a_{n}}y=e_{a_{n}-1}$, shows that for every $n_{0}\ge 1$ there exist an $n\ge n_{0}$ and a $k\in[1,k_{n}]$ such that 
\begin{eqnarray}\label{rond}
 ||T^{c_{k,n}}x-e_{a_{n}}||\leq \frac{10}{a_{n}}+100 \,||(I-Q_{\nu_n})x||.
\end{eqnarray}
Now by Fact \ref{facta} we have $||e_{a_{n}}-e_{0}||<\varepsilon $, so we eventually obtain that
$$||T^{c_{k,n}}x-e_{0}||<\varepsilon +\frac{10}{a_{n}} +100 \,||(I-Q_{\nu_n})x||.$$ Since this is true for infinitely many $n$, this implies that the distance of the orbit of $x$ to $e_{0}$ is less than $\varepsilon $. So Property (P3) is proved for any vector $x$ with $||x||=1$. In order to show that Property (P3) holds true for any \nz\ vector $x\in X$, we only need to observe that by Property (P1), the closures of the orbits of the two vectors $x$ and $\frac{x}{||x||}$ coincide.
\par\smallskip
So our aim is to prove Proposition \ref{prop4}. For this the strategy is to show that if $x$ is a vector of $X$ with $||x||=1$, there exist infinitely many $n$ such that for some $j\in [0, a_{n}-1]$, $|e_{j}^{*(a_n)}(Q_{a_{n}}x)|$ is not too small. We will then check that Fact \ref{factf} can be applied in order to get a \pol\ ${{p}}_{n}$ with suitable properties such that
$||{{p}}_{n}(T)Q_{a_{n}}x-e_{a_{n}-1}||\leq \frac{1}{a_{n}}$. 
\par\smallskip
So here is the main statement we need in order to prove Proposition \ref{prop4}:

\begin{proposition}\label{prop5}
For every sequence $(A_{n})_{n\geq 1}$ of positive real numbers increasing sufficiently rapidly, the following statement holds true:

for any vector $x$ in $X$ with $||x||=1$ and any integer $n_{0}\geq 1$, there exists an $n\geq n_{0}$ and an integer $j\in [0,a_{n}-1]$ such that
$$|e_{j}^{*(a_n)}(Q_{a_{n}}x)|\geq \frac{1}{A_{n}^{(a_n-j+1)!^2}}\cdot$$
\end{proposition}

Observe that the difference between the statements of Fact \ref{facte} and Proposition \ref{prop5} is that we are looking here for a $j\in [0,a_{n}-1]$ (and not a $j\in [0,a_{n}]$, as in Fact \ref{facte}) such that $|e_{j}^{*(a_n)}(Q_{a_{n}}x)|$ is not too small. The reason for requiring that $j$ belongs to $[0, a_{n}-1]$ is that  the method of proof of Property (P2) combined with the statement of Proposition \ref{prop4} yields then a suitable approximation of the vector $Te_{a_{n}-1}=e_{a_{n}}$ by vectors of the form $T^{c_{k,n}}x$. If we only required  $j$ to belong to $[0,a_{n}]$, we would obtain an approximation of the  vector $Te_{a_{n}}=e_{a_{n}+1}$, which is close to $0$ in norm, and such an approximation is not interesting. The reason we obtain an
approximation of the vector $Te_{a_{n}-1}=e_{a_{n}}$ (and not $e_{a_{n}-1}$ itself) by vectors of the form $T^{c_{k,n}}x$
is 
Fact \ref{factg}, which gives an estimate of $||(\frac{T^{b_{n}}}{b_{n}}-I)Ty||$ (and not of
$||(\frac{T^{b_{n}}}{b_{n}}-I)y||$). Ultimately, this comes from the proof of Proposition \ref{prop3}: one has to move the first (b)-\woi\ by $b_{n}+1$ steps (and not $b_{n}$) in order to get to the next interval, so as to be able to ``damp'' properly these intervals in decreasing their lengths and ensure that Proposition \ref{prop3} holds true.

\begin{proof}[Proof of Proposition \ref{prop5}]
 Let $(A_{n})_{n\geq 1}$ be a quickly increasing sequence (how quickly will be seen later on in the proof). Suppose that $x$ with $||x||=1$ is such that for every $n\geq n_{0}$ and every $j\in [0,a_{n}-1]$, $|e_{j}^{*(a_n)}(Q_{a_{n}}x)|\leq \frac{1}{A_{n}^{(a_n-j+1)!^2}}\cdot$ We have for $x=\sum_{j\geq 0}f_{j}^{*}(x)f_{j}$ $$\pi_{[0,a_{n}]}x=\sum_{j= 0}^{a_{n}}f_{j}^{*}(x)f_{j}=
\sum_{j= 0}^{a_{n}}e_{j}^{*(a_n)}(x)e_{j}$$ and
\begin{eqnarray*}
Q_{a_{n}}x=\sum_{j= 0}^{a_{n}}f_{j}^{*}(x)f_{j}-\frac{1}{\alpha_{n+1}}
f_{a_{n+1}}^{*}(x)e_{a_{n}}
=\sum_{j= 0}^{a_{n}}e_{j}^{*(a_n)}(x)e_{j}-\frac{1}{\alpha _{n+1}}
f_{a_{n+1}}^{*}(x)e_{a_{n}}
\end{eqnarray*}
so that  $$e_{j}^{*(a_n)}(Q_{a_{n}}x)=e_{j}^{*(a_n)}(x)\;
\;  \textrm{ for }\; j\in [0,a_{n}-1]$$
 and $$
e_{a_{n}}^{*(a_n)}(Q_{a_{n}}x)=e_{a_{n}}^{*(a_n)}(x)-\frac{1}{\alpha _{n+1}}f_{a_{n+1}}^{*}(x).$$
If $|e_{j}^{*(a_n)}(Q_{a_{n}}x)|\leq \frac{1}{A_{n}^{(a_n-j+1)!^2}}$  for every $j\in [0,a_{n}-1]$, then
$|e_{j}^{*(a_n)}(x)|\leq \frac{1}{A_{n}^{(a_n-j+1)!^2}}$  for every $j\in [0,a_{n}-1]$ so that in particular
$$||\pi_{[0,a_{n}]}x-e_{a_{n}}^{*(a_n)}(x)e_{a_{n}}||\leq \frac{1}{A_{n}}\,\sum_{j=0}^{a_{n}-1}||e_{j}||.$$ If $A_{n}$ is so large that
$\frac{1}{A_{n}}\,\sum_{j=0}^{a_{n}-1}||e_{j}||\leq \frac{1}{a_{n}}$ for instance, we get that $$||\pi_{[0,a_{n}]}x-e_{a_{n}}^{*(a_n)}(x)e_{a_{n}}||$$ tends to zero as $n$ tends to infinity. Hence $e_{a_{n}}^{*(a_n)}(x)e_{a_{n}}$ tends to $x$ as $n$ tends to infinity.
Since $e_{a_{n}}=e_{0}+\sum_{k=1}^{n}\alpha_{k}z_{\kappa_{k}}$,
 applying the functional $f_{0}^{*}$ to both sides of this equality yields that $e_{a_{n}}^{*(a_n)}(x)$ tends to $f_{0}^{*}(x)=\gamma $. Now $||e_{a_{n}}||\leq 1+\varepsilon $ for every $n$, so we get that $\gamma e_{a_{n}}$ tends to $x$ as $n$ tends to infinity. Since $x$ is of norm $1$, we see that $\gamma$ is \nz, and thus $e_{a_{n}}$ converges to $\frac{1}{\gamma }x$. But $e_{a_{n}}=\alpha _{n}z_{\kappa_{n}}+\ldots+\alpha _{1}z_{\kappa_{1}}+e_{0}$, so this implies that the series $\sum_{j\geq 1}\alpha_{j}z_{\kappa_{j}} $ is convergent, which stands in contradiction with our initial assumption that this series is divergent. This proves Proposition \ref{prop5}.
\end{proof}

\begin{proof}[Proof of Proposition \ref{prop4}]
Let us denote by $j_{n}$ be the smallest integer $j$ in $[0,a_{n}]$ such that
$$|e_{j}^{*(a_n)}(Q_{a_{n}}x)|\geq \frac{1}{A_{n}^{(a_n-j+1)!^2}}\cdot$$ Then $j_{n}\le a_{n}-1$. Also
$$||Q_{a_{n}}x||\ge \frac{1}{A_{n}^{(a_n+1)!^2}}\,\cdot \,\frac{1}{\sup_{j=0,\ldots, j_{n}-1}||e_{j}^{*(a_{n})}||}:=\beta _{n}.$$ Since $||Q_{a_{n}}x||\le M=:M_{n}$
by Fact \ref{factbis}, it is possible to apply Fact \ref{factf} so as to get a \pol\ $p_{n}$ of degree at most $a_{n}$, with $|p_{n}|\le D_{a_{n}}A_{n}^{(a_{n}-j_{n}+1)!^{2}(a_{n}-j_{n}+1)}$ such that
\begin{equation}\label{eqa}
 p_{n}(T_{a_{n}})\left(\sum_{j=j_{n}}^{a_{n}}e_{j}^{*(a_{n})}(Q_{a_{n}}x)e_{j}\right)=e_{a_{n}-1}.
\end{equation}
Now we have
\begin{equation*}
||\sum_{j=0}^{j_{n}-1}e_{j}^{*(a_{n})}(Q_{a_{n}}x))e_{j}||\le\sum_{j=0}^{j_{n}-1} \frac{1}{A_{n}^{(a_n-j+1)!^2}}||e_{j}||\le a_{n}\sup_{j=0,\ldots, a_{n}}||e_{j}||\frac{2}{A_{n}^{(a_n-j_{n}+2)!^2}}\cdot
\end{equation*}
Hence
\begin{eqnarray*}
||p_{n}(T_{a_{n}})(\sum_{j=0}^{j_{n}-1}e_{j}^{*(a_{n})}(Q_{a_{n}}x))e_{j})||&\le& \sup_{j=0,\ldots, a_{n}}||e_{j}||
\, \frac{2a_{n} D_{a_{n}} 2^{a_{n}}}{A_{n}^{(a_{n}-j_{n}+2)!^{2}-(a_{n}-j_{n}+1)!^{2}(a_{n}-j_{n}+1)}}\\
&\le & \frac{D'_{a_{n}}}{A_{n}}
\end{eqnarray*}
for some constant $D'_{a_{n}}$ depending only on $a_{n}$. If $A_{n}$ is large enough with respect to 
$a_{n}$, we can ensure that
\begin{equation}\label{eqb}
||p_{n}(T_{a_{n}})\left(\sum_{j=0}^{j_{n}-1}e_{j}^{*(a_{n})}(Q_{a_{n}}x))e_{j}\right)||<\frac{1}{a_{n}}\cdot
\end{equation}
So combining (\ref{eqa}) and (\ref{eqb}), we obtain that $$||p_{n}(T_{a_{n}})(Q_{a_{n}}x)-e_{a_{n}-1}||<\frac{2}{a_{n}}\cdot$$ A by now standard argument shows then that
$$||p_{n}(T)(Q_{a_{n}}x)-e_{a_{n}-1}||<\frac{3}{a_{n}}\cdot$$ Proposition \ref{prop4} is proved.
\end{proof}
Property (P3) follows from Proposition \ref{prop4}, and this finishes the proof of Theorem \ref{th2}.

\subsection{A fourth Property (P4) of the operators of Theorem \ref{th2}}
The \ops\ of Theorem \ref{th2} enjoy a fourth property, proved already in \cite{GR} (see also \cite{GR2}), which we mention here for completeness'sake:

\begin{proposition}\label{prop6}
The \ops\ $T$ of Theorem \ref{th2} additionally satisfy:
\par\smallskip
(P4) If $M$ is any closed \inv\ subspace of $T$, the \op\ induced by $T$ on $M$ is \hy: there exists an $x\in M$ such that $M=\overline{\mathcal{O}\textrm{rb}}(x,T)$. Consequently, if $M$ and $N$ are two closed \inv\ subspaces of $T$ with $\{0\}\subsetneq N\subsetneq M$, then $N$ is of infinite codimension in $M$ (and of course of infinite dimension).
\end{proposition}

This  implies that the set $HC(T )$ of \hy\ vectors of $T$ has a complement which is contained in a countable union of closed subspaces  of infinite codimension in $X$ (this is the same argument as in \cite[Section 5.2]{GR} and \cite[Prop. $2.5$]{GR2}). So $HC(T)^{c}$ is already extremely small, although possibly still different from $\{0\}$.

\section{Operators without non-trivial invariant closed subsets  on a class of non-reflexive spaces : proof of Theorem \ref{th1}}\label{sec4}

Let $Z$ be a non-reflexive Banach space admitting a Schauder basis. We know that there exists  a  Schauder basis $(z_{j})_{j\ge 1}$ of $Z$ with $\inf_{j\ge 1}||z_{j}||>0$ which has the following property: there exists a positive number $\delta _{0}$ and a bounded sequence $(\alpha _{j})_{j\ge 1}$ of positive real numbers such that
\begin{equation*}
 \limsup_{j\to +\infty}\alpha _{j}>\delta _{0}>0
\end{equation*}
and
\begin{equation}\label{EQ3}
 \sup_{J\ge 1}||\sum_{j=1}^{J}\alpha _{j}z_{j}||\le 1.
\end{equation}
That such a basis $(z_{j})_{j\ge 1}$ does exist follows immediately from Zippin's result \cite{Z}, which states that there exists a normalized Schauder basis $(z_{j})_{j\ge 1}$ of $Z$ and a strictly increasing sequence $(\kappa_{m})_{m\ge 1}$ of integers such that
$\sup_{M\ge 1}||\sum_{m=1}^{M}z_{\kappa_{m}}||$ is finite. If $\alpha _{j}$ is defined for $j$ not belonging to the set $\{\kappa_{m} \textrm{ ; } m\ge 1\}$ as a positive number so small that the series $\sum_{j\not\in\{\kappa_{m}\}}\alpha _{j}$ is convergent, and as $\alpha _{\kappa_{m}}=1$ for each $m\ge 1$, then 
$\sup_{J \ge 1}||\sum_{j=1}^{J}\alpha _{j}z_{j}||$ is finite (and of course $0<\alpha_j\le 1$ for each $j\ge 1$). It suffices then to normalize the basis to get the statement above.
\par\smallskip
The space on which we are going to construct \ops\ without \nt\ \inv\ closed subsets is of the form $X=\bigoplus_{\ell_{p}}Z$ or $X=\bigoplus_{c_{0}}Z$, $1\le p<+\infty$. Pulling out a copy of either $\ell_{p}$ or $c_{0}$, we can suppose that $X=\ell_{p}\oplus_{\ell_{p}}\bigoplus_{\ell_{p}}Z$, $1\le p<+\infty$, or $X=c_{0}\oplus_{\ell_{\infty}}\bigoplus_{c_{0}}Z$. Let $(g_{j})_{j\ge 0}$ denote the canonical basis of $\ell_{p}$ or $c_{0}$. For each integer $d\ge 1$, we denote by $Z^{(d)}$ the $d^{th}$ copy of the space $Z$ which appears in the infinite direct sum $\bigoplus_{\ell_{p}}Z$ or $\bigoplus_{c_{0}}Z$, and by $(z_{j}^{(d)})_{j\ge 1}$ the basis $(z_{j})_{j\ge 1}$ in this $d^{th}$ copy of $Z$.

\subsection{Definition and boundedness of the operator $T$}
We keep the notation of Section \ref{sec2}: at step $n$, vectors $f_{j}$ and $e_{j}$ will be constructed for $j\in [\xi  _{n}+1,\xi  _{n+1}]$. The (b)- and (c)-parts of the construction remain the same in spirit (although some technical modifications have to be done in order to cope with the changes in the (a)-part), and the important modification of the proof of Theorem \ref{th2} concerns the (a)-part: recall that in the construction of Section \ref{sec2} there was at each step $n$ only one (a)-working interval, which was of the form $[a_{n}-(\kappa_{n}-\kappa_{n-1}-1), a_{n}]$. We will now need several (a)-working intervals, the number of which increases with $n$. This idea is not new, and is present in all the constructions of Read, but in order to be able to implement it in this new setting one needs to combine it with ideas from \cite{GR} concerning the construction of the (c)-part of such \ops\ outside the $\ell_{1}$-setting. More details about this will be given later on.
\par\smallskip
Let us now fix some notation: there will be $n$ (a)-working intervals, which we denote by $J_{n,r}=[ra_{n},ra_{n}+\xi_{n+1-r}]$, $r=1,\ldots,n$. 
We set $\xi  _{0}=\xi  _{1}=0$. Thus $J_{n,1}$ has length $\xi  _{n}+1$, $J_{n,2}$ has length $\xi  _{n-1}+1$, etc., until $J_{n,n}$ which is the singleton $\{na_{n}\}$. In order that the forthcoming definitions of $f_{j}$ for $j$ in the (a)- and (b)-\woi s make sense, we need to set $a_{0}=b_{0}=1$. This is not coherent with the fact that $\xi  _{n}+1\le a_{n}\le b_{n}\le \xi  _{n+1}$ for each 
$n\ge 1$, but this convention will allow us to use simpler notation and to avoid treating separately some cases. We start the real construction at the index 
$n=1$, and choose $a_{1}$ very much larger than $\xi  _{1}$, etc.
\par\smallskip
Let $J_{0}$ be the set $J_{0}=[1,+\infty[\setminus\bigcup_{n\ge 1}\bigcup_{r=1}^{n}J_{n,r}$, and let $\sigma  $ be the unique increasing bijection from $J_{0}$ onto the set $[1,+\infty[$.
\par\smallskip
$\bullet$ \textbf{Definition of the vectors $f_{j}$:} Let us set $f_{0}=g_{0}$, and $f_{j}=g_{\sigma  (j)}$ for every $j$ belonging to $J_{0}$. Let $(d_{n})_{n\ge 1}$ be the sequence defined by $d_{1}=1$ and $d_{n+1}=d_{n}+\xi  _{n}+1$ for each $n\ge 1$. For $j\in [ra_{n},ra_{n}+\xi_{n+1-r}]$, $r=1\ldots n$, let us set
$$f_{j}=z_{r}^{(d_{n-r+1}+j-ra_{n})}.$$ In other words $f_{ra_{n}}=z_{r}^{(d_{n-r+1})}$, which belongs to the $d_{n-r+1}^{th} $ copy of $Z$. When defining the vectors $f_{ra_{n}+l}$, we shift the vector $z_{r}^{(d_{n-r+1})}$ to the $l^{th}$ next copy of $Z$, which is $Z^{(d_{n-r+1}+l)}$.
\par\smallskip
Thus at step $n$ of the construction, the vectors $z_{r}^{(d)}$ appear in the sequence $(f_{j})_{j\ge 0}$ for the following values of $r$ and $d$:  $1=d_{1}\le d<d_{2}$ and $1\le r\le n$, then  $d_{2}\le d<d_{3}$ and $1\le r\le n-1$, etc... until $d_{n}\le d<d_{n+1}$ and $r=1$. More precisely: if $m\ge 1$
 and $d_{m}\le d<d_{m+1}$, then
 $$z_{r}^{(d)}=f_{ra_{m+r-1}+d-d_{m}}\quad \textrm{for any } r\ge 1.$$ Hence all vectors $z_{r}^{(d)}$ eventually appear in the sequence $(f_{j})_{j\ge 0}$. Moreover, since we do not change the order of the vectors $z_{r}$ within a fixed copy $Z^{(d)}$ of $Z$, and $X=\ell_{p}\oplus\bigoplus_{\ell_p}Z$ or $X=c_{0}\oplus\bigoplus_{c_0}Z$, it is not difficult to see that the sequence $(f_{j})_{j\ge 0}$ thus defined is a Schauder basis of $X$. We will denote as usual for $N\ge 1$ by $\pi_{[0,N]}$ the canonical projection onto the first $N+1$ vectors of the basis $(f_{j})_{j\ge 0}$. Then $\sup_{N\ge 1}||\pi_{[0,N]}||$ is finite.
 \par\smallskip
 If $1\le p<+\infty$, we denote by $P_{0}$ the projection of $X$ onto the isolated $\ell_{p}$ space, and, for each $d\ge 1$, by $P_{Z^{(d)}}$ the projection of $X$ onto the $d^{th}$ copy $Z^{(d)}$ of $Z$. We have then that for every $x\in X$,
 \begin{equation}\label{eqcarre}
  ||x||=\left(||P_{0}x||^{p}+\sum_{d\ge 1}||P_{Z^{(d)}}x||^{p}\right)^{\frac{1}{p}}.
 \end{equation}
 A similar formula holds true in the case where $X=c_{0}\oplus\bigoplus_{c_{0}}Z$.
 \par\smallskip
 Just as in the proof of Theorem \ref{th2}, we will now define a sequence $(e_{j})_{j\ge 0}$ of vectors such that $e_{0}=f_{0}$ and $\textrm{sp}[e_{0},\ldots, e_{j}]=\textrm{sp}[f_{0},\ldots, f_{j}]$ for each $j\ge 1$. Then the \op\ $T$ on $X$ will be defined as usual by setting $Te_{j}=e_{j+1}$ for each $j\ge 0$.
\par\smallskip
$\bullet$ \textbf{Definition of the vectors $e_{j}$ for $j$ in an (a)-working interval:} At step $n\ge 1$, $a_{n}$ is chosen extremely large \wrt\ $\xi  _{n}$. For each $r\in [1,n]$ and each $j$ belonging to the interval $ J_{n,r}=[ra_{n},ra_{n}+\xi  _{n+1-r}]$, we define $e_{j}$ by the relation
$$f_{j}=\frac{a_{n-r}}{\alpha _{r}}(e_{j}-e_{j-ra_{n}+(r-1)a_{n-1}}).$$ We mention here already the analogue of Fact \ref{facta}, which will be crucial in the sequel of the proof in order to prove that $e_{0}$ belongs to the closure of the orbit of any non-zero vector $x$ of $X$.

\begin{fact}\label{fact1bis}
 For any integers $n\ge 1$ and $N\in [0,n-1]$, we have
 \begin{equation}\label{new1}
  e_{(n-N)a_{n}}=\frac{1}{a_{N}}\sum_{k=1}^{n-N}\alpha _{k}z_{k}^{(d_{N+1})}+e_{0},
 \end{equation}
 from which it follows that 
 \begin{equation}\label{new2}
   ||e_{(n-N)a_{n}}-e_{0}||\le\frac{1}{a_{N}}\cdot
 \end{equation}
\end{fact}

\begin{proof}
By definition of $e_{ra_{n}}$ we have
\begin{eqnarray*}
 e_{ra_{n}}&=&\frac{\alpha _{r}}{a_{n-r}}f_{ra_{n}}+e_{(r-1)a_{n-1}}=
\frac{\alpha _{r}}{a_{n-r}}f_{ra_{n}}+\frac{\alpha _{r-1}}{a_{n-r}}f_{(r-1)a_{n-1}}+\ldots+\frac{\alpha _{1}}{a_{n-r}}f_{a_{n-r+1}}+e_{0}\\
&=&\frac{1}{a_{n-r}}\left(\alpha _{r}z_{r}^{(d_{n-r-1})}+\alpha _{r-1}z_{r-1}^{(d_{n-r-1})}+\ldots+\alpha _{1}z_{1}^{(d_{n-r-1})}\right)+e_{0}.
\end{eqnarray*}
Applying this with $r=n-N$ (which belongs to the set $[1,n]$) yields the expression of $e_{(n-N)a_{n}}$ in (\ref{new1}). The estimate (\ref{new2}) is an immediate consequence of our assumption (\ref{EQ3}) that
$$||\sum_{k=1}^{r}\alpha _{k}z_{k}||=||\sum_{k=1}^{r}\alpha _{k}z_{k}^{(d)}||\le 1\quad
\textrm{for each } r\ge 1 \textrm{ and } d\ge 1.$$
\end{proof}

\par\smallskip
$\bullet$ \textbf{Definition of the vectors $e_{j}$ for $j$ in a (b)-working interval:} 
The (b)-working intervals are the same in spirit as in the proof of Theorem \ref{th2}, but their definition needs to be adjusted to the introduction of $n$ (a)-\woi s instead of just one. Let us set $\mu _{n}=na_{n}$, which is the index corresponding to the end of the last (a)-\woi. The (b)-working intervals are the intervals $[r(b_{n}+1), rb_{n}+\mu_{n}]$, $r=1,\ldots, \mu _{n}$, where 
$b_{n}$ is extremely large \wrt\ $a_{n}$. The first of these intervals has length $\mu _{n}$, the second length $\mu _{n}-1$, etc... until the very last one which is just the singleton $\{\mu _{n}(b_{n}+1)\}$. We denote by $\nu _{n}=\mu _{n}(b_{n}+1)$
this last element. For $j$ belonging to one of these intervals $[r(b_{n}+1), rb_{n}+\mu_{n}]$, $r=1,\ldots, \mu _{n}$, $e_{j}$ is defined as in the construction of Section \ref{sec2} by the relation
$$f_{j}=e_{j}-b_{n}e_{j-b_{n}}.$$

\par\smallskip
$\bullet$ \textbf{Definition of the vectors $e_{j}$ for $j$ in a (c)-working interval:} 
The definition of the (c)-\woi s is exactly the same as in Section \ref{sec2}, except that the value of $\nu _{n}$ has changed: $\nu _{n}=\mu_{n}(b_{n}+1)$, which is the index corresponding to the end of the last (b)-\woi. The sequence $(p_{k,n})_{1\le k\le k_{n}}$ of \pol s forms an $\varepsilon _{n}$-net of the closed ball of $\K_{\nu _{n}}[\zeta ]$ of radius $2$ for the norm $|\,.\,|$, and $(c_{k,n})_{1\le k\le k_{n}}$ is a sequence of integers with $ c_{1,n}\ll c_{2,n}
\ll \ldots \ll c_{k_{n},n}$ for each $n\ge 1$. The (c)-intervals are the intervals
$ I_{s_{1},\ldots , s_{k_{n}}}=[s_{1}c_{1,n}+\ldots+s_{k_{n}}c_{k_{n},n}, s_{1}c_{1,n}+\ldots+s_{k_{n}}c_{k_{n},n}+\nu _{n}]$ where $0\le s_{j}\le h_{n}$ with 
$|s|=s_{1}+\ldots +s_{k_{n}}\ge 1$. The integer $h_{n}$ is chosen in the same way as in Section \ref{sec2}. If $t$ is the largest integer
such that $s_{t} \ge 1$, $e_{j}$ is defined for $j\in I_{s_{1},\ldots , s_{k_{n}}}$ by the relation
$$f_{j}=\frac {4^{1-|s|}}{\gamma _{n}}(e_{j}-p_{t,n}(T)e_{j-c_{t,n}}),$$
 where $\gamma _{n}$
is an extremely small positive number depending only on $\nu _{n}$.

\par\smallskip
$\bullet$ \textbf{Definition of the vectors $e_{j}$ for $j$ in a \loi :} 
The definitions are exactly the same as usual: if j belongs to a \loi\ $[k+1,k+l]$,
$$f_{j}=2^{\frac{1}{\sqrt{l}}(\frac{1}{2}l+k+1-j)}e_{j}$$ except in the case where the \loi\ is either the interval $[na_{n}+1,b_{n}]$ or one of the intervals $[rb_{n}+\mu_{n}+1, (r+1)b_{n}-1] $, $r=1,\ldots, n-1$, between two (b)-working intervals. For $j\in [na_{n}+1,b_{n}]$, we set
$$f_{j}=2^{\frac{1}{\sqrt{b_{n}}}(\frac{1}{2}b_{n}+na_{n}+1-j)}e_{j},$$ and for $j\in [rb_{n}+\mu_{n}+1, (r+1)b_{n}-1] $, $r=1,\ldots, n-1$, we set
$$f_{j}=2^{\frac{1}{\sqrt{b_{n}}}(\frac{1}{2}b_{n}+rb_{n}+\mu _{n}+1-j)}e_{j}.$$
Then as usual $\xi  _{n+1}$ is chosen extremely large \wrt\ the index which corresponds to the end of the last (c)-\woi.

\par\smallskip
$\bullet$ \textbf{Boundedness of the \op\ $T$:} 
The proof of the fact that $T$ is bounded is extremely similar to that of Proposition \ref{prop1}, and we will only outline the cases which are different between the two proofs. The most important difference concerns the (a)-part of the \op.

$\bullet$ if $j\in [ra_{n},ra_{n}+\xi  _{n+1-r}-1]$, $r=1,\ldots,n$, we have $Tf_{j}=f_{j+1}$. Recalling that $f_{j}=z_{r}^{(d_{n-r+1}+j-ra_{n})}$ for such values of $j$, this means that $T$ maps 
$z_{r}^{(d_{n-r+1}+j-ra_{n})}$ to $z_{r}^{(d_{n-r+1}+j+1-ra_{n})}$, which is the same vector $z_{r}$ but in the next copy of $Z$.

$\bullet$ we now have to consider the case where $j$ is a right endpoint of one of the (a)-intervals $J_{n,r}$:
if $j=ra_{n}+\xi  _{n+1-r}$, $r=1,\ldots,n$, then
$$Tf_{ra_{n}+\xi  _{n+1-r}}=\frac{a_{n-r}}{\alpha _{r}}\left(e_{ra_{n}+\xi  _{n+1-r}+1}-e_{(r-1)a_{n-1}+\xi  _{n+1-r}+1}\right)$$ so that
$$||Tf_{ra_{n}+\xi  _{n+1-r}}||\ls \frac{a_{n-r}}{\alpha _{r}}\left(2^{-\frac{1}{2}\sqrt{a_{n}}}+2^{-\frac{1}{2}\sqrt{a_{n-1}}}\right)\ls a_{n-1}\,2^{-\frac{1}{4}\sqrt{a_{n-1}}}$$ which is very small if $a_{n-1}$ is sufficiently small at step $n-1$. Indeed if $r\in [2,n]$, the index
$(r-1)a_{n-1}+\xi  _{n+1-r}+1$ is in the beginning of a \loi\ whose length is at least roughly $a_{n-1}$ ($b_{n-1}$ in the case where $r=n$), and if $r=1$ we have $$||e_{\xi  _{n}+1}||\ls 2^{-\frac{1}{2}\sqrt{a_{n}}}\le 2^{-\frac{1}{2}\sqrt{a_{n-1}}}.$$
\par\smallskip
This can be rewritten in the following way:

$\bullet$ if $d\ge 1$ is not of the form $d=d_{m}-1$ for some $m\ge 2$, then $$Tz_{r}^{(d)}=z_{r}^{(d+1)}\quad \textrm{ for each } r\ge 1;$$

$\bullet$ if $d=d_{m}-1$ for some $m\ge 2$, then 
$z_{r}^{(d_{m}-1)}=f_{ra_{m+r-2}+\xi  _{m-1}}$, so that
$$||Tz_{r}^{(d)}||\ls\frac{a_{m-2}}{\alpha _{r}}\,2^{-\frac{1}{2}\sqrt{a_{m+r-3}}}\le
\frac{a_{m+r-3}}{\alpha _{r}}\,2^{-\frac{1}{2}\sqrt{a_{m+r-3}}}
\quad\textrm{ for each }
r\ge 1$$ if $m\ge 3$, while
$$||Tz_{r}^{(1)}||\ls\frac{1}{\alpha _{r}}(2^{-\frac{1}{2}\sqrt{b_{r}}}+2^{-\frac{1}{2}\sqrt{b_{r-1}}})\quad \textrm{ for each } r\ge 1$$
if $m=2$
(recall that by convention $b_{0}=1$).
\par\smallskip
It follows from these estimates that for any fixed positive number $\rho$ we can ensure, by taking at each step $n$ the quantities $a_{n}$ and $b_{n}$ sufficiently large, that
$$\sum_{m\ge 2}\sum_{r\ge 1}||Tz_{r}^{(d_{m}-1)}||<\rho .$$
That $T$ is bounded is now clear: $T$ acts on $\bigoplus_{\ell_{p}}Z^{(d)}$ (or $\bigoplus_{c_{0}}Z^{(d)}$) as a shift of $Z^{(d)}$ to the next copy $Z^{(d+1)}$, except in the case where $d=d_{m}-1$ for some $m\ge 2$, in which case $T$ crushes the unit ball of $Z^{(d)}$ to something very small in norm. So for any $x\in X$ we have
$$||T(\sum_{j\not \in J_{0}}f_{j}^{*}(x)f_{j})||\le (1+\rho )||x||.$$ Since the same estimates as in Section \ref{sec2} show that $$||T(\sum_{j\in J_{0}}f_{j}^{*}(x)f_{j})||\le M||x||$$ for some positive constant $M$, it follows that $T$ is bounded as required.

\subsection{Proof of Theorem \ref{th1}} We follow the pattern of the proof of Theorem \ref{th2}, and we will see which modifications need to be made in order to show that all non-zero vectors $x$ of $X$ are \hy\ for $T$.

For each integer $j\ge 0$, let $F_{j}=\textrm{sp}[f_{0},\ldots,f_{j}]=\textrm{sp}[e_{0},\ldots, e_{j}]$. First observe that Fact \ref{factb} still holds true in this context:

\begin{fact}\label{factbbis}
 For each $n\ge 1$, fix $\delta _{n}$ to be a very small positive number. If at each step $n$ the quantity $\gamma _{n}$ is chosen small enough in the construction of the (c)-\woi s, then
 $$||T^{c_{k,n}}y-p_{k,n}(T)y||\le \delta _{n}||y||\quad \textrm{for all }y\in F_{\nu _{n}} \textrm{ and all }k=1,\ldots,k_{n}.$$
\end{fact}

As a consequence, we obtain already that $e_{0}$ is a \hy\ vector for $T$.
\par\smallskip
In order to obtain the crucial tail estimates for the norm of $T^{c_{k,n}}$, we need, as in the proof of Theorem \ref{th2}, to introduce a projection $Q_{\nu _{n}}$ on $F_{\nu _{n}}$. As previously, the role of $Q_{\nu _{n}}$ is to take away from $[\nu _{n}+1,+\infty[$ the integers $j$ for which the vector $T^{c_{k,n}}f_{j}$ is not manageable. It is defined as follows:

\begin{equation*}
Q_{\nu _{n}}f_{j}=\begin{cases}
f_{j} &\textrm{ if } j\leq \nu _{n}\\
-\frac {a_{n+1-r}}{\alpha _{r}}e_{j-ra_{n+1}+(r-1)a_{n}} & \textrm{ if } j\in [ra_{n+1},ra_{n+1}+\xi  _{n+2-r}] \textrm{, }r=2,\ldots, n+1\\
0 & \textrm{ in all the other cases.}
\end{cases}
\end{equation*}

Hence

\begin{equation*}
(I-Q_{\nu _{n}})f_{j}=\begin{cases}
0 &\textrm{ if } j\leq \nu _{n}\\
\frac {a_{n+1-r}}{\alpha _{r}}e_{j} & \textrm{ if } j\in [ra_{n+1},ra_{n+1}+\xi  _{n+2-r}] \textrm{, }r=2,\ldots, n+1\\
f_{j} & \textrm{ in all the other cases.}
\end{cases}
\end{equation*}

Since $j-ra_{n+1}+(r-1)a_{n}$ belongs to the interval $[(r-1)a_{n},(r-1)a_{n}+\xi  _{n+2-r}]$
which is contained in $[0,na_{n}]=[0,\mu _{n}]$, hence in $[0,\nu _{n}]$, for each $r=2,\ldots, n+1$, $Q_{\nu _{n}}$ is a projection of $X$ onto $F_{\nu _{n}}$ (remark that it is important here that $r$ be larger than $2$). In a similar way we define $Q_{\mu _{n}}$ as 

\begin{equation*}
Q_{\mu _{n}}f_{j}=\begin{cases}
f_{j} &\textrm{ if } j\leq \mu _{n}\\
-\frac {a_{n+1-r}}{\alpha _{r}}e_{j-ra_{n+1}+(r-1)a_{n}} & \textrm{ if } j\in [ra_{n+1},ra_{n+1}+\xi  _{n+2-r}] \textrm{, }r=2,\ldots, n+1\\
0 & \textrm{ in all the other cases.}
\end{cases}
\end{equation*}

Then we see as above that $Q_{\mu _{n}}$ is a projection onto $F_{\mu_{n}}$, and $Q_{\nu _{n}}=Q_{\mu _{n}}+\pi_{[\mu _{n}+1,\nu _{n}]}$. We first need an estimate on the norms of $Q_{\mu _{n}}$ and $Q_{\nu _{n}}$:

\begin{fact}\label{factbisbis}
There exists for each $n\ge 1$ a  constant $M_{a_{n}}$ depending only on $a_{n}$ such that $||Q_{\nu _{n}}||\le M_{a_{n}}$ and $||Q_{\mu _{n}}||\le M_{a_{n}}$.
\end{fact}

\begin{proof}
For every vector $x\in X$ we have
$$Q_{\nu _{n}}x=\pi_{[0,\nu _{n}]}x-\sum_{r=2}^{n+1}\;\sum_{j=ra_{n+1}}^{ra_{n+1}+\xi  _{n+2-r}}
\;f_{j}^{*}(x)\,\frac{a_{n+1-r}}{\alpha _{r}}\,e_{j-ra_{n+1}+(r-1)a_{n}}.$$ For $j\in[ra_{n+1},ra_{n+1}+\xi  _{n+2-r}]$, $||f_{j}^{*}||\le\sup_{k\ge 1}||z_{k}^{*}||_{Z^{*}}$. Also, there exists a positive constant $M$ such that $\sup_{n\ge 1}||\pi_{[0,\nu _{n}]}x||\le M||x||$ for every vector $x\in X$. Renaming constants if necessary, we have
\begin{eqnarray*}
 ||Q_{\nu _{n}}x||&\le & M||x||\left(1+a_{n}\,\frac{n\xi  _{n}}{\inf_{r=2,\ldots, n+1}\alpha _{r}}\sup_{j=0,\ldots,na_{n}}||e_{j}||\right)\le M_{a_{n}}||x||
\end{eqnarray*}
where $M_{a_{n}}$ depends only on $a_{n}$. The same kind of estimate holds true for the norm of $Q_{\mu _{n}}$.
\end{proof}

Thanks to the projections $Q_{\nu _{n}}$, it is possible to obtain, as in Proposition \ref{prop2}, tail estimates for the \ops\ $T^{c_{k,n}}$, $1\le k\le k_{n}$.

\begin{proposition}\label{prop3bis}
 It is possible to ensure that for every $n\ge 1$, every $k\in [1,k_{n}]$ and every $x\in X$,
 $$||T^{c_{k,n}}(I-Q_{\nu _{n}})x||\le 103\,||x||.$$
\end{proposition}

\begin{proof}
Just as usual, we have to compute exactly or estimate the size of the vectors $T^{c_{k,n}}(I-Q_{\nu _{n}})f_{j}$ for $j>\nu _{n}$. As the estimates are rather intricate, we split them into several cases.
\par\smallskip
\textbf{Case 1: $j\in [\nu _{n}+1,a_{n+1}-c_{k,n}-1]$}.
The computations and estimates of $T^{c_{k,n}}f_{j}$ in this case are the same as usual, and we refer the reader to the proof of \cite[Proposition 2.2]{GR} for these.
\par\smallskip
\textbf{Case 2: $j\in [a_{n+1}-c_{k,n}, a_{n+1}-1]$}.
In this case we have $T^{c_{k,n}}f_{j}=\lambda _{j}e_{j+c_{k,n}}$, where $\lambda _{j}\ls 2^{-\frac{1}{2}\sqrt{a_{n+1}}}$. Since $j+c_{k,n}$ belongs to the interval $[a_{n+1}, a_{n+1}+\xi  _{n+1}]$, we have $||e_{j+c_{k,n}}||\ls \sup_{l=0 \ldots \xi _{n+1}}||e_{l}||$, from which it follows that $$||T^{c_{k,n}}f_{j}||\ls
2^{-\frac{1}{4}\sqrt{a_{n+1}}}$$ if $a_{n+1}$ is large enough.
\par\smallskip
\textbf{Case 3: $j\in [a_{n+1},a_{n+1}+\xi  _{n+1}]$}.
Then $(I-Q_{\nu _{n}})f_{j}=f_{j}$ so we have to estimate $||T^{c_{k,n}}f_{j}||$. There are two subcases to consider.
\par\smallskip
-- \textbf{Case 3a: $j+c_{k,n}\in [a_{n+1},a_{n+1}+\xi  _{n+1}]$}.
This is the case except when $j$ lies in the end of the interval $[a_{n+1},a_{n+1}+\xi  _{n+1}]$. In this case $T^{c_{k,n}}f_{j}=f_{j+c_{k,n}}$. 
\par\smallskip
-- \textbf{Case 3b: $j+c_{k,n}>a_{n+1}+\xi  _{n+1}$}. 
Then $$T^{c_{k,n}}f_{j}=\frac{a_{n}}{\alpha _{1}}(e_{j+c_{k,n}}-e_{j-a_{n+1}+c_{k,n}}).$$ Since $j+c_{k,n}$ lies in the beginning of the \loi\ $[a_{n+1}+\xi  _{n+1}+1,2a_{n+1}-1]$, $$||e_{j+c_{k,n}}||\ls 2^{-\frac{1}{2}\sqrt{a_{n+1}}}.$$ Also, $\xi  _{n+1}< j+c_{k,n}-a_{n+1}\le \xi  _{n+1}+c_{k,n}$, so the index $j+c_{k,n}-a_{n+1}$ lies in the beginning of the \loi\ $[\xi  _{n+1}+1,a_{n+1}-1]$, and $||e_{j-a_{n+1}+c_{k,n}}||\ls 2^{-\frac{1}{2}\sqrt{a_{n+1}}}$, so that $$||T^{c_{k,n}}f_{j}||\ls 2^{-\frac{1}{4}\sqrt{a_{n+1}}}$$ as before.
\par\smallskip
\textbf{Case 4:} $j$ belongs to one of the \loi s $[ra_{n+1}+\xi  _{n+1-r}+1, (r+1)a_{n+1}-1]$ for some $r=1,\ldots,n$.
There are again two different subcases to consider here.
\par\smallskip
-- \textbf{Case 4a: $j+c_{k,n}\in [ra_{n+1}+\xi  _{n+1-r}-1, (r+1)a_{n+1}-1]$}.
Then $$T^{c_{k,n}}f_{j}=2^{\frac{c_{k,n}}{\sqrt{a_{n+1}-\xi  _{n+1-r}+1}}}f_{j+c_{k,n}},$$ and the multiplicative coefficient can be made arbitrarily close to $1$
 if $a_{n+1}$ is large enough. 
 \par\smallskip
-- \textbf{Case 4b: $j+c_{k,n}> (r+1)a_{n+1}-1$}. 
Then $j$ belongs to the interval $[(r+1)a_{n+1}-c_{k,n}, (r+1)a_{n+1}-1]$, which is the end of the \loi\ $[ra_{n+1}+\xi  _{n+1-r}+1, (r+1)a_{n+1}-1]$. Hence
 $$||T^{c_{k,n}}f_{j}||=\lambda_{j}||e_{j+c_{k,n}}||\ls 2^{-\frac{1}{2}\sqrt{a_{n+1}}}||e_{j+c_{k,n}}||.$$ Since $j+c_{k,n}$ belongs to the \woi\ $[(r+1)a_{n+1},(r+1)a_{n+1}+\xi  _{n+1-r}]$, we have $$e_{j+c_{k,n}}=\frac{\alpha _{r+1}}{a_{n}}f_{j+c_{k,n}}+e_{j+c_{k,n}-(r+1)a_{n+1}+ra_{n}}.$$ Thus
 $||e_{j+c_{k,n}}||\ls \sup_{l=0 \ldots \mu  _{n}}||e_{l}||$, from which it follows that if $a_{n+1}$ is sufficiently large,
 $$||T^{c_{k,n}}f_{j}||\ls 2^{-\frac{1}{4}\sqrt{a_{n+1}}}$$ for instance.
 \par\smallskip
\textbf{Case 5:} $j$ belongs to one of the (a)-\woi s $[ra_{n+1},ra_{n+1}+\xi  _{n+2-r}]$ for some $r\in [2,n+1]$. Then  $$T^{c_{k,n}}(I-Q_{\nu _{n}})f_{j}=\frac{a_{n+1-r}}{\alpha _{r}}e_{j+c_{k,n}}.$$ Observe now that $j+c_{k,n}\ge ra_{n+1}+c_{k,n}>ra_{n+1}+\xi  _{n+2-r}$, because $r\ge 2$ and $\xi  _{n+2-r}\le\xi  _{n}$ (here again, the fact that $r$ is larger than $2$ is important). Hence $j+c_{k,n}$ lies in the beginning of a \loi\ of length at least $a_{n+1}$, and so $$||T^{c_{k,n}}(I-Q_{\nu _{n}})f_{j}||\ls 2^{-\frac{1}{2}\sqrt{a_{n+1}}}.$$
\par\smallskip
\textbf{Case 6:} $j$ belongs to a \loi\ or to a (b)- or (c)-\woi\ at steps $n+1, n+2$, etc. Then the same computations and estimates as in the proof of \cite[Proposition 2.2]{GR} apply.
\par\smallskip
\textbf{Case 7:} The only cases which remain to be considered are those when $j$ belongs to an (a)-\woi\ of the form $[ra_{n+q},ra_{n+q}+\xi  _{n+q+1-r}]$, $r=1,\ldots, n+q$, with $q\ge 2$. Then $$T^{c_{k,n}}(I-Q_{\nu _{n}})f_{j}=T^{c_{k,n}}f_{j}=\frac{a_{n+q-r}}{\alpha _{r}}(e_{j+c_{k,n}}-e_{j-ra_{n+q}+(r-1)a_{n+q-1}+c_{k,n}}).$$ There we split again  the argument into  two separate subcases:
\par\smallskip
-- \textbf{Case 7a:} the interval $[ra_{n+q},ra_{n+q}+\xi  _{n+q+1-r}]$ is ``long'' compared to $c_{k,n}$, i.e. $q-r\ge 0$. In this case if $j+c_{k,n}$ belongs to $[ra_{n+q},ra_{n+q}+\xi  _{n+q+1-r}]$ (which is what happens except when $j$ lies in the end of the interval), then $$T^{c_{k,n}}f_{j}=f_{j+c_{k,n}}.$$ If $j+c_{k,n}>ra_{n+q}+\xi  _{n+q+1-r}$, the same kind of argument as previously shows that
$$||e_{j+c_{k,n}}||\ls 2^{-\frac{1}{2}\sqrt{a_{n+q}}}.$$ Also, $j-ra_{n+q}+(r-1)a_{n+q-1}+c_{k,n}>(r-1)a_{n+q-1}+\xi_{n+q+1-r}$, so that
$$||e_{j-ra_{n+q}+(r-1)a_{n+q-1}+c_{k,n}}||\ls 2^{-\frac{1}{2}\sqrt{a_{n+q-1}}}.$$ It follows that 
 $$||T^{c_{k,n}}f_{j}||\ls 2^{-\frac{1}{4}\sqrt{a_{n+q-1}}};$$
\par\smallskip
-- \textbf{Case 7b:} the interval $[ra_{n+q},ra_{n+q}+\xi  _{n+q+1-r}]$ is ``short'' compared to $c_{k,n}$, i.e. $q-r<0$. Then $j+c_{k,n}$ always lies in the beginning of the \loi\  which lies after $[ra_{n+q},ra_{n+q}+\xi  _{n+q+1-r}]$, so that $$||e_{j+c_{k,n}}||\ls 2^{-\frac{1}{2}\sqrt{a_{n+q}}}.$$
Lastly, $j-ra_{n+q}+(r-1)a_{n+q-1}$ belongs to $[(r-1)a_{n+q-1},(r-1)a_{n+q-1}+\xi  _{n+q+1-r}]$, and $j-ra_{n+q}+(r-1)a_{n+q-1}+c_{k,n}$ lies again in the beginning of the \loi\ which starts at the index $(r-1)a_{n+q-1}+\xi  _{n+q+1-r}+1$. So $$||e_{j-ra_{n+q}+(r-1)a_{n+q-1}+c_{k,n}}||\ls 2^{-\frac{1}{2}\sqrt{a_{n+q-1}}},$$ and hence
$$||T^{c_{k,n}}f_{j}||\ls 2^{-\frac{1}{4}\sqrt{a_{n+q-1}}}.$$
\par\smallskip
All cases have now been considered, and in order to finish the proof of Proposition \ref{prop3bis}, we have to put together all estimates. It is clear that we can ensure that
$$||T^{c_{k,n}}(I-Q_{\nu _{n}})\sum_{j>\nu _{n}, j\not\in\cup_{q\ge 1}\cup_{r=1}^{n+q}J_{n+q,r}} f_{j}^{*}(x)f_{j}||\le 100\,||x||\quad \textrm{ for all }x\in X.$$
It remains to estimate 
$$||T^{c_{k,n}}(I-Q_{\nu _{n}})\sum_{q\ge 1}\sum_{r=1}^{n+q}\sum_{j=ra_{n+q}}^{ra_{n+q}+\xi  _{n+q+1-r}}f_{j}^{*}(x)f_{j}||.$$
We have seen that we can ensure for instance that
$$||T^{c_{k,n}}(I-Q_{\nu _{n}})\sum_{q\ge 1}\sum_{r=q+1}^{n+q}\sum_{j=ra_{n+q}}^{ra_{n+q}+\xi  _{n+q+1-r}}f_{j}^{*}(x)f_{j}||\le ||x||$$
(this is Case 5 for $q=1$ and Case 7b of ``short'' intervals for $q\ge 2$)
so that we only have to estimate
$$||T^{c_{k,n}}(I-Q_{\nu _{n}})\sum_{q\ge 2}\sum_{r=1}^{q}\sum_{j=ra_{n+q}}^{ra_{n+q}+\xi  _{n+q+1-r}}f_{j}^{*}(x)f_{j}||$$
(which corresponds to Case 7a of ``long'' intervals).
Inside this term, we know that we can ensure that
$$||T^{c_{k,n}}(I-Q_{\nu _{n}})\sum_{q\ge 2}\sum_{r=1}^{q}\sum_{j=ra_{n+q}+\xi  _{n+q+1-r}+1-c_{k,n}}^{ra_{n+q}+\xi  _{n+q+1-r}}f_{j}^{*}(x)f_{j}||\le ||x||,$$ so that ultimately we have to bound
$$||T^{c_{k,n}}(I-Q_{\nu _{n}})\sum_{q\ge 2}\sum_{r=1}^{q}\sum_{j=ra_{n+q}}^{ra_{n+q}+\xi  _{n+q+1-r}-c_{k,n}}f_{j}^{*}(x)f_{j}||.$$
Now for $j\in [ra_{n+q},ra_{n+q}+\xi  _{n+q+1-r}]$, $f_{j}=z_{r}^{(d_{n+q+1-r}+j-ra_{n+q})}$. We have thus in the case where $X=\ell_{p}\oplus\bigoplus_{\ell_{p}}Z$, $1\le p<+\infty$ (there is an analogous formula for the case where $X=c_{0}\oplus\bigoplus_{c_{0}}Z$), that
\begin{eqnarray}\label{EQ4}
\notag
|| \sum_{q\ge 2}\sum_{r=1}^{q}\sum_{j=ra_{n+q}}^{ra_{n+q}+\xi  _{n+q+1-r}-c_{k,n}}f_{j}^{*}(x)f_{j}||&=&
||\sum_{m\ge n+1}\sum_{d=d_{m}}^{d_{m+1}-c_{k,n}-1}\sum_{r\ge 1}z_{r}^{(d)*}(x)z_{r}^{(d)}||\\ 
&=& \notag \left(\sum_{m\ge n+1}\sum_{d=d_{m}}^{d_{m+1}-c_{k,n}-1}||\sum_{r\ge 1}z_{r}^{(d)*}(x)z_{r}^{(d)}||^{p}_{Z^{(d)}}\right)^{\frac{1}{p}}\\ 
&=& \left(\sum_{m\ge n+1}\sum_{d=d_{m}}^{d_{m+1}-c_{k,n}-1}||P_{Z^{(d)}}x||^{p}\right)^{\frac{1}{p}}.
\end{eqnarray}
Hence by (\ref{eqcarre}), we have for every $x\in X$ that
\begin{equation}\label{eqcarrebis}
 || \sum_{q\ge 2}\sum_{r=1}^{q}\sum_{j=ra_{n+q}}^{ra_{n+q}+\xi  _{n+q+1-r}-c_{k,n}}f_{j}^{*}(x)f_{j}||\le ||x||.
\end{equation}
 Now
\begin{eqnarray*}\label{ea}
&& ||T^{c_{k,n}}(I-Q_{\nu _{n}})\sum_{q\ge 2}\sum_{r=1}^{q}\sum_{j=ra_{n+q}}^{ra_{n+q}+\xi  _{n+q+1-r}-c_{k,n}}f_{j}^{*}(x)f_{j}||\qquad\qquad\qquad\qquad\qquad\qquad\\
 &\qquad&\qquad\qquad\qquad\qquad\qquad=
 ||\sum_{q\ge 2}\sum_{r=1}^{q}\sum_{j=ra_{n+q}}^{ra_{n+q}+\xi  _{n+q+1-r}-c_{k,n}}f_{j}^{*}(x)f_{j+c_{k,n}}||\\
 &\qquad&\qquad\qquad\qquad\qquad\qquad= ||\sum_{m\ge n+1}\sum_{d=d_{m}}^{d_{m+1}-c_{k,n}-1}\sum_{r\ge 1}z_{r}^{(d)*}(x)z_{r}^{(d+c_{k,n})}||,
\end{eqnarray*}
because we have seen that the \ops\ $T^{c_{k,n}}$ act simply as a shift on the parts of the basis $(f_{j})_{j\ge 0}$ which are involved here. Using formula (\ref{EQ4}), we now observe that 
\begin{equation*}\label{eb}
 ||\sum_{m\ge n+1}\sum_{d=d_{m}}^{d_{m+1}-c_{k,n}-1}\sum_{r\ge 1}z_{r}^{(d)*}(x)z_{r}^{(d+c_{k,n})}||= ||\sum_{m\ge n+1}\sum_{d=d_{m}}^{d_{m+1}-c_{k,n}-1}\sum_{r\ge 1}z_{r}^{(d)*}(x)z_{r}^{(d)}||.
\end{equation*}
Hence by (\ref{eqcarrebis}), we get that
$$||T^{c_{k,n}}(I-Q_{\nu _{n}})\sum_{q\ge 2}\sum_{r=1}^{q}\sum_{j=ra_{n+q}}^{ra_{n+q}+\xi  _{n+q+1-r}-c_{k,n}}f_{j}^{*}(x)f_{j}||\le ||x||.$$
So putting things together, we obtain that $||T^{c_{k,n}}(I-Q_{\nu _{n}})x||\le 103\,||x||$, and we are done.
\end{proof}

A straightforward but important corollary of Proposition \ref{prop3bis} is:

\begin{proposition}\label{prop4bis}
 For all $n\ge 1$, all $k\in [1,k_{n}]$ and all $x\in X$,
 $$||T^{c_{k,n}}(I-Q_{\nu _{n}})x||\le 103\,||x-\pi_{[0,\nu _{n}]}x||.$$
\end{proposition}

\begin{proof}
 Since $\pi_{[0,\nu _{n}]}x$ belongs to $F_{\nu _{n}}$, we have $\pi_{[0,\nu _{n}]}x=Q_{\nu _{n}}\pi_{[0,\nu _{n}]}x$, and thus we get that $(I-Q_{\nu _{n}})(x-\pi_{[0,\nu _{n}]}x)=(I-Q_{\nu _{n}})x$. Hence $$||T^{c_{k,n}}(I-Q_{\nu _{n}})x||=||T^{c_{k,n}}(I-Q_{\nu _{n}})(x-\pi_{[0,\nu _{n}]}x)||\le 103\, ||x-\pi_{[0,\nu _{n}]}x||$$
 by Proposition \ref{prop3bis}.
\end{proof}

Let now $x$ be any vector of $X$. If we apply Fact \ref{factbbis} to the vector $Q_{\nu _{n}}x$ (which belongs to $F_{\nu _{n}}$), we get that for each $n\ge 1$ and $k\in [1,k_{n}]$,
$$||T^{c_{k,n}}(Q_{\nu _{n}}x)-p_{k,n}(T)(Q_{\nu _{n}}x)||\le \delta _{n}||Q_{\nu _{n}}x||\le \delta _{n}||Q_{\nu _{n}}||\,||x||\le \delta _{n}M_{a_{n}}||x||$$ by Fact \ref{factbisbis}. Recalling that $\delta _{n}$ can be made arbitrarily small if $\gamma _{n}$ is sufficiently small, that $M_{a_{n}}$ depends only on $a_{n}$ and that $\gamma _{n}$ is chosen after $a_{n}$ is fixed, we can ensure that $\delta _{n}$ is so small that $\delta _{n}M_{a_{n}}\le \frac{1}{a_{n}}$ for instance. Now by Proposition \ref{prop4bis} we have the tail estimate $||T^{c_{k,n}}x-T^{c_{k,n}}(Q_{\nu _{n}}x)||\le 103\,||x-\pi_{[0,\nu _{n}]}x||$,
and thus
$$||T^{c_{k,n}}x-p_{k,n}(T)(Q_{\nu _{n}}x)||\le \frac{1}{a_{n}}||x||+ 103\,||x-\pi_{[0,\nu _{n}]}x||.$$
We record this estimate as Proposition \ref{prop5bis}:

\begin{proposition}\label{prop5bis}
 We can ensure that for every vector $x\in X$, every integers $n\ge 1$ and $k\in [1,k_{n}]$,
 $$||T^{c_{k,n}}x-p_{k,n}(T)(Q_{\nu _{n}}x)||\le \frac{1}{a_{n}}||x||+ 103\,||x-\pi_{[0,\nu _{n}]}x||.$$
\end{proposition}

And this is as far as we can go in this direction: we will not prove that Property (P1) holds true, because one of the ingredients for this is missing. If we wish to prove in the same way as in the proof of Theorem \ref{th2} that for a given polynomial $p$ with $|p|\le 2$ there exists for sufficiently large $n$ an integer $k\in [1,k_{n}]$ such that $||T^{c_{k,n}}x-p(T)x||$ is less than a given $\varepsilon $, we need to use the fact that $\liminf_{n\to +\infty}||(I-Q_{\nu _{n}})x||=0$. This was rather easy to prove under the assumption of Theorem \ref{th2} (this is Fact \ref{factc}) thanks to the simple form of the projection $Q_{\nu _{n}}$ which was modifying only the vector $f_{a_{n+1}}$. In the present context the situation is much more intricate, since $Q_{\nu _{n}}f_{j}$ is different from $f_{j}$ on quite a big set of indices $j$.

Of course, since we are going to prove that any non-zero vector $x\in X$ has dense orbit under the action of $T$, Property (P1) will ultimately hold, but for a different reason than in the proof of Theorem \ref{th2}.
\par\smallskip

The next step in the proof of Theorem \ref{th1} is to prove an analogue of Proposition \ref{prop5}, namely that if $x\in X$ is any vector of norm $1$, some coordinate in the $(e_{j})$-basis of $F_{\mu_{n}}$ of one of the vectors $Q_{\mu _{n}}x$ must be ``large'' (i.e. not too small) for arbitrarily large $n$. Recall that $F_{\mu _{n}}$ is the linear span of the vectors $e_{0},\ldots, e_{\mu _{n}}$, and that we denote for each $n\ge 0$ and each $j\in [0,\mu _{n}]$ by $e_{j}^{*(\mu _{n})}(y)$ the $j^{th}$ coordinate of a vector $y$ of $F_{\mu _{n}}$ in the basis $(e_{j})_{j=0, \ldots, \mu_{n}}$:
$$y=\sum_{j=0}^{\mu _{n}}e_{j}^{*(\mu _{n})}(y)e_{j}.$$

\begin{proposition}\label{prop6bis}
 Let $x$ be a vector of $X$ with $||x||=1$. If $(A_{n})_{n\ge 1}$ is a sequence of integers such that each $A_{n}$ is sufficiently large \wrt\ $a_{n}$, there exists for each integer $N\ge 1$ an integer $n\ge N+2$ such that the following property holds true:
 
 there exists an integer $j\in [0, (n-N)a_{n}-1]$ such that
 $$|e_{j}^{*(\mu _{n})}(Q_{\mu _{n}}x)|\ge \frac{1}{A_{n}^{(\mu _{n}-j+1)!^{2}}}\cdot$$
\end{proposition}

\begin{proof}
 Suppose that $N\ge 1$ is such that for all $n\ge N+2$ and all $j\in [0, (n-N)a_{n}-1]$, 
 \begin{equation}\label{eqnouvellebis}
   |e_{j}^{*(\mu _{n})}(Q_{\mu _{n}}x)|<\frac{1}{A_{n}^{(\mu _{n}-j+1)!^{2}}}<\frac{1}{A_{n}}\cdot
 \end{equation}
 We have
 \begin{eqnarray}\label{supl1}
  Q_{\mu _{n}}x &=& \pi_{[0,\mu _{n}]}x-\sum_{r=2}^{n+1}\frac{a_{n+1-r}}{\alpha _{r}}\sum_{j=ra_{n+1}}^{ra_{n+1}+\xi  _{n+2-r}}f_{j}^{*}(x)\,e_{j-ra_{n+1}+(r-1)a_{n}}\\
  &=&  \pi_{[0,\mu _{n}]}x-\sum_{r=2}^{n+1}\frac{a_{n+1-r}}{\alpha _{r}}\sum_{l=0}^{\xi  _{n+2-r}}f_{ra_{n+1}+l}^{*}(x)\,e_{(r-1)a_{n}+l}\notag.
 \end{eqnarray}
 
Let us determine $e_{j}^{*(\mu _{n})}(Q_{\mu _{n}}x)$ for $j$ belonging to one of the (a)-\woi s $[sa_{n},sa_{n}+\xi  _{n+1-s}]$ for $s=1,\ldots, n$. The vector $e_{j}$ appears in two places in the expression (\ref{supl1}) of $Q_{\mu _{n}}x$:

-- it appears in the term $\pi_{[0,\mu _{n}]}x$, and the only contribution comes from the vector $f_{j}=\frac{a_{n-s}}{\alpha _{s}}(e_{j}-e_{j-sa_{n}+(s-1)a_{n-1}})$ so that for every $j\in [sa_{n},sa_{n}+\xi  _{n+1-s}]$
$$e_{j}^{*(\mu _{n})}(\pi_{[0,\mu _{n}]}x)=f_{j}^{*}(x)\,\frac{a_{n-s}}{\alpha _{s}}\cdot$$ 

-- it appears in the second term of (\ref{supl1}), and here the only contribution appears when $r=s+1$ and $l=j-sa_{n}$ (observe that $2\le r\le n+1$, so that there is indeed a contribution). Hence
$$e_{j}^{*(\mu _{n})}\left(\sum_{r=2}^{n+1}\frac{a_{n+1-r}}{\alpha _{r}}\sum_{l=0}^{\xi  _{n+2-r}}f_{ra_{n+1}+l}^{*}(x)\,e_{(r-1)a_{n}+l}\right)=\frac{a_{n-s}}{\alpha _{s+1}} f^{*}_{(s+1)a_{n+1}+j-sa_{n}}(x).
$$
So
\begin{eqnarray}\label{supl2}
 e_{j}^{*(\mu _{n})}(Q_{\mu _{n}}x)=a_{n-s}\left(\frac{1}{\alpha_{s}}f_{j}^{*}(x)-\frac{1}{\alpha _{s+1}}f^{*}_{j-sa_{n}+(s+1)a_{n+1}}(x)\right)
\end{eqnarray}
for every $j\in [sa_{n},sa_{n}+\xi  _{n+1-s}]$. Our assumption (\ref{eqnouvellebis}) thus implies, provided $A_{n}$ is large enough for each $n$, that for every such $s$ such that $1\le s<n-N$ and every $j\in [sa_{n},sa_{n}+\xi_{n+1-s}]$,
$$\left|\frac{1}{\alpha_{s}}f_{j}^{*}(x)-\frac{1}{\alpha _{s+1}}f^{*}_{j-sa_{n}+(s+1)a_{n+1}}(x)\right|<\frac{1}{a_{n-s}}\,\frac{1}{A_{n}}\le \frac{1}{a_{N+1}}\,\frac{1}{A_{n}}\cdot$$ Hence we have for every $l\in [0,\xi  _{n+1-s}]$ that
$$\left|\frac{1}{\alpha_{s}}f_{sa_{n}+l}^{*}(x)-\frac{1}{\alpha _{s+1}}f^{*}_{(s+1)a_{n+1}+l}(x)\right|<\frac{1}{a_{N+1}}\,\frac{1}{A_{n}}\cdot$$
This being true for every $n\ge N+2$, we
can apply it for 
every $q\ge 0$ with $n$ replaced by $n+q$ and $s$ replaced by $s+q$ 
($s+q$ belongs to $[1, n+q-N-1]$ for each $s\in [1, n-N-1]$), and we get
that
$$\left|\frac{1}{\alpha_{s+q}}f_{(s+q)a_{n+q}+l}^{*}(x)-\frac{1}{\alpha _{s+q+1}}f^{*}_{(s+q+1)a_{n+q+1}+l}(x)\right|<\frac{1}{a_{N+1}}\,\frac{1}{A_{n+q}}
$$ for each $q\ge 0$  and  $s\in [1, n-N-1]$.
Adding the first $q$ such inequalities, we get that
for each $s\in [1,n-N-1]$, $l\in [0,\xi  _{n+1-s}]$ and $q\ge 1$,
\begin{eqnarray}\label{new3}
 \left|\frac{1}{\alpha_{s}}f_{sa_{n}+l}^{*}(x)-\frac{1}{\alpha _{s+q}}f^{*}_{(s+q)a_{n+q}+l}(x)\right|&<&\frac{1}{a_{N+1}}\left(\frac{1}{A_{n}}+\frac{1}{A_{n+1}}+\ldots+\frac{1}{A_{n+q-1}}\right)\\
 &<&\frac{1}{a_{N+1}}\,\frac{2}{A_{n}}\notag
\end{eqnarray}
if the sequence $(A_{n})$ grows sufficiently quickly.

Let us now fix integers $s\ge 1$, $n\ge s+(N+1)$ and $l\in [0,\xi  _{n+1-s}]$. Then (\ref{new3}) implies that for every $q\ge 1$,
$$|f^{*}_{(s+q)a_{n+q}+l}(x)|\ge \alpha _{s+q}\left(\frac{1}{\alpha _{s}}|f^{*}_{sa_{n}+l}(x)|-
\frac{2}{a_{N+1}A_{n}}\right)\cdot$$ 
Suppose that $$\frac{1}{\alpha _{s}}|f^{*}_{sa_{n}+l}(x)|-
\frac{2}{a_{N+1}A_{n}}>0,$$ and recall  our assumption that $\limsup_{j\to +\infty}\alpha _{j}>\delta _{0}>0$, which has never been used until this point. It implies that 
$$\limsup_{q\to +\infty}|f^{*}_{(s+q)a_{n+q}+l}(x)|\ge {\delta _{0}}\left(\frac{1}{\alpha _{s}}|f^{*}_{sa_{n}+l}(x)|-
\frac{2}{a_{N+1}A_{n}}\right)>0.$$ But $(s+q)a_{n+q}+l$ tends to infinity as $q$ tends to infinity, and so $\limsup_{j\to +\infty}|f^{*}_{j}(x)|>0$, which is a contradiction since $(f_{j})_{j\ge 0}$ is a Schauder basis of $X$. We thus come to the conclusion that $$|f^{*}_{sa_{n}+l}(x)|\le\frac{2\alpha_s}{a_{N+1}A_{n}}$$ for every $s\ge 1$, $n\ge s+(N+1)$ and $l\in [0,\xi  _{n+1-s}]$. This means that 
$|f_{j}^{*}(x)|$ is very small, in a way which can be quantified, for every $j$ in the intervals $[a_{N+2},a_{N+2}+\xi  _{N+2}]$, $[a_{N+3},a_{N+3}+\xi  _{N+3}]$, $[2a_{N+3},2a_{N+3}+\xi  _{N+2}]$, etc... and more generally in any interval of the form $[sa_{N+q},sa_{N+q}+\xi  _{N+q-s+1}]$, where $q\ge 2$ and $s\in [1,q-1]$.
\par\smallskip
Let us now consider, for $n\ge N+2$, the vector
$$\pi_{[0,\mu _{n}]}x=\sum_{j=0}^{\mu _{n}}e_{j}^{*(\mu _{n})}(x)e_{j}.$$ Recall first that 
\begin{eqnarray}\label{supl3}
 e_{j}^{*(\mu _{n})}(x)=e_{j}^{*(\mu _{n})}(Q_{\mu _{n}}x)\quad \textrm{for every } j\in [0,\mu _{n}]\setminus\bigcup_{s=1}^{n}[sa_{n},sa_{n}+\xi  _{n+1-s}]
\end{eqnarray}
and that, by (\ref{supl3}), we have for every $s=1 \ldots n$ and every $j\in [sa_{n},sa_{n}+\xi  _{n+1-s}] $ that
$$e_{j}^{*(\mu _{n})}(x)=e_{j}^{*(\mu _{n})}(Q_{\mu _{n}}x)+\frac{a_{n-s}}{\alpha _{s+1}}
f_{j-sa_{n}+(s+1)a_{n+1}}^{*}(x).$$
Now we have seen that $$|f^{*}_{(s+1)a_{n+1}+l}(x)|\le\frac{2\alpha_{s+1}}{a_{N+1}A_{n+1}}$$ for every $0\le s\le n-N-1$ and every $l\in [0,\xi  _{n+1-s}]$. Hence
$$|e_{j}^{*(\mu _{n})}(x)-e_{j}^{*(\mu _{n})}(Q_{\mu _{n}}x)|\le 2\,\frac{a_{n-1}}{a_{N+1}A_{n+1}} < \frac{1}{A_{n}}\quad \textrm{ for every } j\in \bigcup_{s=1}^{n-N-1}[sa_{n},sa_{n}+\xi  _{n+1-s}]$$
 if $A_{n}$ is chosen sufficiently large.
Combining this with (\ref{supl3}), we get that
\begin{eqnarray*}\label{supl4}
 |e_{j}^{*(\mu _{n})}(x)-e_{j}^{*(\mu _{n})}(Q_{\mu _{n}}x)| \le\frac{1}{A_{n}}\quad\textrm{for every } j\in [0,\mu _{n}]\setminus\bigcup_{s=n-N}^{n}[sa_{n},sa_{n}+\xi  _{n+1-s}].
\end{eqnarray*}
Recall now that our initial assumption implies that
$|e_{j}^{*(\mu _{n})}(Q_{\mu _{n}}x)|<\frac{1}{A_{n}}$ for every $j$ belonging to the interval $ [0, (n-N)a_{n}-1]$. So 
$$|e_{j}^{*(\mu _{n})}(x)|<\frac{2}{A_{n}}\quad\textrm{for every } j\in [0, (n-N)a_{n}-1].$$  Hence
$$\left|\left|\sum_{j=0}^{(n-N)a _{n}-1}e_{j}^{*(\mu _{n})}(x)e_{j}\right|\right|\le \frac{2(n-N)a_{n}}{A_{n}}\,.\,\sup_{j=0,\ldots, \mu _{n}}||e_{j}^{*(\mu _{n})}||\,.\,||e_{j}||<\frac{1}{a_{n}}$$ if $A_{n}$ is sufficiently large. Observe also that for each $j$ belonging to one of the \loi s $[ra_{n}+\xi  _{n+1-r}+1, (r+1)a_{n}-1]$, $r=n-N,\ldots, n-1$, we have $f_{j}=\lambda _{j}e_{j}$ for some scalar $\lambda _{j}$, and that
$$e_{j}^{*(\mu _{n})}(\pi_{[0,\mu _{n}]}x) =\lambda _{j}f_{j}^{*}(x), \quad \textrm{ i.e. } \quad
e_{j}^{*(\mu _{n})}(\pi_{[0,\mu _{n}]}x )e_{j}=f_{j}^{*}(x)f_{j}.$$
Hence
\begin{eqnarray*}
 \pi_{[0,\mu _{n}]}x&=&\sum_{j=0}^{(n-N)a _{n}-1}e_{j}^{*(\mu _{n})}(x)e_{j}+\sum_{r=n-N}^{n-1}\sum_{j=ra_{n}+\xi  _{n+1-r}+1}^{(r+1)a_{n}-1}f_{j}^{*}(x)f_{j}\\
 &+&\sum_{r=n-N}^{n-1}\sum_{j=ra_{n}}^{ra_{n}+\xi  _{n+1-r}}e_{j}^{*(\mu _{n})}(x)e_{j}.
\end{eqnarray*}
Since $$\left|\left|\sum_{j=0}^{(n-N)a _{n}-1}e_{j}^{*(\mu _{n})}(x)e_{j}\right|\right|<\frac{1}{a_{n}}$$ 
and (in the case where $X=\ell_{p}\oplus\bigoplus_{\ell_{p}}Z$)
$$\left|\left|\sum_{r=n-N}^{n-1}\sum_{j=ra_{n}+\xi  _{n+1-r}+1}^{(r+1)a_{n}-1}f_{j}^{*}(x)f_{j}\right|\right|=
\left(\sum_{r=n-N}^{n-1}\sum_{j=ra_{n}+\xi  _{n+1-r}+1}^{(r+1)a_{n}-1}|f_{j}^{*}(x)|^{p}\right)^{\frac{1}{p}}
$$ which tends to $0$ as $n$ tends to infinity, we obtain that 
$$\left|\left|\pi_{[0,\mu _{n}]}x-\sum_{r=n-N}^{n-1}\sum_{j=ra_{n}}^{ra_{n}+\xi  _{n+1-r}}e_{j}^{*(\mu _{n})}(x)e_{j}\right|\right|\to 0\quad \textrm{ as } n\to +\infty,$$ i.e.
$$\left|\left|\pi_{[0,\mu _{n}]}x-\sum_{q=1}^{N}\sum_{l=0}^{\xi  _{q+1}}e_{(n-q)a_{n}+l}^{*(\mu _{n})}(x)e_{(n-q)a_{n}+l}\right|\right|\to 0\quad \textrm{ as } n\to +\infty.$$
Recall now that by Fact \ref{fact1bis}, $$e_{(n-q)a_{n}}=\frac{1}{a_{q}}\sum_{k=1}^{n-q}\alpha _{k}z_{k}^{(d_{q+1})}+e_{0},$$ so that $$e_{(n-q)a_{n}+l}=\frac{1}{a_{q}}\sum_{k=1}^{n-q}\alpha _{k}z_{k}^{(d_{q+1}+l)}+e_{l}\quad \textrm{for every } l\in [0, \xi  _{q+1}].$$ Thus
\begin{eqnarray}\label{eqc}
 \Big|\Big| \pi_{[0,\mu _{n}]}x &-&\sum_{q=1}^{N}\sum_{l=0}^{\xi  _{q+1}}e_{(n-q)a_{n}+l}^{*(\mu _{n})}(x)e_{l}\\
 &-& 
 \sum_{q=1}^{N}\frac{1}{a_{q}}\sum_{l=0}^{\xi  _{q+1}}e_{(n-q)a_{n}+l}^{*(\mu _{n})}(x)
 \left(\sum_{k=1}^{n-q}\alpha _{k}z_{k}^{(d_{q+1}+l)}\right)
 \Big|\Big|\to 0\quad \textrm{ as } n\to +\infty. \notag
\end{eqnarray}
 Since all the terms $e_{(n-q)a_{n}+l}^{*(\mu _{n})}(x)e_{l}$ in the second sum in the expression
above belongs, whatever the value of $n$, $q$ and $l$, to the finite-dimensional vector space $F_{\xi  _{N+1}}$, there exists an integer $r_{0}$ independent of $q$, $l$ and $n$ such that
$$z_{r_{0}}^{(d_{q+1}+l)*}\left( e_{(n-q)a_{n}+l}^{*(\mu _{n})}(x)e_{l}\right)=0$$
for every $q\in [1,N]$, $l\in [0,\xi  _{q+1}]$ and $n\ge N+2$. Hence, since for sufficiently large $n$ the index $r_{0}$ belongs to $ [1,n-q]$ for all $q$, applying the functional $z_{r_{0}}^{(d_{q+1}+l)*}$ to both sides of (\ref{eqc}) we get that
$$\left|z_{r_{0}}^{(d_{q+1}+l)*}(\pi _{[0,\mu _{n}]}x)-\frac{1}{a_{q}}e_{(n-q)a_{n}+l}^{*(\mu _{n})}(x)\alpha _{r_{0}}\right|\to 0\quad \textrm{ as } n\to +\infty,$$ i.e. that
\begin{equation}\label{new4}
 e_{(n-q)a_{n}+l}^{*(\mu _{n})}(x)\to \frac{a_{q}}{\alpha _{r_{0}}}z_{r_{0}}^{(d_{q+1}+l)*}(x):=\gamma _{q,l}\quad \textrm{ as } n\to +\infty
\end{equation}
for all $q\in [1,N]$ and $l\in [0,\xi_{q+1}]$.
Hence the second term in the display (\ref{eqc}) converges to the vector $$y:=\sum_{q=1}^{N}\sum_{l=0}^{\xi_{q+1}}\gamma _{q,l}e_{l}$$ as $n$ tends to infinity. Since $\pi_{[0,\mu _{n}]}x$ tends to $x$ as $n$ tends to infinity, it follows again from (\ref{eqc}) that
\begin{eqnarray}\label{eqd}
 \sum_{q=1}^{N}\sum_{l=0}^{\xi_{q+1}}e_{(n-q)a_{n}+l}^{*(\mu _{n})}(x)\frac{1}{a_{q}}
\left( \sum_{k=1}^{n-q}\alpha _{k}z_{k}^{(d_{q+1}+l)}\right) \to x-y\quad \textrm{ as } n\to +\infty.
\end{eqnarray}
Recall now that by (\ref{EQ3}) we have for each $q$ and $l$ that
$$||
 \sum_{k=1}^{n-q}\alpha _{k}z_{k}^{(d_{q+1}+l)}||=||\sum_{k=1}^{n-q}\alpha _{k}z_{k}||\le 1.$$ Since the bound on the right-hand side is independent of $q$ and $l$ and we sum only finitely many terms in (\ref{eqd}), the combination of (\ref{new4}) and (\ref{eqd}) yields that
 \begin{eqnarray}\label{eqe}
   \sum_{q=1}^{N}\sum_{l=0}^{\xi_{q+1}}\gamma _{q,l}\frac{1}{a_{q}}\left(\sum_{k=1}^{n-q}\alpha _{k}z_{k}^{(d_{q+1}+l)}\right)\to x-y\quad \textrm{ as } n\to +\infty.
 \end{eqnarray}
Hence, for each $q\in [1,N]$ and each $l\in [0,\xi  _{q+1}]$, the sequence $$\left(\gamma _{q,l}
\frac{1}{a_{q}}\sum_{k=1}^{n-q}\alpha _{k}z_{k}^{(d_{q+1}+l)}\right)_{n\ge q+1}$$
converges as $n$ tends to infinity in the ${(d_{q+1}+l)}^{th}$ copy $Z^{(d_{q+1}+l)}$ of $Z$. This means that for each such $q$ and $l$, the sequence $$\left(
\gamma _{q,l}\frac{1}{a_{q}}\sum_{k=1}^{n-q}\alpha _{k}z_{k}\right)_{n\ge q+1}$$ converges in $Z$. But since our assumption is that the sequence $(\sum_{k=1}^{n}\alpha _{k}z_{k})_{n\ge 1}$ does not converge in $Z$, this forces $\gamma _{q,l}$ to be equal to zero for each $q$ and $l$. So $y=0$, and by (\ref{eqe}) we get that $x=y=0$, which contradicts our assumption that $||x||=1$. This final contradiction proves Proposition \ref{prop6bis}.
\end{proof}

An important consequence of Proposition \ref{prop6bis} is Proposition \ref{prop7bis} below. As in Section \ref{sec3}, we will denote by $T_{\mu _{n}}$
the truncated forward shift on $F_{\mu _{n}}$ defined by $T_{\mu _{n}}e_{j}=e_{j+1}$ for every $j\in [0,\mu _{n}-1]$ and $T_{\mu _{n}}e_{\mu _{n}}=0$.

\begin{proposition}\label{prop7bis}
 There exists a sequence $(C_{\mu _{n}})_{n\ge 1}$ of positive real numbers, with $C_{\mu _{n}}$ depending only on $\mu _{n}$ for each $n$, such that the following property holds true:
 
 for every vector $x\in X$ with $||x||=1$ and every integer
 $N\ge 1$, there exist an integer $n\ge N+2$ and a polynomial $p_n$ of degree at most $\mu _{n}$ such that
 $$|p_n|\le C_{\mu _{n}}\quad\textrm{and}\quad ||p_n(T_{\mu _{n}})T_{\mu _{n}}(Q_{\mu _{n}}x)-
 e_{(n-N)a_{n}}||<\frac{1}{a_{n}}\cdot$$
\end{proposition}

\begin{proof}
Let $x$ and $N\ge 1$ be as given. By Proposition \ref{prop6bis} there exist an $n\ge N+2$ and a $j\in [0,(n-N)a_{n}-1]$ such that $$|e_{j}^{*(\mu _{n})}(Q_{\mu _{n}}x)|\ge\frac{1}{A_{n}^{(\mu _{n}-j+1)!^{2}}}$$ where $A_{n}$ depends only on $a_{n}$. So if we write $Q_{\mu _{n}}x$ as $Q_{\mu _{n}}x=\sum_{j=0}^{\mu _{n}}e_{j}^{*(\mu _{n})}(Q_{\mu _{n}}x)e_{j}$, and if $j_{n}$ is the smallest integer in $[0,\mu _{n}]$ such that
$$|e_{j_{n}}^{*(\mu _{n})}(Q_{\mu _{n}}x)|\ge\frac{1}{A_{n}^{(\mu _{n}-j_{n}+1)!^{2}}}$$ then $j_{n}\le (n-N)a_{n}-1$. We have thus
\begin{eqnarray}\label{supl5}
 \left|\left|\sum_{j=0}^{j_{n}-1}e_{j}^{*(\mu _{n})}(Q_{\mu _{n}}x)e_{j}\right|\right|&\le& \sum_{j=0}^{j_{n}-1}\frac{1}{A_{n}^{(\mu _{n}-j+1)!^{2}}}\,\sup_{j=0,\ldots,\mu _{n}}||e_{j}|\\|&\le&\mu _{n}
\,\sup_{j=0,\ldots,\mu _{n}}||e_{j}||\,
\frac{2}{A_{n}^{(\mu _{n}-j_{n}+2)!^{2}}}\cdot\notag
\end{eqnarray}
Now observe that, on the one hand, $$|e_{j}^{*(\mu _{n})}(Q_{\mu _{n}}x)|\ge \frac{1}{A_{n}^{(\mu _{n}-j+1)!^{2}}}$$ for some $j\in [0,(n-N)a_{n}-1]$, and on the other hand 
$||Q_{\mu _{n}}x||\le M_{a_{n}}$ by Fact \ref{factbis}. So we have
 $$M_{a_{n}}\ge ||Q_{\mu _{n}}x||\ge \frac{1}{A_{n}^{(\mu _{n}+1)!^{2}}}\,\frac{1}{\sup_{j=0,\ldots,\mu _{n}}||e_{j}^{*(\mu _{n})}||}\cdot$$ Also, the two quantities $\inf_{r=1,\ldots,n}||e_{ra_{n}-1}||$ and $\sup_{r=1,\ldots,n}||e_{ra_{n}-1}||$ are both positive and finite, and they depend only on $a_{n}$. Thus if we set
 $$\beta _{n}=\inf\left(\frac{1}{A_{n}^{(\mu _{n}+1)!^{2}}}\,\frac{1}{\sup_{j=0,\ldots,\mu _{n}}||e_{j}^{*(\mu _{n})}||}, \inf_{r=1,\ldots,n}||e_{ra_{n}-1}||\right)>0$$ and
 $$M_{n}=\sup\left(M_{a_{n}}, \sup_{r=1,\ldots,n}||e_{ra_{n}-1}||\right)<+\infty,$$ $\beta _{n}$ and $M_{n}$ depend only on $\mu _{n}$, and we have
 $$\beta _{n}\le ||Q_{\mu _{n}}x||\le M_{n}\quad \textrm{ and }\quad
 \beta _{n}\le ||e_{(n-N)a_{n}-1}||\le M_{n}.$$
 Now, the same proof as that of Fact \ref{factf}, starting from the sequences $(\beta _{n})_{n\ge 1}$ and $(M_{n})_{n\ge 1}$ but with $a_{n}$ replaced by $\mu _{n}$, shows that there exist a constant $D_{\mu _{n}}$, depending only on $\mu _{n}$, and a polynomial $p_n$ of degree at most $\mu _{n}$ with $$|p_n|\le D_{\mu _{n}}\,A_{n}^{(\mu _{n}-j_{n}+1)!^{2}(\mu _{n}-j_{n}+1)}$$ such that
 \begin{eqnarray}\label{eqf}
 p_n({T_{\mu _{n}}})\left(\sum_{j=j_{n}}^{\mu _{n}}e_{j}^{*(\mu _{n})}(Q_{\mu _{n}}x)e_{j}\right)=
 e_{(n-N)a_{n}-1}.
 \end{eqnarray}
 Also, by (\ref{supl5}),
 \begin{eqnarray*}
  \left|\left|p_n({T_{\mu _{n}}})\left(\sum_{j=0}^{j _{n}-1}e_{j}^{*(\mu _{n})}(Q_{\mu _{n}}x)e_{j}\right)\right|\right|
  &\le& |p_n|\,||T_{\mu _{n}}||^{\mu _{n}}\,\mu _{n}\,\sup_{j=0,\ldots,\mu _{n}}||e_{j}||\,
  \frac{2}{A_{n}^{(\mu _{n}-j_{n}+2)!^{2}}}\\
  &\le &D'_{\mu _{n}}\frac{A_{n}^{(\mu _{n}-j_{n}+1)!^{2}(\mu _{n}-j_{n}+1)}}{A_{n}^{(\mu _{n}-j_{n}+1)!^{2}(\mu _{n}-j_{n}+2)^{2}}}\\
  \end{eqnarray*}
  where $D'_{\mu _{n}}=2D_{\mu _{n}}||T_{\mu _{n}}||^{\mu _{n}}\,\mu _{n}\,\sup_{j=0,\ldots,\mu _{n}}||e_{j}||$
  is independent of $A_{n}$. Hence
  \begin{eqnarray*}
  \left|\left|p_n({T_{\mu _{n}}})\left(\sum_{j=0}^{j _{n}-1}e_{j}^{*(\mu _{n})}(Q_{\mu _{n}}x)e_{j}\right)\right|\right|
  &\le& D'_{\mu _{n}}\frac{1}{A_{n}^{(\mu _{n}-j_{n}+1)!^{2}((\mu _{n}-j_{n}+2)^{2}-(\mu _{n}-j_{n}+1))}}\\
  &\le&\frac{D'_{\mu _{n}}}{A_{n}}<\frac{1}{||T_{\mu _{n}}||a_{n}}
 \end{eqnarray*}
 if $A_{n}$ is sufficiently large.
Putting this estimate together with the equality (\ref{eqf}), we obtain that
\begin{eqnarray}\label{eqg}
|| p_n({T_{\mu _{n}}})(Q_{\mu _{n}}x)-
 e_{(n-N)a_{n}-1}||<\frac{1}{||T_{\mu _{n}}||a_{n}}\cdot
\end{eqnarray}
Also $$|p_n|\le  D_{\mu _{n}}\,A_{n}^{(\mu _{n}-j_{n}+1)!^{2}(\mu _{n}-j_{n}+1)}\le
 D_{\mu _{n}}\,A_{n}^{(\mu _{n}+1)!^{2}(\mu _{n}+1)}:=C_{\mu _{n}},$$ and $C_{\mu _{n}}$ depends only on $\mu _{n}$ ($A_{n}$ is chosen very large once $\mu _{n}$ is fixed, but $A_{n}$ depends only on the construction until the index $\mu _{n}$). Applying $T_{\mu _{n}}$ to the inequality (\ref{eqg}) yields that
 $$||p_n(T_{\mu _{n}})T_{\mu _{n}}(Q_{\mu _{n}}x)-
 e_{(n-N)a_{n}}||<\frac{1}{a_{n}},$$ which is the conclusion we were looking for.
\end{proof}

The proof of Theorem \ref{th1} follows now rather classical lines. Let $x$ be a vector of $X$ with $||x||=1$, and let $N\ge 1$ be a fixed integer. By Proposition \ref{prop7bis} above there exists an integer $n\ge N+2$ and a \pol\ $p$ of degree less than $\mu _{n}$ such that $|p_n|\le C_{\mu _{n}}$
 and $$||p_n(T_{\mu _{n}})T_{\mu _{n}}(Q_{\mu _{n}}x)-
 e_{(n-N)a_{n}}||<\frac{1}{a_{n}}\cdot$$ A standard argument shows then that, provided $b_{n}$ is large enough,
 \begin{equation}\label{new5}
 ||p_n(T)T(Q_{\mu _{n}}x)-
 e_{(n-N)a_{n}}||<\frac{2}{a_{n}} 
 \end{equation}
(the vector $p_n(T)T(Q_{\mu _{n}}x)-p_n(T_{\mu _{n}})T_{\mu _{n}}(Q_{\mu _{n}}x)$ is supported in the interval $[\mu _{n}+1,2(\mu _{n}+1)]$, which is in the beginning of the \loi\ $[\mu_{n},b_{n}]$). Let now $q_n$ be the \pol\ defined by $q_n(\zeta )=\frac{\zeta ^{b_{n}+1}}{b_{n}}p_n(\zeta )$: we have $|q_n|\le \frac{C_{\mu _{n}}}{b_{n}}$ and the degree of $q_n$ is less than $\mu _{n}+b_{n}+1$, so less than 
$\nu_{n}$. Recall the following fact, which is the reason why the (b)-\woi s were introduced in the construction:
 \begin{fact}\label{fact??}
  If, for each $n\ge 1$, the integer 
  $b_{n}$ is sufficiently large \wrt\ $a_{n}$, we have
  $$||(\frac{T^{b_{n}}}{b_{n}}-I)Ty||\le \frac{1}{\sqrt{b_{n}}}||y||\quad \textrm{for every }n\ge 1 \textrm{ and } y\in F_{\mu _{n}}.$$
 \end{fact}

Applying Fact \ref{fact??} to the vector $Q_{\mu _{n}}x$ (and remembering that $||x||=1$), we get that
$$||(\frac{T^{b_{n}}}{b_{n}}-I)p_n(T)T(Q_{\mu _{n}}x)||\le |p_n|\,||T||^{\mu _{n}}\frac{1}{\sqrt{b_{n}}}\,||Q_{\mu _{n}}||\le \frac{D^{''}_{\mu _{n}}}{\sqrt{b_{n}}}$$ where $D^{''}_{\mu _{n}}$ depends only on $\mu _{n}$. Hence we can ensure by taking $b_{n}$ sufficiently large that
$$||(\frac{T^{b_{n}}}{b_{n}}-I)p_n(T)T(Q_{\mu _{n}}x)||<\frac{1}{a_{n}},$$ which in turn combined with (\ref{new5}) yields that
\begin{eqnarray}\label{eqh}
 ||q_n(T)(Q_{\mu _{n}}x)-e_{(n-N)a_{n}}||<\frac{3}{a_{n}}\cdot
\end{eqnarray}
Since $|q_n|\le \frac{C_{\mu _{n}}}{b_{n}}$, we can ensure by taking $b_{n}$ sufficiently large at each step $n$ that $|q|\le 2$. Recalling that the degree of $q$ is less than $\mu _{n}+b_{n}+1<\nu _{n}$, this implies that if we choose $\varepsilon_n$ in the construction of the (c)-part to be less than $4^{-\nu _{n}}$, there exists a $k\in [1,k_{n}]$ such that $|p_{k,n}-q_n|<4^{-\nu _{n}}$. Then, since the degree of $p_{k,n}$ is less than $\nu _{n}$, we have
$$||p_{k,n}(T)(Q_{\mu _{n}}x)-q_n(T)(Q_{\mu _{n}}x)||\le 4^{-\nu _{n}}||T||^{\nu _{n}}||Q_{\mu _{n}}||\le 2^{-\nu _{n}}||Q_{\mu _{n}}||<\frac{1}{a_{n}}$$ if $b_{n}$ is sufficiently large.
So plugging this into (\ref{eqh}) yields that
\begin{eqnarray}\label{eqi}
 ||p_{k,n}(T)(Q_{\mu _{n}}x)-e_{(n-N)a_{n}}||<\frac{4}{a_{n}}\cdot
\end{eqnarray}
Our goal is now to replace $Q_{\mu _{n}}x$ by $x$ in this inequality. Proposition \ref{prop5bis} gives suitable estimates for the norm of the difference $T^{c_{k,n}}x-p_{k,n}(T)(Q_{\nu _{n}}x)$, and these yield that
\begin{eqnarray}\label{eqj}
 ||T^{c_{k,n}}x-e_{(n-N)a_{n}}||&\le&||T^{c_{k,n}}x-p_{k,n}(T)(Q_{\nu _{n}}x)||\notag\\
& +&||p_{k,n}(T)(Q_{\nu _{n}}x)-p_{k,n}(T)(Q_{\mu _{n}}x)||\notag\\
 &+&||p_{k,n}(T)(Q_{\mu _{n}}x)-e_{(n-N)a_{n}}||\\
 &<& \frac{5}{a_{n}}+103\,||x-\pi _{[0,\nu _{n}]}x||+||p_{k,n}(T)(Q_{\nu _{n}}-Q_{\mu _{n}})x||.\notag
\end{eqnarray}
Just as in the proof of Theorem \ref{th2}, the only term we need to estimate now is the last one. We have
\begin{eqnarray}\label{eqk}
||p_{k,n}(T)(Q_{\nu _{n}}-Q_{\mu _{n}})x||&=&||p_{k,n}(T)\pi_{[\mu _{n}+1,\nu _{n}]}x||\notag\\
&\le& ||(p_{k,n}(T)-q_n(T))\pi_{[\mu _{n}+1,\nu _{n}]}x||\notag\\
&+&||q_n(T)\pi_{[\mu _{n}+1,\nu _{n}]}x||\\
&\le& 2^{-\nu _{n}}+|p_n|2^{\mu _{n}}\frac{1}{b_{n}}||T^{b_{n}+1}\pi_{[\mu _{n}+1,\nu _{n}]}x||\notag\\
&\le&2^{-\nu _{n}}+\frac{C_{\mu _{n}}2^{\mu _{n}}}{b_{n}}||T^{b_{n}+1}\pi_{[\mu _{n}+1,\nu _{n}]}x||.\notag
\end{eqnarray}
But the successive shades of the (b)-\woi s ensure again that $$||T^{b_{n}+1}\pi_{[\mu _{n}+1,\nu _{n}]}y||\le 2||y||\quad \textrm{for every } y\in X$$ (i.e. that the analogue of Proposition \ref{prop3} is still true in this context), and thus we obtain if $b_{n}$ is large enough that
$$||p_{k,n}(T)(Q_{\nu _{n}}-Q_{\mu _{n}})x||<\frac{2}{a_{n}},$$ for instance. Finally,
$$ ||T^{c_{k,n}}x-e_{(n-N)a_{n}}||<\frac{7}{a_{n}}+103\,||x-\pi _{[0,\nu _{n}]}x||.$$ We are now almost done: by (\ref{new2}) of Fact \ref{fact1bis} we have $$||e_{(n-N)a_{n}}-e_{0}||<\frac{1}{a_{N}},$$ and thus we get that there exist $n\ge N+2$ and $k\in [1,k_{n}]$ such that
\begin{eqnarray*}
||T^{c_{k,n}}x-e_{0}||&<&\frac{1}{a_{N}}+\frac{7}{a_{n}}+103\,||x-\pi _{[0,\nu _{n}]}x||\\
&\le& \frac{8}{a_{N}}+103\,\sup_{n>N}||x-\pi _{[0,\nu _{n}]}x||.
\end{eqnarray*}
Since this holds true for any 
$N\ge 1$, we eventually obtain that for every $\varepsilon >0$ there exists an integer $c\ge 1$ such that $||T^{c}x-e_{0}||<\varepsilon $. Since $e_{0}$ is a \hy\ vector for $T$ (recall that this follows from Fact \ref{factbbis}), $x$ is \hy\ as well.
So any vector of norm $1$ is \hy\ for $T$, and obviously any \nz\ vector $x$ is \hy\ too.
This finishes the proof of Theorem \ref{th1}.

\begin{remark}
The reader may have observed that there is a technical difference in the proofs of Theorems \ref{th2} and \ref{th1} in the way we incorporate the vectors $z_{j}$ (or $z_{j}^{(d)}$, $d\ge 1$) of $Z$ (or $Z^{(d)}$) in the construction. In the proof of Theorem \ref{th2} the important vector at step $n$ is $z_{\kappa_n}$, which appears in the definition of $e_{a_{n}}$; the ``useless'' vectors $z_{j}$ for $\kappa_{n-1}<j<\kappa_{n}$ are incorporated before it, but play no role in the construction. But the advantage of this approach is that the sequence $(\alpha_n)_{n\ge 1}$ is bounded away from $0$.  In the proof of Theorem \ref{th1}, all vectors $z_{j}^{(d)}$ play a similar role and are brought in the construction in the same way. On the other hand, we simply know that the inferior limit of the sequence $(\alpha_n)_{n\ge 1}$ is zero. There is no deep reason for this difference, and we could also have carried out the construction of Theorem \ref{th2} in the same way as that of Theorem \ref{th1}, with at each step $n$ an (a)-\woi\ reduced to the singleton $\{a_{n}\}$. The present proof of Theorem \ref{th2} is however a bit simpler technically, and it also has the advantage of illustrating how one can incorporate useless vectors in a Read's type construction: it is exactly thanks to this kind of procedure that one can prove Theorem \ref{th2} (resp. Theorem \ref{th1}) for spaces of the form 
$Y \oplus\ell_{p}\oplus Z$ or $Y \oplus c_{0}\oplus Z $ (resp. $Y \oplus\bigoplus_{\ell_{p}}Z$ or $Y \oplus\bigoplus_{c_{0}}Z$), where $Y$ is any separable Banach space.
\end{remark}

\section{Some Hilbert space results}\label{sec5}

We begin this section with the proof of Theorem \ref{th3}, which states the following:  if there exists a bounded \op\ $T$ on the complex Hilbert space $\ell_{2}$ which has no eigenvector and is such that, for some  $\varepsilon \in (0,1)$,
the distance of the space $M_{x}=\overline{\textrm{sp}}[T^{n}x \textrm{ ; } n\ge 0]$ to a fixed norm $1$ vector $g_{0}$ is less than $\varepsilon $ 
for every \nz\  vector $x\in\ell_{2}$, then there exists a bounded \op\ $T'$ on $\ell_{2}$ which has no non-trivial \inv\ closed subspace. The assumption of Theorem \ref{th3} is equivalent to supposing that
the distance of any closed invariant subspace
to  $g_{0}$ is less than $\varepsilon $.

\subsection{Proof of Theorem \ref{th3}}

Consider the set $\mathcal{L}_1$ of \nz\ closed subspaces of $\ell_{2}$ which are \inv\ by $T$, and set $d_1=\sup_{M \in \mathcal{L}_1}d(M,g_0)$. By Property (P3), $d_1\leq \varepsilon $. Choose $L_1\in \mathcal{L}_1$ such that $d(L_1,g_0)\geq\frac{1}{2}d_1$. Then consider the set $\mathcal{L}_2$ of \nz\ closed subspaces of $L_1$ which are \inv\ by $T$, and set $d_2=\sup_{M\in \mathcal{L}_2}d(M,g_0)$. Again $d_2\leq \varepsilon $, and since $\mathcal{L}_2\subseteq \mathcal{L}_1$, we have $d_2\leq d_1$. Choose $L_2\in \mathcal{L}_2$ such that $d(L_2,g_0)\geq\frac{3}{4}d_2$. Since $L_2$ is a closed subspace of $L_1$, we have
$d(L_2,g_0)\geq d(L_1,g_0)$. Continuing in this fashion, we construct by induction a decreasing sequence $(L_n)_{n\geq 1}$ of closed subspaces of $\ell_{2}$ such that if $\mathcal{L}_n$ is the set of \nz\ closed subspaces of $L_{n-1}$ which are \inv\ by $T$, and $d_n=\sup_{M\in \mathcal{L}_n}d(M,g_0)$, then $L_n\in \mathcal{L}_n$ is chosen such that $$d(L_n,g_0)\geq\frac{2^n-1}{2^n}d_n.$$ The sequence $(d_n)_{n\geq 1}$ decreases to its limit $\inf_{n\geq 1}d_n$, and $d_n\leq\varepsilon $ for each $n\ge 1$ while the sequence $(d(L_n,g_0))_{n\geq 1}$ is increasing. For every $n\geq 1$ we have $$d_n\geq d(L_n,g_0)\geq \frac{2^n-1}{2^n}d_n\cdot$$
\par\smallskip
For any $n\geq 1$ let $x_n\in L_n$ be such that $$||x_n-g_0||\leq
\frac{2^n+1}{2^n} d(L_{n},g_{0})\leq
\frac{2^n+1}{2^n}d_n.$$ Since the sequence $(x_n)_{n\ge 1}$ is bounded in norm, it admits a sub-sequence $(x_{n_k})_{k\ge 1}$ which converges to some vector $x_0\in \ell_{2}$ in the weak topology of $\ell_{2}$. Since 
\begin{equation}\label{new6}
 ||x_0-g_0||\leq \liminf_{n_k\rightarrow +\infty} ||x_{n_k} -{g_0} ||\leq \liminf_{n_k\rightarrow +\infty} \frac{2^{n_k}+1}{2^{n_k}}d_{n_k}
\end{equation}
 and $d_{n_k}\leq\varepsilon <1$ for every $n\geq 1$, (\ref{new6}) plus the fact that $||g_{0}||=1$  imply that $x_0$ is \nz, and thus $L=M_{x_0}$ is a non-zero closed subspace of $\ell_{2}$ which is \inv\ by $T$.
\par\smallskip
 Let us now compute the distance $d(L, g_0)$ of $L$ to the vector $g_{0}$: since $(L_n)_{n\geq 1}$ is a decreasing sequence of subspaces and $x_{n_k}\in L_{{n_k}}$ for each $k$, we have that for any $n\geq 1$, $x_{n_k}$ belongs to $L_n$ for $k$ large enough. But now, since $L_n$ is in particular weakly closed this implies, making $k$ tend to infinity, that $x_{0}$ belongs to $L_n$. Hence $L\subseteq L_n$ for every $n\geq 1$, and thus
$$d(L,g_0)\geq d(L_n,g_0) \geq \frac{2^n-1}{2^n}d_n.$$ But since $x_{0}\in L_{n}$ we also have $L\in \mathcal{L}_n$ for every $n\geq 1$, so $d(L,g_0)\leq d_n$. Making $n$ tend to infinity yields that $d(L,g_0)=\inf_{n\geq 1}d_n=\lim _{n\rightarrow +\infty}d_n$. Plugging this into (\ref{new6}), we get that $||{x_{0}} -{g_0} ||\le d(L, g_0)$. Also we have that $||{x_{0}} -{g_0} ||\ge d(L, g_0)$ (since $x_{0}\in L$), and thus $||{x_{0}} -{g_0} ||= d(L, g_0)$.
\par\smallskip
 We are now going to show that the \op\ induced by $T$ on $L$ has no \nt\ \inv\ closed subspace. 
Suppose that $M$ is a \nz\ closed \inv\ subspace of $L$. Since $M\subseteq L$ we have $d(M,g_{0})\ge d(L,g_{0})$. But $M\subseteq L_{n}$ for every $n$, so $M\in \mathcal{L}_{n}$, and thus $d(M,g_{0})\le d_{n}$ for each $n\ge 1$. It follows from this that
$d(M, g_0)=\inf_{n\geq 1}d_n$, i.e. $d(L, g_0)=d(M, g_0)$. Let now $y_0\in M$ be such that
$||y_0-g_0||=d(M, g_0)=d(L, g_0)=||x_0-g_0||$. Then $$||\frac{x_0+y_0}{2}-g_0||\leq d(L, g_0).$$ But now since $y_0$ belongs to $M$ which is a closed subspace of $ L$, $\frac{x_0+y_0}{2}$ belongs to $L$, and as a consequence we get that  $$||\frac{x_0+y_0}{2}-g_0||\ge d(L,g_0).$$
Hence $$||\frac{x_0+y_0}{2}-g_0||=||{x_0}-g_0||=||{y_0}-g_0||= d(L,g_0).$$
Since the norm on $\ell_{2}$ is strictly convex, $x_0$ and $y_0$ must be equal. Hence $M_{x_0}=M_{y_0}=L$, and since $M$ contains $M_{y_0}$ we conclude that $M=L$.
\par\smallskip
Hence we have shown that the \op\ induced by $T$ on $L$ has no \nt\ \inv\ closed subspace.
So either $L$ is a one-dimensional space, which is ruled out by our assumption that $T$ has no eigenvector, or $L$ is infinite-dimensional. Hence $L$ is isometric to $\ell_{2}$, and this finishes the proof of Theorem \ref{th3}.

\subsection{Proof of Corollary \ref{cor3bis}} It suffices to observe that the assumption that Property (P3) holds true prevents the \op\ $T$ from having eigenvectors. Indeed, suppose that $x_{0}$ is an eigenvector of $T$, with $Tx_{0}=\lambda x_{0}$ for some scalar $\lambda $. The set $\Gamma  =\{\alpha \in\C \textrm{ ; } ||\alpha x_{0}-g_{0}||\le \varepsilon \}$ is non-empty (by Property (P3)), compact, and does not contain $0$. It follows that there exists an $R>0$ such that $\frac{1}{R}\le |\alpha |\le R$ for any $\alpha \in \Gamma  $. Property (P3) applied to the vectors $\mu x_{0}$, $\mu \in\C\setminus\{0\}$, implies that for each $\mu \not =0$ there exists an $n\ge 0$ such that
$\frac{1}{R}\le |\mu \lambda ^{n}|\le R$. Taking $|\mu |>R$ implies that $|\lambda |< 1$, while taking $|\mu |<\frac{1}{R}$ implies that $|\lambda |>1$, which is contradictory. Hence $T$ has no eigenvector, and Corollary \ref{cor3bis} follows from Theorem \ref{th3}.

\subsection{Proof of Theorem \ref{th3bis}}
The proof of Theorem \ref{th3bis} is a mix of the techniques used in the proofs of Theorems \ref{th1} and \ref{th2}. The \op\ $T$ is constructed on the space $H=\ell_{2}\oplus\ell_{2}$. We denote by $(g_{j})_{j\ge 0}$ the canonical basis of the first copy of the space $\ell_{2}$, and by $(z_{j})_{j\ge 1}$ the canonical basis of the second copy of $\ell_{2}$. The sequence 
 $(z_{j})_{j\ge 1}$ will play the role of the non semi-boundedly complete basis of $Z$ in the proof of Theorem \ref{th2}, but of course it is now boundedly complete in $\ell_{2}(\N)$. Fix $\varepsilon \in (0,1)$, and for each $j\ge 1$ denote by $\alpha _{j}$ the scalar $\alpha _{j}=\frac{\sqrt{3}\varepsilon }{\pi}\,\frac{1}{j}$. We have
 $$\Bigl|\Bigl|\sum_{n\ge 1}\alpha _{j}z_{j}\Bigr|\Bigr|=\left(\sum_{j\ge 1}\frac{3\varepsilon ^{2}}{\pi^{2}}\,\frac{1}{j^{2}}\right)^{\frac{1}{2}}=\frac{\varepsilon }{\sqrt{2}}<\varepsilon .$$
Set $\kappa_{j}=j$ for each $j\ge 1$.
With these choices of the sequences $(z_{j})_{j\ge 1}$, $(\alpha _{j})_{j\ge 1}$ and $(\kappa_{j})_{j\ge 1}$, we construct the \op\ $T$ exactly as in the proof of Theorem \ref{th2} (in particular $e_{0}=g_{0}$). Then $T$ is a bounded \op\ on $H$, and we have the following obvious fact:

\begin{fact}\label{factnewA}
For every $n\geq 1$ we have
$$e_{a_{n}}=\sum_{j=1}^{n}\alpha _{j}z_{{j}}+e_{0},\quad\textrm{so that}\quad||e_{a_{n}}-e_{0}||<\varepsilon .$$ The sequence $(e_{a_{n}})_{n\ge 1}$ converges in norm in $H$ to the vector $u_{0}+e_{0}$, where $$u_{0}=\frac{\sqrt{3}\varepsilon }{\pi}\sum_{j\ge 1}\frac{1}{j}z_{j}.$$
\end{fact}

The convergence of the sequence $(e_{a_{n}})_{n\ge 1}$ is the most crucial difference with what happens in the proof of Theorem \ref{th2}. Since $||e_{a_{n}+1}||\ls 2^{-\frac{1}{2}\sqrt{b_{n}}}$ for each $n\ge 1$, it follows that $T(u_{0}+e_{0})=0$. The vector $u_{0}+e_{0}$ being obviously \nz, $T$ is not an injective \op.
\par\smallskip
Let us now review the various steps in the proof of Theorem \ref{th2}, and see whether they still hold true or not in this new context. Fact \ref{factb} is obviously still true, while Fact \ref{factbis} is wrong since the sequence $(\alpha _{j})_{j\ge 1}$ is not bounded away from zero. But the same proof shows that

\begin{fact}\label{factnewB}
There exists a positive constant $M$ such that $||Q_{\nu _{n}}||\le n\, M$ for each $n\ge 1$.
\end{fact}

The fact that the norm of $Q_{\nu _{n}}$ is a bit larger than previously does not change anything to the relevant parts  of the proof (actually, and this was used in the proof of Theorem \ref{th1}, the only important thing is that the norm of $Q_{\nu _{n}}$ is controlled for each $n\ge 1$ by a quantity which depends only on $a_{n}$). Thus, going again through the  proof, we see that Proposition \ref{prop2} still holds true in this Hilbertian setting. 

Now we come to an important difference between the non-reflexive and the Hilbertian cases, which is that, for a vector $x\in H$, $Q_{\nu _{n}}x$ does not necessarily tend to $x$ as $n$ tends to infinity. Indeed, since
\begin{eqnarray*}
 (I-Q_{\nu _{n}})x&=&\frac{1}{\alpha _{n+1}}f^{*}_{a_{n+1}}(x)e_{a_{n+1}}+\sum_{j>\nu _{n},\, j\not =a_{n+1}}f_{j}^{*}(x)f_{j}\\
 &=& \frac{\pi}{\sqrt{3}\varepsilon }\, (n+1) f^{*}_{a_{n+1}}(x)e_{a_{n+1}}+\sum_{j>\nu _{n},\, j\not =a_{n+1}}f_{j}^{*}(x)f_{j},
\end{eqnarray*}
$(I-Q_{\nu _{n}})x$ converges to zero \ifff\ $ (n+1)| f^{*}_{a_{n+1}}(x)|$ tends to zero as $n$ tends to infinity, which is of course not always the case (actually, the important point is that $\inf_{n\ge 1}||(I-Q_{\nu _{n}})x||$ may be strictly positive). So Facts \ref{factc} and \ref{factd} are not true anymore here. This
is a difficulty we already faced in the proof of Theorem \ref{th1}, and so at this point we shift from the proof of Theorem \ref{th2} to the proof of Theorem \ref{th1}, and observe that Proposition \ref{prop5bis} holds true here. Also, it is not difficult to check, thanks to the remark above, that if $x\in H$ is a vector such that $ (n+1)| f^{*}_{a_{n+1}}(x)|$ tends to zero as $n$ tends to infinity, then the closure of the orbit of $x$ under the action of $T$ coincides with the closure of the linear span of this orbit. In other words, Property (P1) is true for these vectors $x$. In particular, the vector $e_{0}=g_{0}$ is still a \hy\ vector for $T$.

The most important step in the proof of Theorem \ref{th3bis} is the following proposition:

\begin{proposition}\label{propnewC}
If $(A_{n})_{n\geq 1}$ is a  sequence of positive real numbers such that each $A_{n}$ is sufficiently large \wrt\ $a_{n}$, the following statement holds true:

for any vector $x$ in $H$ with $||x||=1$  which is not colinear to $x_{0}=u_{0}+e_{0}$, and for any integer $n_{0}\geq 1$, there exist an $n\geq n_{0}$ and a $j\in [0,a_{n}-1]$ such that
$$|e_{j}^{*(a_n)}(Q_{a_{n}}x)|> \frac{1}{A_{n}^{(a_n-j+1)!^2}}\cdot$$
\end{proposition}

It is important to remark that this analogue of Proposition \ref{prop5} cannot hold true for \emph{all} vectors $x$ of norm $1$: if we choose $x_{1}=x_{0}/||x_{0}||$, then $$\pi_{[0,a_{n}]}x_{0}=\sum_{j=1}^{n}\alpha _{j}z_{j}+e_{0}=e_{a_{n}},$$ so that $Q_{a_{n}}x_{1}=0$. So the conclusion of Proposition \ref{propnewC} does not hold true for the vector $x_{1}$. But these are the only vectors for which such a situation can take place.

\begin{proof}[Proof of Proposition \ref{propnewC}]
 The beginning of the proof is the same as that of Proposition \ref{prop5}. Suppose that $x\in H$, with $||x||=1$, is such that for every $n\ge n_{0} $ and every $j\in [0,a_{n}-1]$,
 $$|e_{j}^{*(a_n)}(Q_{a_{n}}x)|\leq \frac{1}{A_{n}^{(a_n-j+1)!^2}}\cdot$$ The same argument as in the proof of Proposition \ref{prop5} shows then that the quantity $||\pi_{[0,a_{n}]}x-e_{a_{n}}^{*(a_{n})}(x)e_{a_{n}}||$ tends to $0$ as $n$ tends to infinity, i.e. that $e_{a_{n}}^{*(a_{n})}(x)e_{a_{n}}$ tends to 
 $x$. Since $f_{0}^{*}(e_{a_{n}})=1$ for each $n\ge 1$, $e_{a_{n}}^{*(a_{n})}(x)$ tends to $f_{0}^{*}(x)=\gamma $; moreover, $e_{a_{n}}$ tends to $x_{0}$ as $n$ tends to infinity, and so it follows that $x=\gamma x_{0}$. So if $x$ is not colinear to $x_{0}$, there exists a $j\in [0,a_{n}-1]$ such that
 $$|e_{j}^{*(a_n)}(Q_{a_{n}}x)|>\frac{1}{A_{n}^{(a_n-j+1)!^2}},$$
 and we are done.
\end{proof}

We can now come back to the proof of Theorem \ref{th3bis}. We have to be a bit careful here about the fact that we cannot suppose to start with that $x\in X$ has norm $1$: since Property (P1) is not true here, it may happen that the closures of the orbits of $x$ and $\frac{x}{||x||}$ do not coincide. So let $x\in H$ be a \nz\ vector which is not colinear to $x_{0}=u_0+e_{0}$, and let $n$ be so large that $2^{-n}\le ||x||\le 2^{n}$. We follow the proof of either Property (P2) in Section \ref{sec3} or
Proposition \ref{prop7bis}: let $n $ arbitrarily large be such that there exists a  $j\in [0,a_{n}-1]$ with
$$|e_{j}^{*(a_n)}(Q_{a_{n}}x)|> \frac{1}{A_{n}^{(a_n-j+1)!^2}}\,||x||\cdot$$
We denote by $j_{n}$ the smallest integer with this property. Then there exists a polynomial $p_{n}$ of degree at most $a_{n}$ with 
$|p_{n}|\le D_{a_{n}}A_{n}^{(a_n-j_{n}+1)!^2(a_n-j_{n}+1)}$ such that
$$p_{n}(T_{a_{n}})\left(\sum_{j=j_{n}}^{a_{n}} e_{j}^{*(a_{n})}\left(Q_{a_{n}}\frac{x}{||x||}\right)e_{j}\right)=e_{a_{n}-1}.$$
 If we replace the polynomial $p_{n}$ by $\frac{p_{n}}{||x||}$, the degree of this new polynomial remains the same, and its modulus is less than $2^{n}D_{a_{n}}A_{n}^{(a_n-j_{n}+1)!^2(a_n-j_{n}+1)}$. So replacing $D_{a_{n}}$ by $2^{n}D_{a_{n}}$ in the computations, we can run through exactly the same argument and obtain 
 (since Fact \ref{factg}
is still true here) a $k\in [1,k_{n}]$ such that
$$||p_{k,n}(T)(Q_{a_{n}}x)-e_{a_{n}}||<\frac{4}{a_{n}}$$
(this is (\ref{star}) applied with $y=e_{a_{n}-1}$ and the bound $\frac{4}{a_{n}}$ instead of $\frac{8}{a_{n}}$ in the proof of Theorem \ref{th2}, or (\ref{eqi}) in the proof of Theorem \ref{th1}). As previously, the polynomial $p_{k,n}$ is chosen so that $|p_{k,n}-q_{n}|<4^{-\nu_{n}}$, where $q_{n}$ is the polynomial defined by $q_{n}(\zeta )=\frac{\zeta ^{b_{n}+1}}{b_{n}}p_{n}(\zeta )$.
Combining this with  Proposition \ref{prop5bis}, we infer in the same manner as in the proof of Theorem \ref{th1} that
$$||T^{c_{k,n}}x-e_{a_{n}}||<\frac{5}{a_{n}}+103\,||x-\pi_{[0,\nu _{n}]}x||+||p_{k,n}(T)\pi_{[a_{n}+1, \nu_{n}]}x||$$
(this is the analogue of (\ref{eqj})).
Using
Proposition \ref{prop3} (which is still true here), we get, again as in the proof of Theorem \ref{th1}, that
$$||T^{c_{k,n}}x-e_{a_{n}}||<\frac{7}{a_{n}}+103\,||x-\pi_{[0,\nu _{n}]}x||.$$
 It follows that the vector $x_{0}$, which is the limit as $n$ tends to infinity of the vectors $e_{a_{n}}$, belongs to the closure of the orbit of $x$ under the action of $T$. Since $||u_{0}||<\varepsilon $, the distance of $x_{0}=u_{0}+e_{0}$, and hence of $\overline{\mathcal{O}rb}(x,T)$, to $g_{0}=e_{0}$ is less than $\varepsilon $, and Property (P3'') of Theorem \ref{th3bis} is proved.
If $x=\kappa x_{0}$ for some scalar $\kappa $, it is obvious that the span of the orbit of $x$, which is the line $\C.x_{0}$, comes within $\varepsilon $-distance of the vector $g_{0}$. This finishes the proof of Theorem \ref{th3bis}.

\begin{remark}
The proof of Theorem \ref{th3bis} shows that any \nt\ \inv\ subspace of the \op\ $T$ contains the linear span of the vector $x_{0}$. It is not a new fact that the lattice of \inv\ subspaces of a bounded \op\ on the Hilbert space can have such a property: Donoghue constructed in \cite{Do} compact weighted backward shifts on $\ell_{2}(\mathbb{N})$ (equipped with the canonical basis $(g_{j})_{j\ge 0}$) whose only \nt\ \inv\  subspaces were the finite-dimensional spaces $M_{j}=\textrm{sp}[g_{0},\ldots, g_{j}]$, $j\ge 0$. So in this case each \nt\ \inv\ subspace contains the vector $g_{0}$. One of the main interests of Theorem \ref{th3bis} is that for all \nz\ vectors $x$ except those which are colinear to the vector $x_{0}$, with no restriction whatever on the norm of $x$, the closure of the orbit of $x$ (and not only the space $M_{x}$ spanned by this orbit) contains $x_{0}$.
\end{remark}

\section{Operators without non-trivial invariant $w^{*}$-closed subspaces  on some quasi-reflexive spaces of order $1$}\label{sec6}

Recall that a Banach space $X$ is said to be \emph{quasi-reflexive} of order $m$, $1\leq m<+\infty $, when $X$ is a subspace of codimension $m$ of its bidual $X^{**}$. The most famous example of a \qr\ space is the James space $J$ \cite{J}: it consists of the space of sequences
$x=(x_{j})_{j\geq 0}$ of $c_{0}$ such that
$$||x||=\sup_{m\geq 1,\; 0\leq p_{1}<\ldots<p_{m}} \left(\sum_{j=1}^{m}|x_{p_{j}}-x_{p_{j+1}}|^{2}\right)^{\frac{1}{2}}<+\infty.$$ The sequence $(h_{j})_{j\geq 0}$ where $h_{j}=(\delta _{j,m})_{m\geq 0}$ forms a basis of $J$ which is shrinking (the coordinate functionals $(h_j^*)_{j\geq 0}$ form a basis of the dual space $J^*$) but not boundedly complete: $||h_{0}+\ldots+h_{j}||=1$ for every $j\geq 0$ but the series $\sum_{j\geq 0}h_{j}$ does not converge in $J$. We have $J^{**}=J\oplus \textrm{sp}[x_{0}^{**}]$, where $x_{0}^{**}=\sum_{j\geq 0}h_{j}$. Notice also that $J$ contains a complemented copy of $\ell_{2}$, so that $J$ is isomorphic to $\ell_{2}\oplus J$.

\subsection{Proof of Theorem \ref{th4}}
Theorem \ref{th4} states that a counterexample to the Invariant Subspace Problem on a quasi-reflexive space of order $1$ would give a counterexample living on a reflexive space. This result relies on the following statement, which highlights the role of weak compactness in the problem:

\begin{proposition}\label{propnewD}
 Suppose that there exists a separable Banach space 
 $Z$ and a bounded \op\ $T$ on $Z$ which is weakly compact and has no \nt\ \inv\ closed subspace (resp. subset). Then there exists a reflexive separable Banach space $X$ and a bounded \op\ $T_{0}$ on $X$ which has no \nt\ \inv\ closed subspace (resp. subset).
\end{proposition}

\begin{proof}The proof of Proposition \ref{propnewD} is a straightforward consequence of a 
 classical result of Davis, Figiel, Johnson and Pelczynski \cite{DFJP}, which states that the weakly compact \op\ $T$ on $Z$ can be factorized through a \re\ space $X$: there exist bounded linear \ops\ $A: Z\To X$ and
$B:X\To Z$ such that $T=BA$. We can suppose that $A$ has dense range (this is obvious), and also
that $B$ is injective. Indeed if $B$ is not injective, let $X_0$ be the kernel of $B$, and define $A': Z\To X/X_0$ and
$B':X/X_0\To Z$ by the formula $A'z=Az+X_0$, $z\in Z$, and $B'(x+X_0)=Bx$, $x\in X$, respectively: $B'$ is well-defined and injective. Also $B'A'z=B(Az+X_0)=Tz$, and thus we have got a factorization of $T$ through a reflexive space, but this time with $B'$ injective. So we can suppose that $B$ is injective. Now consider the bounded \op\ $T_0$ on $X$ defined by $T_0=AB$. We have for every integer $n\geq 0$ and every vector $x\in X$ the identity $T_0^{n+1}x=AT^n Bx$, so that
${\mathcal{O}\textrm{rb}}(x,T_0)=A({\mathcal{O}\textrm{rb}}(Bx,T))\cup \{x\}$. If $x$ is \nz, $Bx\in Z$ is \nz\ as well.
Suppose that $T$ has no \nt\ \inv\ closed subspace: every \nz\ vector $z\in Z$ is cyclic for $T$, hence the linear span of the orbit 
${\mathcal{O}\textrm{rb}}(Bx,T_0)$ is dense in $Z$ for every \nz\ vector $x\in X$. Since $A$ has dense range, it follows that
the linear span of the orbit ${\mathcal{O}\textrm{rb}}(x,T_0)$ is dense in $X$ for every such $x$. So $T_{0}$ has no \nt\ \inv\ closed subspace. Exactly the same kind of argument applies if we are considering \inv\ closed sets instead of \inv\ subspaces.
\end{proof}

Theorem \ref{th4} is a direct corollary of Proposition \ref{propnewD} and the fact that every bounded \op\ $T$ on a quasi-reflexive space 
$Z$ of order $1$ has a unique decomposition as $\lambda I+W$, where $\lambda $ is a scalar and $W$ is a weakly compact \op\ on $Z$ (see the paper  \cite{FL} by Fonf, Lin and  Wojtaszczyk). If $T$ has no \nt\ \inv\ closed subspace, $W$ is  weakly compact and has no \nt\ \inv\ closed subspace either. Proposition \ref{propnewD} concludes the proof.

\begin{remark}
 As we observed in Remark \ref{remaddbis} above, the \ops\ $T$ of Theorem \ref{th2} constructed on $X=\ell_{p}\oplus Z$ where $1<p<+\infty$ are weakly compact, but they may - and they indeed do, as will be seen in Section \ref{sec7} - have some \nt\ \inv\ closed subspaces. On the other hand, the \ops\ constructed in the proof of Theorem \ref{th1} are never weakly compact. This comes from the fact that, for most values of the integer $d$, $T$ acts as a shift from the $d^{th}$ copy $Z^{(d)}$ of $Z$ onto the next copy $Z^{(d+1)}$.
\end{remark}

\subsection{Proof of Theorem \ref{th5}}
Our aim in this section is to prove Theorem \ref{th5}, which states that
there exists a \qr\ space $X$ of order  $1$ whose dual $X^{*}$, which is also \qr\ of order $1$, supports a bounded \op\ $T$ with
no \nt\ \inv\ $w^{*}$-closed subspace.
\par\smallskip

Let us begin by recalling several well-known facts about \qr\ spaces of order $1$, of which we give short proofs for completeness's sake. Fact \ref{fact1} below is a special case of a result of Civin and Yood \cite{CY}.

\begin{fact}\label{fact1}
 Let $X$ be a \qr\ space of order $1$, and $M$ a closed subspace of $X$. Then

$\bullet$ $M$ is either \re\ or \qr\ of order $1$;

$\bullet$ $X/M$ is either \re\ or \qr\ of order $1$.
\end{fact}

\begin{proof}
Let $x_{0}^{**}\in X^{**}\setminus X$, $||x_{0}^{**}||=1$. Then $X^{**}=X\oplus \textrm{sp}[x_{0}^{**}]$.
The bidual $M^{**}$ of $M$ can be isometrically identified to $$M^{\perp\perp}=\{x^{**}
\in X^{**} \textrm{ ; for every } x^{*}\in X^{*} \textrm{ such that } x^{*}=0
\textrm{ on } M, \; \pss {x^{**}} {x^{*}}=0\}.$$ Suppose that $M$ is not \re, and let $x_{1}^{**}\in M^{\perp\perp}\setminus M$: since $M=M^{\perp\perp}\cap X$, $x^{**}_{1}\not \in X$. Let us show that $M^{\perp\perp}=M\oplus \textrm{sp}[x_{1}^{**}]$. There exists a scalar $\alpha _{1}\not =0$ and $x_{1}\in X$ such that $x^{**}_{1}=\alpha _{1}x_{0}^{**}+x_{1}$. If $x^{**}$ is an element of $ M^{\perp\perp}\setminus M$, then there exist a scalar $\alpha $ and a vector $x\in X$ such that
$x^{**}=\alpha x_{0}^{**}+x$. Hence $x^{**}=\frac{\alpha }{\alpha _{1}}x_{1}^{**}+x-\frac{\alpha }{\alpha _{1}}x_{1}$, and since $x^{**}$ and $x^{**}_{1}$ belong to $M^{\perp\perp}$, $x-\frac{\alpha }{\alpha _{1}}x_{1}$ belongs to $M^{\perp\perp}\cap X=M$. Hence $M^{\perp\perp}=M\oplus \textrm{sp}[x_{1}^{**}]$ and the first assertion is proved.
\par\smallskip
The proof of the second assertion is exactly similar, using the fact that $X^{**}/M^{\perp\perp}$ can be identified to $(X/M)^{**}$  via the application which associates $x^{**}+M^{\perp\perp}\in X^{**}/M^{\perp\perp} $  to the functional on $M^{\perp}=(X/M)^{*}$
which maps $x^{*}$ onto $\pss {x^{**}} {x^{*}} $. With this identification, elements of $X/M$ correspond to functionals on $M^{\perp}$ which map $x^{*}$ onto $\pss {x} {x^{*}} $ for some $x\in X$. Hence if $x_{0}^{**}\in M^{\perp\perp}$, $X/M$ is \re, while if $x_{0}^{**}\not \in M^{\perp\perp}$, $X/M$ is \qr\ of order $1$.
\end{proof}

Our second fact concerns duals of \qr\ spaces or order $1$:

\begin{fact}\label{fact2}
 If $X$ is a \qr\ space of order $1$, so is its dual $X^{*}$.
\end{fact}

\begin{proof}
Let us write again $X^{**}=X\oplus \textrm{sp}[x_0^{**}]$ with $x_{0}^{**}\in X^{**}\setminus X$ and $||x_{0}^{**}||=1$. Let $x_0^{***}\in X^{***}$ be such that $x_0^{***}=0$ on $X$ and $\pss {x_0^{***}} {x_0^{**}} =1$. Then for any $x^{***}\in X^{***}$ we have $x^{***}=x^{*}+\pss {x^{***}} {x_0^{**}} x_0^{***}$ where $x^{*}$ is the restriction of $x^{***}$ to $X$, i.e. an element of $X^{*}$. Obviously $x_0^{***}\not \in X^{*}$, and thus
$X^{**}=X^{*}\oplus \textrm{sp}[x_0^{***}]$, so that $X^{*}$ is \qr\ of order $1$.
\end{proof}

As a consequence of Facts \ref{fact1} and \ref{fact2}, we obtain:

\begin{fact}\label{fact3}
 If $X$ is a \qr\ space of order $1$, and $M$ is a $w^{*}$-closed subspace of $X^{*}$, then $M$ is a dual space which is either reflexive or \qr\  of order $1$.
\end{fact}

\begin{proof}
Since $M$ is \ww, it can be isometrically identified to $(X/{\empty^{\perp}M})^{*}$ where
$${\empty^{\perp}M}=\{x\in X\textrm{ ; } \pss {x^*} {x} =0 \textrm{ for every } x^*\in M\}.$$ By
Facts \ref{fact1} and \ref{fact2},  $(X/{\empty^{\perp}M})^{*}$ is either reflexive or \qr\ of order $1$.
\end{proof}

We finally mention an elementary result on renorming of separable Banach spaces. It can be found in any classical book on the subject, such as \cite{DGZ} or \cite{FHHM}.

\begin{fact}\label{fact5}
Let $X$ be a \sep\ Banach space. There exists an equivalent norm on $X$ whose dual norm
 on $X^{*}$ is strictly convex.
\end{fact}

We are now ready for the proof of Theorem \ref{th5}.

\begin{proof}[Proof of Theorem \ref{th5}]
 Consider the space $X=\ell_2\oplus J$, where $J$ is the James space which we renorm according to Fact \ref{fact5} in such a way that the dual norm on $J^{*}$ is strictly convex. We denote by $||\,.\,||_{*}$ this new norm. Then $X^{*}=\ell_2\oplus J^{*}$ is \nr, has a Schauder basis and contains a complemented copy of $\ell_2$ (of course $X^{*}$ is isomorphic to $J^{*}$). 
 The norm on $X^{*}$ is defined by $||x\oplus z||=\left(||x||^{2}+||z||_{*}^{2}\right)^{\frac{1}{2}}$ for every $x\in \ell_{2}$ and $z\in J^{*}$. This norm on $X^{*}$ is strictly convex. Denoting as usual by $(g_{j})_{j\ge 0}$ the canonical basis of $\ell_{2}$, we have $||g_{0}||=1$. 
 We can apply Theorem \ref{th2} and get a bounded \op\ on $X^{*}$ which satisfies Properties (P1), (P2) and (P3) for a fixed $\varepsilon\in (0,1)$. Then Property (P4) is satisfied too.
\par\smallskip
Our aim is to construct a \nz\ \inv\ closed subspace $L$ of $X^{*}$ such that the \op\ induced by $T$ on $L$ has no \nt\ \inv\ $w^{*}$-closed subspace. 
 For this we proceed in the same way as in the proof of Theorem \ref{th3}, replacing the weak topology by the $w^{*}$-topology. 
\par\smallskip

Consider the set $\mathcal{L}_1$ of \nz\ $w^{*}$-closed subspaces of $X^{*}$ which are $T$-\inv, and set $d_1=\sup_{M\in \mathcal{L}_1}d(M,g_0)$. We have $d_1\leq \varepsilon $. Choose $L_1\in \mathcal{L}_1$ such that $d(L_1,g_0)\geq\frac{1}{2}d_1$. 
We continue then as in the proof of Theorem \ref{th3}, and construct by induction a decreasing sequence $(L_n)_{n\geq 1}$ of $w^{*}$-closed subspaces of $X^{*}$ such that if $\mathcal{L}_n$ is the set of all \nz\ $w^{*}$-closed subspaces of $L_{n-1}$ which are \inv\ by $T$, and $d_n=\sup_{M\in \mathcal{L}_n}d(M,g_0)$, then $L_n\in \mathcal{L}_n$ is chosen in such a way that $$d(L_n,g_0)\geq\frac{2^n-1}{2^n}d_n.$$ If, for each $n\geq 1$, $x_n^{*}\in L_n$ is such that $$||x_n^{*}-g_0||\leq
\frac{2^n+1}{2^n}d_n,$$ the sequence $(x_n^{*})_{n\ge 1}$ is bounded. Hence there exists a sub-sequence $(x_{n_k}^{*})_{k\ge 1}$ of $(x_n^{*})_{n\ge 1}$ which converges to some vector $x_0^{*}\in X^{*}$ in the $w^{*}$-topology. The same argument as in the proof of Theorem \ref{th3} (i.e. that $d_{n}\le\varepsilon<1$ for every $n\ge 1$) shows that $x_0^{*}$ is \nz. Let $L$ be the smallest $T$-\inv\ $w^{*}$-closed subspace of $X^{*}$ containing the vector $x_{0}^{*}$. By Property (P4), $L$ is an infinite-dimensional subspace of $X^{*}$.
Since, for each $n\ge 1$, $x_{n_k}^{*}$ belongs to $L_n$ for sufficiently large $k$, the fact that $L_{n}$ is $w^{*}$-closed implies that  $x_{0}^{*}$ belongs to $L_n$.
It follows from this that  $||{x_{0}^{*}} -{g_0} ||=\lim _{n\rightarrow +\infty}d_n=
d(L,g_0)=\inf_{n\geq 1}d_n$.
\par\smallskip
It remains to show that the \op\ induced by $T$ on $L$ has no \nt\ \inv\ $w^{*}$-closed subspace. 
Suppose that $M\subseteq L$ is a \nz\ $w^{*}$-closed subspace of $L$ which is $T$-\inv: then, as in the proof of Theorem \ref{th3},
$d(M, g_0)=\inf_{n\geq 1}d_n=d(L, g_0)$.
Since the distance of an element to a $w^{*}$-closed subspace is always attained, there exists a vector 
$y_0^{*}\in M$ such that
$||y_0^{*}-g_0||=d(M, g_0)=d(L, g_0)=||x_0^{*}-g_0||$.
It follows from this that $||\frac{x_0^{*}+y_0^{*}}{2}-g_0||= d(L,g_0)$ as well.
Now, since the norm $||\,.\,||$ on $X^{*}$ is strictly convex,  we get that $x_0^{*}=y_0^{*}$. So $M=L$.
\par\smallskip
Hence we have shown that the \op\ induced by $T$ on $L$ has no \nt\ $w^{*}$-closed \inv\ subspace.
In order to  finish the proof of Theorem \ref{th5}, we still have to show that the space $L$ cannot be reflexive. For this we have to go back to the structure of the \op\ $T$ acting on $X^{*}$. The following general statement immediately implies that the subspace $L$ of $X^{*}$ we just constructed cannot be reflexive:

\begin{prop}\label{facti}
Let $X=\ell_{p}\oplus Z$, $1\le p<+\infty$, or $X=c_{0}\oplus Z$, be a  space satisfying the assumptions of Theorem \ref{th2}, and let $T$ be one of the \ops\ on $X$ constructed in the proof of Theorem \ref{th2}. Then all \nz\ \inv\ subspaces of $T$ are non-\re.
\end{prop}

\begin{proof}[Proof of Proposition \ref{facti}]
 By Property (P4) any \nz\ \inv\ subspace of $X$ is of the form $M_x$ for some \nz\ vector $x\in X$. The proof of Theorem \ref{th2} (more precisely, equation (\ref{rond}) combined with Fact \ref{factc}) shows that there exists an infinite sequence $(n_l)_{l\geq 1}$ of integers such that $d(e_{a_{n_l}},M_x)$ tends to $0$ as $l$ tends to infinity. Since the sequence $(e_{a_{n_l}})_{l\geq 1}$ is bounded, we can, if we consider it as living in $X^{**}$, extract from it a $w^{*}$-convergent subsequence in $X^{**}$. Without loss of generality we suppose that $e_{a_{n_l}}$ tends  to some element 
$x_1^{**}\in X^{**}$ in this $w^{*}$-topology. Now, $M_x^{\perp\perp}$ being $w^{*}$-closed in $X^{**}$, $x_1^{**}$ belongs to $M_x^{\perp\perp}$. In order to prove that $M_x$ cannot be reflexive, it suffices to prove that $x_1^{**}$ cannot belong to $X$. Suppose that $x_1^{**}=x_1$ with $x_{1}\in X$. The $w^{*}$-convergence of $e_{a_{n_{l}}}$ to $x_{1}$ in $X^{**}$ implies that  for every $j\geq 1$, $\pss {z_j^{*}} {e_{a_{n_l}}} $ tends to $\pss {z_j^{*}} {x_1} $  as $l$ tends to infinity (recall that $(z_{j}^{*})_{j\ge 1}$ is the sequence of coordinate functionals associated to the basis $(z_{j})_{j\ge 1}$ of $Z$). Thus, since $e_{a_{n_{l}}}=\sum_{k=1}^{n_{l}}\alpha _{k}z_{\kappa _{k}}+e_{0}$, we get that $\pss {z_j^{*}} {x_1} =0$ for every $l$ not belonging to the set $\{\kappa_k\textrm{ ; } k\geq 1\}$
while $\pss {z_{\kappa_k}^{*}} {x_1} =\alpha_{k}$ for every $k\geq 1$. Hence the projection of the vector $x_{1}$ onto the space $Z$ can only be equal to the vector $\sum_{k\geq 1}\alpha_{k}z_{\kappa_k}$. But this contradicts the fact that the series $\sum_{k\geq 1}\alpha_{k}z_{\kappa_k}$ does not converge in $Z$. Hence  $x_1^{**}$ does not belong to $M_x$, and $M_x$ is non-\re.
\end{proof}
Thus Proposition \ref{facti} completes the proof of Theorem \ref{th5}.
\end{proof}

\section{Final remarks and comments}\label{sec7}

We finish this paper with several remarks concerning the structure of the Read's type \ops\ constructed here.

\subsection{Read's type operators are never adjoints}
 An argument similar to the one employed in the proof of Fact \ref{facti} (and also to the one used by Schlumprecht and Troitsky in \cite{TS} to show that the \ops\ of \cite{R1} are never adjoints of \ops\ on a predual of $\ell_{1}$) shows that:

\begin{proposition}
  The \ops\ constructed in the proofs of Theorems \ref{th2} and \ref{th1} can never be adjoints of \ops\ on some predual of the space $X$. 
\end{proposition}

\begin{proof}
Indeed suppose that $X=Y^{*}$ is a dual space satisfying the assumptions of Theorem \ref{th2}, and let $T$ be one of the \ops\ constructed in the proof. Suppose that $T$ is the adjoint of some \op\ $S$ on $Y$. Then $T$ is $w^{*}$-$w^{*}$ continuous on $X$. Consider the bounded sequence $(e_{a_{n}})_{n\ge 1}$ of vectors in $X$, and let $x_{0}\in X$ be one of its $w^{*}$-limit points. Then, since $T$ is $w^{*}$-$w^{*}$ continuous, $Tx_{0}$ is a  $w^{*}$-limit point of the sequence $(Te_{a_{n}})_{n\ge 1}$. But $Te_{a_{n}}=e_{a_{n}+1}$ tends to zero in norm as $n$ tends to infinity, so that necessarily $Tx_{0}=0$. Applying Property (P4) to the invariant subspace $M_{x_{0}}$ (see Proposition \ref{prop6}), we infer that $x_{0}=0$, i.e. that any $w^{*}$-limit point of the sequence $(e_{a_{n}})_{n\ge 1}$ must be zero.
This stands in contradiction with the fact that $\pss{e_{0}^{*}}{e_{a_{n}}}=1$ for each $n\ge 1$. The argument is exactly the same for the \ops\ constructed in the proof of Theorem \ref{th1}, except that one now considers $w^{*}$-limit points of the sequence $(e_{na_{n}})_{n\ge 1}$.
 \end{proof}

\subsection{Invariant subspaces of the operators obtained in Theorem \ref{th2}}
One might wonder whether the \ops\ constructed in the proof of Theorem \ref{th2} for $p$ in the reflexive range $1<p<+\infty$  have \nt\ \inv\ subspaces or not. It is indeed the case:

\begin{proposition}\label{propbizarre}
 Let $X=\ell_{p}\oplus Z$, $1<p<+\infty$, be a space satisfying the assumptions of Theorem \ref{th2}, and let $T$ be one of the \ops\ on $X$ constructed in the proof. Provided the quantities $\xi_{n}\ll a_{n}\ll b_{n}\ll c_{1,n}\ll c_{2,n}\ll
\ldots\ll c_{k_{n},n}\ll \xi_{n+1}$ are sufficiently large, the \op\ $T$ has a \nt\ \inv\ closed subspace.
\end{proposition}

The proof of Proposition \ref{propbizarre} builds on an argument employed in \cite{GR2} to show that the  Read's type \ops\ on the Hilbert space constructed in \cite{GR} do have 
 \nt\ \inv\ subspaces. It relies on the Lomonosov inequality of \cite{L}, which states the following:
\par\smallskip 
Let $X$ be a complex separable Banach space, and let $\mathcal{A}$ be a weakly closed sub-algebra of $\mathcal{B}(X)$ with $\mathcal{A}\neq\mathcal{B}(X)$. Then there exist two non-zero vectors $x^{**}\in X^{**} $ and $x^{*}\in X^{*}$ such that for every $A\in \mathcal{A}$,
$$
|\pss{x^{**}}{A^{*}x^{*}}|\le||A||_{e},
$$
where $||A||_{e}$ denotes the essential norm of $A$, i.e.\ the distance of $A$ to the space of compact operators on $X$.
\par\smallskip 
In particular, if $T$ is an adjoint \op\ on a complex dual space, and $T$ is a compact perturbation of a power-bounded \op, then $T$ has a \nt\ \inv\ closed subset.
\par\smallskip 
The \ops\ constructed in the proof of Theorem \ref{th2} on $X=\ell_{p}\oplus Z$ for $1<p<+\infty$ are indeed compact perturbations of a power bounded \op\ (this was observed in Remark \ref{remadd}), but, as seen in Section $7.1$ above, they are never adjoint \ops. So one needs to be a bit more subtle in the argument. Moreover, Lomonosov's result \cite{L} as stated above applies only to the complex case. When the space $X$ is real, one has to extract from the proof of \cite[Th. $1$]{L} a statement which applies to both real and complex spaces, and which will be sufficient for the proof of Proposition \ref{propbizarre}.

\begin{proof}[Proof of Proposition \ref{propbizarre}]
 The \op\ $T$ is of the form described in Propostion \ref{prop1}: it can be written as $T=S+K$, where $S=S_0\oplus 0$ acts on $\ell_p\oplus Z$, $S_0$ is a forward weighted  shift on $\ell_p$ and
$Kx=\sum_{j\geq 0}f_{j}^{*}(x)u_{j} $
for some sequence $(u_{j})_{j\geq 0}$ of vectors of $X$ with
$$\sum_{j\geq 0}(1+||f_{j}^{*}||)||u_{j}||<+\infty .$$ It is possible to require that the stronger condition
\begin{equation}\label{EQ5}
\sum_{j\geq 0}(1+||f_{j}^{*}||^{p}){||u_{j}||}^{\min(\frac{1}{p},\frac{1}{q})}<+\infty
\end{equation}
 holds true, where $q$ defined by the equation $\frac{1}{p}+\frac{1}{q}=1$ is the conjugate exponent of $p$. Moreover, we have seen in Remark \ref{remadd} that $u_j=0$ \ifff\ $j$ does not belong to the set $\tilde{J}$. Recall that the set $\tilde{J}$ consists of all points in the (a)-\woi s, plus all right endpoints of either working or \loi s. Thus if $j$ does not belong to $\tilde{J}$, $f_{j}=g_{\sigma  (j)}$ and $f_{j+1}=g_{\sigma  (j+1)}=g_{\sigma (j)+1}$ are consecutive elements of the basis $(g_{j})_{j\ge 0}$ of $\ell_{p}$. Hence $Tf_{j}=w_{j}f_{j+1}$ for every $j\not\in \tilde{J}$, and
 for every $x\in\ell_{p}\oplus Z$ we have
$$Tx=\sum_{j\not\in\tilde{J}} f_{j}^{*}(x)w_{j}f_{j+1}+\sum_{j\in\tilde{J}}f_{j}^{*}(x)u_{j}.$$ 
Consider the \ops\ $A:\ell_p\oplus Z \To \ell_p$ and $B:\ell_p\To \ell_p\oplus Z$ defined by $$ Ax=\sum_{j\not\in\tilde{J}} f_{j}^{*}(x)g_{j}+\sum_{j\in\tilde{J}}
f_{j}^{*} (x)\;
{||u_{j}||}^{\frac{1}{p}}g_j \quad \textrm{ for } x \in
\ell_p \oplus Z $$ and $$By=\sum_{j\not\in\tilde{J}} g_{j}^{*}(y) w_{j}f_{j+1}+\sum_{j\in\tilde{J}} g_{j}^{*}(y)\; {{||u_{j}||}^{-\frac{1}{p}}} u_j
\quad \textrm{ for } y\in
\ell_p$$ where  $(g_j^{*})_{j\geq 0}$ is the sequence of canonical functionals \wrt\ the canonical basis $(g_{j})_{j\ge 0}$ of $\ell_{p}$. Observe that since $u_j$ is \nz\ for every $j\in\tilde{J}$, it makes sense to write ${{||u_{j}||}^{-\frac{1}{p}}}$ in the definition of $B$.
The key lemma is now the following:

\begin{lemma}\label{factbizarre}
The \ops\ $A:\ell_p\oplus Z \To \ell_p$ and $B:\ell_p\To \ell_p\oplus Z$ factor $T$  through the space $\ell_{p}$: we have $T=BA$. Moreover, $A$ has dense range and the kernel of $B$ is contained in the closed linear subspace of $\ell_{p}$ spanned by the vectors $g_{0}$ and $g_{a_{n}-1} $, $g_{a_{n}}$ for $n\ge 1$.
\end{lemma}

\begin{proof}[Proof of Lemma \ref{factbizarre}]
We have $$||Ax||=\left(\sum_{j\not\in\tilde{J}}|f_{j}^{*}(x)|^{p}+\sum_{j\in\tilde{J}}|f_{j}^{*}(x)|^p {||u_{j}||}
\right)^{\frac{1}{p}}
\leq ||x||\,\left(1+\sum_{j\in\tilde{J}}||f_{j}^{*}||^p {||u_{j}||}\right)^{\frac{1}{p}}$$
since  for $j\not \in \tilde{J}$, $f_{j}^{*}(x)$ is equal to  the ${\sigma  (j)}$-th coordinate in the basis $(g_{j})_{j\ge 0}$ of the projection of $x$ onto $\ell_{p}$.
Thus by (\ref{EQ5}), $A$ is a bounded \op\ on $\ell_p\oplus Z$. Observe already that the range of $A$ is dense in $\ell_p$. Then, using the triangle inequality plus the H\"older inequality, we get that 
$$||By||\leq
||\sum_{j\not\in\tilde{J}} g_{j}^{*}(y) w_{j}f_{j+1}||+\left(\sum_{j\in\tilde{J}}|g_{j}^{*}(y)|^p\right)^{\frac{1}{p}}
\,\left(\sum_{j\not\in\tilde{J}}{||u_{j}||^{q(1-\frac{1}{p})}}\right)^{\frac{1}{q}}.$$
Recall now that for $j\not\in\tilde{J}$, $f_{j+1}=g_{\sigma  (j+1)}$. Hence
$$||\sum_{j\not\in\tilde{J}} g_{j}^{*}(y) w_{j}f_{j+1}||=
\left(\sum_{j\not\in\tilde{J}}|g_{j}^{*}(y)|^{p}w_{j}^{p}\right)^{\frac{1}{p}}.$$
Thus, using the fact that $0<w_{j}\leq 2$, we obtain that for all $y\in \ell_{p}$
$$||By||\leq  ||y||\,\left(2+\left(\sum_{j\in\tilde{J}} {||u_{j}||}\right)^{\frac{1}{q}}\right).$$
Hence $B$ is also a bounded \op\ on $\ell_p$ by (\ref{EQ5}).
We have for any $x\in \ell_p\oplus Z$
$$BAx=
\sum_{j\not\in\tilde{J}} f_{j}^{*}(x)w_{j}f_{j+1}+\sum_{j\in\tilde{J}}f_{j}^{*}(x){||u_j||}^{\frac{1}{p}}{||u_j||}^{-\frac{1}{p}} u_j=Tx,$$ 
and thus we have shown that $T$ can be factorized through the reflexive space $\ell_p$, thanks to the \ops\ $A$ and $B$. It remains to prove the statement about the kernel of $B$, and this is the most technical part of the proof. 
\par\smallskip
Suppose that $y\in\ell_{p}$ is such that $||y||=1$ and
\begin{eqnarray}\label{zero}
 By=\sum_{j\not\in\tilde{J}} g_{j}^{*}(y) w_{j}f_{j+1}+\sum_{j\in\tilde{J}} g_{j}^{*}(y)\; {{||u_{j}||}^{-\frac{1}{p}}} u_j=0. 
\end{eqnarray}
Let $I=[1,+\infty[\setminus\bigcup_{n\ge 0}\{a_{n}-1, a_{n}\}$. 
Our first goal is to show that if $j$ belongs to $I$, then $g_{j}^{*}(y) =0$. Recall that for every $j\in\tilde{J}$, $u_{j}$ is defined as $u_{j}=Tf_{j}$, and that
$\tilde{J}=J_{a}\cup J_{b}^{-}\cup J_{b}^{+}\cup J_{c}^{-}\cup J_{c}^{+}$, where
$$J_{a}=\bigcup_{n\ge 1}[a_{n}-(\kappa _{n}-\kappa _{n-1}-1)-1, a_{n}],$$
$$J_{b}^{-}=\bigcup_{n\ge 1}\{r(b_{n}+1)-1 \textrm{ ; } r\in [1, a_{n}]\},\quad J_{b}^{+}=\bigcup_{n\ge 1}\{ rb_{n}+a_{n} \textrm{ ; } r\in [1, a_{n}]\}$$
and 
$$J_{c}^{-}=\bigcup_{n\ge 1}\{s_{1}c_{1,n}+\ldots +s_{k_{n}}c_{k_{n},n}-1 \textrm{ ; } s_{1},\ldots, s_{k_{n}}\in [0,h_{n}],\; |s|\ge 1\},$$
$$ J_{c}^{+}=\bigcup_{n\ge 1}\{ s_{1}c_{1,n}+\ldots +s_{k_{n}}c_{k_{n},n}+\nu _{n}\textrm{ ; } s_{1},\ldots, s_{k_{n}}\in [0,h_{n}],\; |s|\ge 1\}.$$
For each $n\ge 1$ denote by $\mathcal{S}_{n}(y)$ the sum
\begin{eqnarray}\label{38}
\mathcal{S}_{n}(y)=\sum_{j\not\in\tilde{J},\, j\le\xi  _{n+1}} g_{j}^{*}(y) w_{j}f_{j+1}+\sum_{j\in\tilde{J},\, j\le\xi  _{n+1}} g_{j}^{*}(y)\; {{||u_{j}||}^{-\frac{1}{p}}} u_j
\end{eqnarray}
and by $\mathcal{R}_{n}(y)$ the remainder term
\begin{eqnarray}\label{39}
\mathcal{R}_{n}(y)=\sum_{j\not\in\tilde{J},\, j>\xi  _{n+1}} g_{j}^{*}(y) w_{j}f_{j+1}+\sum_{j\in\tilde{J},\, j>\xi  _{n+1}} g_{j}^{*}(y)\; {{||u_{j}||}^{-\frac{1}{p}}} u_j.
\end{eqnarray}
Let also $I_{n}$ be the set $[0, \xi  _{n+1}]\cap I$. Let us first look at the quantities $\pi_{[0, \xi  _{n+1}]}(\mathcal{R}_{n}(y))$ and $\pi_{I_{n}}(\mathcal{R}_{n}(y))$, where $\pi_{I_{n}}$ denotes the canonical projection on the linear span of the vectors $f_{j}$, $j\le \xi  _{n+1}$ and $j\in I_{n}$: there are two ways in which a vector $f_{j}$, $j\le\xi  _{n+1}$, can appear in the projection of $\mathcal{R}_{n}(y)$ onto the first $\xi  _{n+1}+1$ coordinates:
\par\smallskip
\textbf{Case a:} if $j=a_{m}-1$ for some $m\ge n+1$, then:

-- if $\kappa _{m}=\kappa _{m-1}+1$, we have $u_{a_{m}-1}=Tf_{a_{m}-1}=\lambda _{a_{m}-1}e_{a_{m}}$ and
$$\pi_{[0, \xi  _{n+1}]}(u_{a_{m}-1})=\lambda _{a_{m}-1}(f_{0}+\sum_{k=1}^{n}\alpha _{k}f_{a_{k}}).$$

-- if $\kappa _{m}>\kappa _{m-1}+1$, we have $u_{a_{m}-1}=Tf_{a_{m}-1}=\frac{1}{a_{m}^{2}}e_{a_{m}}$ and
$$\pi_{[0, \xi  _{n+1}]}(u_{a_{m}-1})=\frac{1}{a_{m}^{2}}(f_{0}+\sum_{k=1}^{n}\alpha _{k}f_{a_{k}}).$$

\par\smallskip
In these two situations $\pi_{I_{n}}(u_{a_{m}-1})$ belongs to the linear span of the vectors $f_{a_{0}},  \ldots, f_{a_{n}}$, and so $\pi_{I_{n}}(u_{a_{m}-1})=0$ for each $m\ge n+1$.
\par\smallskip
\textbf{Case b:} if $j\in J_{c}^{-}\cap [\xi_{m}+1,\xi_{m+1}]$ for some $m\ge n+1$, then
$j=s_{1}c_{1,m}+\ldots +s_{t}c_{t,m}-1$
where $t\in [1,k_{m}]$ is the largest integer such that $s_{t}\ge 1$, and $s_{1},\ldots, s_{t}$ belong to $[0, h_{m}]$. In this case
$$u_{j}=\lambda_{j}\frac{\gamma _{m}}{4^{1-|s|}}f_{j+1}+p_{t,m}(T)e_{j-c_{t,m}}.$$
One sees in this expression that the term $\lambda _{j}\,p_{t,m}(T)e_{j-c_{t,m}}$ can contribute something to both $\pi_{[0, \xi  _{n+1}]}(\mathcal{R}_{n}(y))$ and $\pi_{I_{n}}(\mathcal{R}_{n}(y))$.
\par\smallskip
Inspection of the expressions of $u_{j}$ for $j> \xi  _{n+1}$ in all the other situations show that these are the only cases where $\pi_{[0, \xi  _{n+1}]}(u_{j})$ can be \nz. So
\begin{eqnarray}\label{40}
 \pi_{[0, \xi  _{n+1}]}(\mathcal{R}_{n}(y))&=&\sum_{m\ge n+1}g_{a_{m}-1}^{*}(y)\,||u_{a_{m}-1}||^{-\frac{1}{p}}\pi_{[0,\xi  _{n+1}]}(u_{a_{m}-1})\\
 &+&\sum_{j\in J_{c}^{-},\, j>\xi_{n+1}}g_{j}^{*}(y)\,||u_{j}||^{-\frac{1}{p}}\pi_{[0,\xi  _{n+1}]}(u_{j})\notag
\end{eqnarray}
and
\begin{eqnarray}\label{41}
 \pi_{I_{n}}(\mathcal{R}_{n}(y))&=&
 \sum_{j\in J_{c}^{-},\, j>\xi_{n+1}}g_{j}^{*}(y)\,||u_{j}||^{-\frac{1}{p}}\pi_{I_{n}}(u_{j}).
\end{eqnarray}
If at each step $m$ of the construction of the \op\ $T$ the quantities $c_{1,m}\ll \ldots\ll c_{k_{m},m}\ll \xi  _{m+1}$ are chosen sufficiently large \wrt\ each other, we can ensure that for every $j\in J_{c}^{-}\cap[\xi_{m}+1,\xi_{m+1}]$, $||u_{j}|| <\tau _{m},$ where $\tau_{m}$ can be chosen as small as we wish. So
$$||\pi_{I_{n}}(\mathcal{R}_{n}(y))||\le ||\pi_{I_{n}}||\sum_{m\ge n+1}\sum_{j\in J_{c}^{-}\cap[\xi_{m}+1,\xi_{m+1}]}||u_{j}||^{\frac{1}{q}}\le  ||\pi_{I_{n}}||\sum_{m\ge n+1}h_{m}k_{m}\tau_{m}^{\frac{1}{q}}.$$
The outcome of this is that if, after step $n$ of the construction is completed, we fix a number $\rho  _{n+1}$ as small as we wish, we can carry out the construction after step $n$ in such a way that $
 ||\pi_{I_{n}}(\mathcal{R}_{n}(y))||<\rho  _{n+1}$.
So we have for each $n\ge 1$
\begin{eqnarray}\label{43}
 0=\pi_{I_{n}}(\mathcal{S}_{n}(y))+\pi_{I_{n}}(\mathcal{R}_{n}(y))\quad \textrm{with}
\quad ||\pi_{I_{n}}(\mathcal{R}_{n}(y))||<\rho  _{n+1}.
\end{eqnarray}

We now aim to prove the following

\begin{claim}\label{claim}
 For every $j\in I_{n}$ there exists a positive constant $C_{j,n}$,
 depending only on the construction until step $n$, such that
 \begin{eqnarray}\label{44}
|g_{j}^{*}(y)|<C_{j,n}\,\rho  _{n+1}.
\end{eqnarray}
\end{claim}

\begin{proof}[Proof of Claim \ref{claim}]
In order to prove this claim, we need to understand how a vector $f_{j}$
for $j\in\tilde{J}$ and $j\in [\xi  _{m}+1,\xi  _{m+1}]$ for some $m\in [0,n]$, may appear in the expression of $\pi_{I_{n}}(\mathcal{S}_{n}(y))+\pi_{I_{n}}(\mathcal{R}_{n}(y))$. We adress this question for $m=n$ first, and consider separately different cases.
\par\smallskip
\textbf{Case 1:} if $j\in J_{c}^{+} $ with $\xi  _{n}+1\le j\le \xi  _{n+1}$, then $j=s_{1}c_{1,n}+\ldots+s_{k_{n}}c_{k_{n},n}+\nu _{n}$ with $|s|\ge 1$ and $s_{1},\ldots, s_{k_{n}}\in [0,h_{n}]$. Consider first integers of the form $j=s_{1}c_{1,n}+\ldots+h_{n}c_{k_{n},n}+\nu _{n}$. Then
\begin{eqnarray}\label{enplus}
 u_{j}=\frac{\gamma _{n}}{4^{1-|s|}}f_{j+1}+p_{k_{n},n}(T)e_{j-c_{k_{n},n}+1}.
\end{eqnarray}
The only place where the vector $f_{j+1} $ appears in the expression (\ref{38}) of $\mathcal{S}_{n}(y)$ is in the formula above for $u_{j}$. Hence
$$g_{j}^{*}(y)\frac{\gamma _{n}}{4^{1-|s|}}||u_{j}||^{-\frac{1}{p}}=-f_{j+1}^{*}(\pi_{I_{n}}(\mathcal{R}_{n}(y))$$ so that
$$|g_{j}^{*}(y)|\le \frac{4^{1-|s|}}{\gamma _{n}}||u_{j}||^{\frac{1}{p}}\rho  _{n+1}.$$
Setting $C_{j,n}=\frac{4^{1-|s|}}{\gamma _{n}}||u_{j}||^{\frac{1}{p}}$, we observe that $C_{j,n} $ depends only on the construction until step $n$, and so (\ref{44}) is proved for this $j$.

If we go down one step and consider integers 
$j=s_{1}c_{1,n}+\ldots+(h_{n}-1)c_{k_{n},n}+\nu _{n}$, then we see that the vector $f_{j+1}$ may appear in the expression of $\mathcal{S}_{n}(y)$ in two ways: first, as coming from $u_{j}$ just as above, and, second, as coming from the expression (\ref{enplus}) of $u_{j+c_{k_{n},n}}$. Indeed, $p_{k_{n},n}(T)e_{j+1}$ may have a \nz\ component on the vector $f_{j+1}$. So $f_{j+1}^{*}(\mathcal{S}_{n}(y))$ is equal to 
$$g_{j}^{*}(y)\frac{\gamma _{n}}{4^{1-|s|}}||u_{j}||^{-\frac{1}{p}}+
g_{j+c_{k_{n},n}}^{*}(y)\frac{\gamma _{n}}{4^{1-|s|}}||u_{j+c_{k_{n},n}}||^{-\frac{1}{p}}f_{j+1}^{*}(p_{k_{n},n}(T)e_{j+1})$$ and also to
$$-f_{j+1}^{*}(\pi_{I_{n}}(\mathcal{R}_{n}(y)).$$ Since we have proved already that $|g_{j+c_{k_{n},n}}^{*}(y)|<C_{j+c_{k_{n},n},n}\rho  _{n+1}$, we eventually obtain that $|g_{j}^{*}(y)|<C_{j,n}\rho  _{n+1}$.

We now can go down the lattice of (c)-\woi s in the same way, and obtain the estimate (\ref{44}) for every integer $j=s_{1}c_{1,n}+\ldots+s_{k_{n}}c_{k_{n},n}+\nu _{n}$ with $s_{k_{n}}\ge 1$, then for every $j=s_{1}c_{1,n}+\ldots+s_{k_{n}-1}c_{k_{n}-1,n}+\nu _{n}$ with $s_{k_{n}-1}\ge 1$, etc. until we get it for every $j\in J_{c}^{+}\cap [\xi  _{n}+1,\xi  _{n+1}]$. The proof is exactly the same for $j\in J_{c}^{-}\cap [\xi  _{n}+1,\xi  _{n+1}]$.

\par\smallskip

\textbf{Case 2:} The case of (b)-intervals is dealt with in exactly the same fashion, considering first the index $j=a_{n}b_{n}+a_{n}$. There are several constributions to the vector $f_{j+1}$: one coming from $u_{j}$, some others coming from the (c)-intervals above, and a last one coming from the quantity $\pi_{I_{n}}(\mathcal{R}_{n}(y))$. The terms coming from the (c)-intervals are controlled thanks to the study of Case 1 above, and the norm of the last contribution is estimated as usual by $\rho  _{n+1}$. So we get that $|g_{j}^{*}(y)|<C_{j,n}\rho  _{n+1}$ for $j=a_{n}b_{n}+a_{n}$. Then we go down the (b)-intervals to obtain a similar inequality for every $j$ of the form $j=ra_{n}+b_{n}$, $r\in [1,a_{n}]$. Then we treat (in decreasing order again) the indices $j=rb_{n}-1$ for $r\in [1,a_{n}]$, and this proves (\ref{44}) for every $j\in J_{b}\cap [\xi  _{n}+1,\xi_{n+1}]$.

\par\smallskip

\textbf{Case 3:} It remains to deal with indices $j\in J_{a}\cap [\xi  _{n}+1,\xi_{n+1}]$. Consider first indices $j\in [a_{n}-(\kappa _{n}-\kappa _{n-1}-1)-1, a_{n}-1]$.

-- suppose that $\kappa _{n}>\kappa _{n-1}+1$. If $j=a_{n}-k$, $k\in [2,\kappa _{n}-\kappa _{n-1}]$, then the same argument as above shows that 
$|g_{j}^{*}(y)|<C_{j,n}\rho  _{n+1}$. If $j=a_{n}-1$, then
$u_{a_{n}-1}=Tf_{a_{n}-1}=\frac{1}{a_{n}^{2}}e_{a_{n}}$
belongs to the linear span of the vectors $f_{a_{0}}, f_{a_{1}},\ldots, f_{a_{n}}$.

-- suppose that $\kappa _{n}=\kappa _{n-1}+1$. Then $j=a_{n}-1$, and 
$u_{a_{n}-1}=\lambda _{a_{n}-1}e_{a_{n}}$ belongs to the linear span of the vectors $f_{a_{0}}, f_{a_{1}},\ldots, f_{a_{n}}$.

Lastly if $j=a_{n}$, then in both cases we have
$$u_{a_{n}}=Tf_{a_{n}}=\frac{1}{\alpha _{n}}\left(\frac{1}{\lambda _{a_{n}+1}}f_{a_{n}+1}-\frac{1}{\lambda _{a_{n-1}+1}}f_{a_{n-1}+1}\right)$$
so that $u_{a_{n}}$ belongs to the linear span of the vectors $f_{a_{n-1}+1}$ and $f_{a_{n}+1}$.
\par\smallskip
Putting together all these cases, we see that we have proved that if $j$ belongs to $\tilde{J}$ and $\xi  _{n}+1\le j\le \xi  _{n+1}$, then 
 except in the cases where $j=a_{n}-1$ and $j=a_{n}$ there exists a constant $C_{j,n}$ depending only on the construction until step $n$ such that $|g_{j}^{*}(y)|< C_{j,n}\rho  _{n+1}$. In the  cases where $j=a_{n}-1$ or $j=a_{n}$, $u_{j}$ belongs to the linear span of the vectors $f_{a_{0}}, f_{a_{1}},\ldots,  f_{a_{n-1}}, f_{a_{n-1}+1}, f_{a_{n}}, f_{a_{n}+1}$.
\par\smallskip
We now apply exactly the same procedure at steps $n-1$, then at step $n-2$, etc. in order to obtain that  for each $0\le m\le n-1$,
$$|g_{j}^{*}(y)|<C_{j,m}\,\rho  _{n+1}\quad \textrm{for every } j\in \tilde{J}\cap[\xi  _{m}+1, \xi  _{m+1}]\setminus\{a_{m}-1,a_{m}\},$$
and that $u_{j}$ belongs to the linear span of the vectors $f_{a_{0}}, f_{a_{1}}, f_{a_{1}+1},\ldots,  f_{a_{n}}, f_{a_{n}+1}$ for every $j\in\bigcup_{0\le m\le n}\{a_{m}-1,a_{m}\}$.
This proves the Claim.
\end{proof}

\par\smallskip
Setting $C_{n}$ to be the supremum of all these quantities $C_{j,m}$ for $0\le m\le n$, we have eventually obtained that 
\begin{eqnarray}\label{44bis}
|g_{j}^{*}(y)|<C_{n}\,\rho  _{n+1}\quad \textrm{for every } j\in I_{n}.
\end{eqnarray}
\par\smallskip
 Let us now consider a fixed integer $j\in I$. For each $n$ sufficiently large (\ref{44bis}) holds true,
where the constant $C_{n}$ depends only on the construction until step $n$ while $\rho  _{n+1}$ is chosen after the construction at step $n$ is completed. So we can for each $n\ge 1$ choose $\rho  _{n+1} $ so small that $C_{n}\,\rho  _{n+1}<2^{-n}$. Making $n$ tend to infinity in (\ref{44bis}) yields that 
$$g_{j}^{*}(y)=0\quad  \textrm{for every } j\in\tilde{J}\setminus\bigcup_{n\ge 0}\{a_{n}-1,a_{n}\}.$$
Equation (\ref{zero}) can hence be rewritten as 
\begin{eqnarray*}
 \sum_{j\not\in\tilde{J}} g_{j}^{*}(y) w_{j}f_{j+1}&+&\sum_{n\ge 1}\left( g_{a_{n}-1}^{*}(y)\; {{||u_{a_{n}-1}||}^{-\frac{1}{p}}} u_{a_{n}-1}+g_{a_{n}}^{*}(y)\; {{||u_{a_{n}}||}^{-\frac{1}{p}}} u_{a_{n}}\right)\\
&+&g_{0}^{*}(y)||u_{0}||^{-\frac{1}{p}}u_{0}=0.
\end{eqnarray*}
The second sum as well as the third term in this expression belong to the closed linear span of the vectors $f_{a_{n}}$ and $f_{a_{n}+1}$ for $n\ge 0$. Also, if $j$ does not belong to $\tilde{J}$, $j+1$ cannot be equal to $a_{n}$ or $a_{n}+1$ for some $n\ge 1$, nor to $0$. It follows that $g_{j}^{*}(y)=0$ for every $j\not\in \tilde{J}$. Putting things together, we see that we have proved that $g_{j}^{*}(y)=0$ for every $j\ge 0$ not belonging to the set $\bigcup_{n\ge 0}\{a_{n}-1, a_{n}\}$. This finishes the proof of Lemma \ref{factbizarre}.
\end{proof}

We now go back to the proof of Proposition \ref{propbizarre}.
Consider, as in the proof of Theorem \ref{th4}, the \op\ $T_{0}=AB$ acting on $\ell_{p}$. For every $y\in \ell_{p}$ we have
$$T_{0}y=\sum_{j\not \in\tilde{J}} g_{j}^{*}(y) w_{j}Af_{j+1}+\sum_{j\in\tilde{J}}
g_{j}^{*} (y)
{{||u_{j}||}}^{-\frac{1}{p}}Au_{j}.$$ If both $j$ and $j+1$ do not belong to $\tilde{J}$, then $Af_{j+1}=g_{j+1}$. If $j$ does not belong to $\tilde{J}$ but $j+1$ does, then $Af_{j+1}={||u_{j+1}||}^{\frac{1}{p}}g_{j+1}$. It follows that
$$T_{0}y=\sum_{\{j \, ;\, j,j+1\not\in\tilde{J}\}} g_{j}^{*}(y) w_{j}g_{j+1}+
\sum_{\{j \, ;\, j\not\in\tilde{J},\, j+1\in\tilde{J}\}}g_{j}^{*}(y) w_{j}{||u_{j+1}||}^{\frac{1}{p}}g_{j+1}+
\sum_{j\in\tilde{J}} g_{j}^{*} (x)
{{||u_{j}||}}^{-\frac{1}{p}}Au_{j} $$
i.e. that $T_{0}=S_{1}+K_{1}$ where $S_{1}$ and $K_{1}$ are \ops\ on 
$\ell_{p}$ defined as
$$S_{1}y=\sum_{\{j\, ;\, j,j+1\not\in\tilde{J}\}} g_{j}^{*}(y) w_{j}g_{j+1}$$
and
$$K_{1}y=\sum_{\{j \, ;\, j\not\in\tilde{J},\,j+1\in\tilde{J}\}}g_{j}^{*}(y) w_{j}{||u_{j+1}||}^{\frac{1}{p}}g_{j+1}+
\sum_{j\in\tilde{J}} g_{j}^{*} (x)
{{||u_{j}||}}^{-\frac{1}{p}}Au_{j}.$$
 Then the \op\ $K_{1}$ is easily seen to be nuclear, as
$$\sum_{\{j \, ;\, j\not\in\tilde{J},\,j+1\in\tilde{J}\}}w_{j}{||u_{j+1}||}^{\frac{1}{p}}+
\sum_{j\in\tilde{J}}{{||u_{j}||}}^{-\frac{1}{p}}||Au_{j}||\le\max(2,||A||)\,\sum_{j\in\tilde{J}}{||u_{j}||}^{\min(\frac{1}{p},\frac{1}{q})}<+\infty$$ by (\ref{EQ5}).
The \op\ $S_{1}$ is a forward weighted shift on 
$\ell_{p}$, which we can further decompose as in Remark \ref{remadd} as the sum of a contraction and a compact \op: setting
$$S_{2}y=\sum_{\{j\, ;\, j,j+1\not\in\tilde{J}\}} g_{j}^{*}(y) g_{j+1}$$
and
$$K_{2}y=\sum_{\{j\, ;\, j,j+1\not\in\tilde{J}\}} (w_{j}-1) g_{j}^{*}(y) g_{j+1}+K_{1}y$$ we have $T_{0}=S_{2}+K_{2}$, where $S_{2}$ is a contraction on $\ell_{p}$ and $K_{2}$ is a compact \op.
\par\smallskip
Denote by $E_{0}$ the kernel of $B$, and by $E$ the quotient space $\ell_{p}/E_{0}$. Let $\overline{A}:\ell_{p}\oplus Z\rightarrow E$
 and $\overline{B}:\ell_{p}\rightarrow E\oplus Z$ be the \ops\ defined by $\overline{A}x=Ax+E_{0}$ for every $x\in \ell_{p}\oplus Z$, and $\overline{B}(y+E_{0})=By$ for every $y\in \ell_{p}$. Obviously $T=\overline{B}\,\overline{A}$, and $\overline{T}_{0}=\overline{A}\,\overline{B}$ is a bounded \op\ on $E$. Lemma \ref{factbizarre} now implies that $E_{0}$ is invariant by the \op\ $S_{2}$. Indeed, if $y$ belongs to $E_{0}$, then $g_{j}^{*}(y)=0$ for every $j$ not belonging to the set $\bigcup_{n\ge 0}\{a_{n}-1, a_{n}\}$. So in particular 
 $g_{j}^{*}(y)=0$ for every $j$ such that $j$ and $j+1$ do not belong to $\tilde{J}$. Thus $S_{2}y=0$. In particular $S_{2}(E_{0})\subseteq E_{0}$, and it makes sense to define $\overline{S}_{2}:E\rightarrow E$ by setting $\overline{S}_{2}(y+E_{0})=S_{2}y+E_{0}$ for $y\in \ell_{p}$. Hence $\overline{T}_{0}$ can be decomposed as $\overline{T}_{0}=\overline{S}_{2}+\overline{K}_{2}$, where $\overline{S}_{2}$ is a contraction of $E$ and $\overline{K}_{2}$ is a compact \op\ on $E$. So $||\overline{T}_{0}^{n}||\le 1$ for every $n\ge 1$.
\par\smallskip
If the space $X$ is complex, $E$ is a complex reflexive Banach space
and we can apply the Lomonosov inequality to the weakly closed sub-algebra of $\mathcal{B}(E^{*})$ generated by $\overline{T}_{0}^{*}$:
there exist two \nz\ vectors $e\in E$ and $e^{*}\in E^{*}$ such that
$|\pss{e^{*}}{\overline{T}_{0}^{n}e}|\le||\overline{T}_{0}^{*n}||\le 1$
for every $n\ge 1$. So the closure of the orbit of the vector $e\in E$ under the action of $\overline{T}_{0}$ is
 a \nt\ \inv\ closed subset of $\overline{T}_{0}$. Since $\overline{A}$ has dense range and $\overline{B}$ is injective, 
the same reasoning as in the proof of Theorem \ref{th4} implies then that $T$ itself has a \nt\ \inv\ closed subset, which is an \inv\ subspace by Property (P1). Proposition \ref{propbizarre} is hence proved in the complex case.
\par\smallskip
If $X$ is a real Banach space, we need to go back to the proof of Theorem $1$ in \cite{L} in order to show that $\overline{T}_{0}$ has a \nt\ \inv\ closed subset in $E$. It is not difficult to see that Lemma $1 - 6$ as well as the main part of Lemma $8$ of \cite{L} hold true in the real case as well, so that we have the following result:

\begin{theorem}\label{thlom}
 Let $X$ be a real or complex separable Banach space, and let $\mathcal{A}$ be a uniformly closed subalgebra of $\mathcal{B}(X)$.
 Then we have the following alternative: 
 either
 
(A1)  there exist two \nz\ vectors $x^{*}\in X^{*}$ and
 $x^{**}\in X^{**}$ such that 
\begin{eqnarray*}
|\pss{x^{**}}{A^{*}x^{*}}|\le ||A||_{e}\quad \textrm{ for every } A\in \mathcal{A}
\end{eqnarray*}
or

(A2)  the set $\{A^{*}x^{*} \textrm{ ; } A\in \mathcal{A}\}$ is dense in $X^{*}$ for every \nz\ vector $x^{*}\in X^{*}$, and in this case there exist an \op\ $A_{0}\in\mathcal{A}$, different from the identity \op, such that $1$ is an eigenvalue of $A_{0}^{*}$.
\end{theorem}

Let $\mathcal{A}$ be the uniformly closed sub-algebra of $\mathcal{B}(E^{*})$ generated by the \op\ $\overline{T}_{0}^{*}$. If (A1) of Theorem \ref{thlom} above holds true, then we know that 
$\overline{T}_{0}$ has a \nt\ \inv\ closed subset in $E$, and we are done. So suppose that (A2) is true: since $E$ is reflexive, there exist an \op\ $A_{0}\in \mathcal{B}(E^{*})$  and a \nz\ vector $x_{0}\in E$ such that $A_{0}$ is not the identity \op\ and $A_{0}^{*}x_{0}=x_{0}$. Since $\mathcal{A}$ is commutative, we have $A^{*}x_{0}=A^{*}A_{0}^{*}x_{0}=A_{0}^{*}A^{*}x_{0}$ for every $A\in \mathcal{A}$. Assumption (A2) implies that the set $\{A^{*}x_{0} \textrm{ ; } A\in \mathcal{A}\}$ is dense in $E$, so $A_{0}=I$ which is impossible. So (A2) cannot hold, and (A1) is true. This proves Proposition \ref{propbizarre} in the real case.
\end{proof}

\subsection{Invariant subspaces of the operators obtained in Theorem \ref{th5}}

We finish this section by showing that the \op\ constructed in the proof of Theorem \ref{th5} has \nt\ \inv\ subspaces too:

\begin{proposition}
 Let $T$ be the \op\ defined on a closed subspace $L$ of $\ell_{2}\oplus J^{*}$ constructed in the proof of Theorem \ref{th5}. Then $T$  has a \nt\ \inv\ subspace.
\end{proposition}

\begin{proof}
The \op\ $T$ is induced by an \op\ $T_{1}$ acting on  $\ell_{2}\oplus J^{*}$ on one of its \inv\ subspaces $L$, where $T_{1}$ is one of the \ops\ on $\ell_{2}\oplus J^{*}$ given by Theorem \ref{th2} for some $\varepsilon \in (0,1)$. The proof of Proposition \ref{propbizarre} above, plus the observation that the space $E=\ell_{2}/E_{0}$ is isometric to $\ell_{2}$ in this case, show that $T_{1}$ can be factorized through the space $\ell_{2}$ in the following way: there exist $\overline{A}_{1}:\ell_{2}\oplus J^{*}\rightarrow \ell_{2}$ with dense range and $\overline{B}_{1}:\ell_{2}\rightarrow \ell_{2}\oplus J^{*}$ injective such that $T_{1}=\overline{B}_{1}\overline{A}_{1}$. Let us now define $H$ to be the closure in $\ell_{2}$ of the linear space $\overline{A}_{1}(L)$. Then $H$ is a Hilbert space, and since $L$ is $T$-\inv\ the \op\ $\overline{B}_{1}$ maps $H$ into $L$. It follows that the \op\ $\overline{T}_{0}=\overline{A}_{1}\overline{B}_{1}$ defined on $\ell_{2}$ leaves $H$ \inv. Moreover, we have seen in the proof of Proposition \ref{propbizarre} that $\overline{T}_{0}$ could be written as $\overline{T}_{0}=\overline{S}_{2}+\overline{K}_{2}$, where $||\overline{S}_{2}||\le 1$ and $\overline{K}_{2}$ is a compact \op\ on $\ell_{2}$. Let $\overline{T}_{2}$ be the \op\ induced by $\overline{T}_{0}$ on $H$. If $i_{H}$ denotes the canonical injection of $H$ into $\ell_{2}$ and $P_{H}$ denotes the orthogonal projection of $\ell_{2}$ onto $H$, then $\overline{T}_{2}=P_{H}\overline{S}_{2}i_{H}+P_{H}\overline{K}_{2}i_{H}$, and $P_{H}\overline{S}_{2}i_{H}$ is again a contraction on $H$ while $P_{H}\overline{K}_{2}i_{H}$ is compact.
Proceeding in the same way as in the proof of Proposition \ref{propbizarre}, we see that 
 $\overline{T}_{2}$ has a \nt\ \inv\ closed subset. Hence we obtain that $T$ itself acting on $L$ has a \nt\ \inv\ closed subset. This set is a subspace by Property (P1), and we are done.
\end{proof}

\end{document}